\let\JournalBibliographyStyle\bibliographystyle
\renewcommand{\bibliographystyle}[1]{\JournalBibliographyStyle{sn-mathphys-num-local}}
\documentclass[pdflatex,sn-mathphys-num]{sn-jnl}
\usepackage[T1,OT1]{fontenc}
\DeclareTextCommand{\l}{OT1}{{\fontencoding{T1}\fontfamily{lmr}\selectfont\l}}
\usepackage{amsmath,amssymb,mathtools,bm,microtype}
\usepackage{graphicx,xcolor,booktabs,longtable,array,enumitem,needspace}
\numberwithin{equation}{section}
\theoremstyle{thmstyleone}
\newtheorem{theorem}{Theorem}[section]
\newtheorem{proposition}[theorem]{Proposition}
\newtheorem{lemma}[theorem]{Lemma}
\newtheorem{corollary}[theorem]{Corollary}
\theoremstyle{thmstylethree}

\theoremstyle{thmstyletwo}

\newcommand{\R}{\mathbb R}
\newcommand{\Z}{\mathbb Z}

\newcommand{\E}{\mathbb E}
\newcommand{\Prob}{\mathbb P}

\newcommand{\dd}{\,\mathrm d}
\newcommand{\one}{\mathbf 1}

\newcommand{\cL}{\mathcal L}
\newcommand{\cA}{\mathcal A}
\newcommand{\id}{\mathrm{Id}}
\newcommand{\TV}{\mathrm{TV}}
\newcommand{\Law}{\operatorname{Law}}
\newcommand{\capB}{\operatorname{cap}_{B}}
\newcommand{\HH}{\mathcal H}

\DeclareMathOperator{\dist}{dist}

\makeatletter
\def\ps@plain{\let\@mkboth\@gobbletwo
  \let\@oddhead\@empty\let\@evenhead\@empty
  \def\@oddfoot{\hfil\thepage\hfil}\let\@evenfoot\@oddfoot}
\makeatother
\hypersetup{hidelinks,pdftitle={Branching explosion of Brownian Fleming--Viot systems},pdfauthor={Yaohui Lin and Tiankui Lu},pdfsubject={Probability theory},pdfkeywords={Fleming--Viot, Brownian motion, branching explosion, copying current, Kelvin inversion}}
\begin{document}
\raggedbottom
\title[Brownian Fleming--Viot explosion]{Branching explosion of Brownian Fleming--Viot systems}
\author*[1]{\fnm{Yaohui} \sur{Lin}}\email{linyuan@m.scnu.edu.cn}
\author[2]{\fnm{Tiankui} \sur{Lu}}
\affil*[1]{\orgdiv{School of Mathematical Sciences}, \orgname{South China Normal University}, \orgaddress{\city{Guangzhou}, \country{China}}}
\affil[2]{\orgdiv{Information and Artificial Intelligence}, \orgname{Anhui Agricultural University}, \orgaddress{\city{Hefei}, \country{China}}}
\abstract{We resolve the long-standing non-extinction problem for Brownian Fleming--Viot systems in bounded domains. For every fixed $N\ge3$, we construct a bounded connected domain in $\R^3$ in which branching explosion occurs with positive probability before every prescribed positive time, from every interior initial configuration. On this event all particles converge to one boundary point. Together with two-particle non-extinction, this gives the sharp population threshold for bounded-domain Brownian explosion. The mechanism is directed transport generated by uniform copying in a periodic absorbing environment. We compute a positive cubic current for the same two-harmonic killing profile at every $N\ge3$, with an explicit population coefficient and a controlled fourth-order remainder. To realize this transport by spherical absorption, we develop a marked genealogical transfer theorem that retains the Euclidean displacement inherited at each copy. Common ancestry makes the integer winding reconstructible from phase records and yields convergence of the stationary signed current. A boundary restart law with P\'olya weights supplies the continuity needed to turn that current into stable escape. Diffusion coefficients adapted to a common Kelvin clock preserve the joint Brownian law and uniform donor marks while compressing escape into finite lifetime. Disjoint spherical holes in a ball accumulate only at the explosion point. The boundary is smooth elsewhere, with finite area and perimeter. Boundary area within distance $r$ of the singularity is of order $r$. For three particles, the transfer theorem covers every smooth positive periodic killing rate.
}
\maketitle
\pagestyle{plain}
\thispagestyle{plain}
\setcounter{tocdepth}{1}
\tableofcontents

\newpage

\section{Introduction}\label{sec:introduction}
The non-extinction problem for Brownian Fleming--Viot systems asks whether uniform rebirth prevents infinitely many boundary deaths in finite time in every bounded domain. It is the basic global existence question for the classical particle construction: at each exit, a Brownian particle jumps to a uniformly chosen survivor, and the recursion ends if the replacement times accumulate. We give a negative answer for every fixed population of at least three particles.

The model was introduced by Burdzy, Ho{\l}yst, Ingerman and March in 1996~\cite{BHIM1996} and developed as a particle representation of the Dirichlet Laplacian by Burdzy, Ho{\l}yst and March in 2000~\cite{BHM2000}. An irreparable error in the non-extinction proof in that paper left the general bounded-domain question open, as recorded in Problem~6 of Burdzy's list~\cite{BurdzyProblems}. Its difficulty lies in the feedback between absorption and copying: each rebirth changes the population configuration that determines subsequent victims and donor displacements. Our construction exploits this feedback to produce directed transport from driftless Brownian pieces, and then converts that transport into branching explosion.

For an open set $D\subset\R^d$ and $N\ge2$, write $Z_i(t)$ for the position of label $i$. The particles move as independent standard Brownian motions between replacements. Let $\tau_n$ be the $n$th replacement time and put
\begin{equation}\label{eq:explosion-time}
 \tau_\infty=\lim_{n\to\infty}\tau_n,
\end{equation}
with $\tau_n=\infty$ if fewer than $n$ replacements occur. Branching explosion is the event $\{\tau_\infty<\infty\}$. Finite prefixes are well defined in arbitrary open sets. We write $B(x,r)$ for the open Euclidean ball and $\HH^2$ for two-dimensional Hausdorff measure.

\begin{theorem}[Bounded-domain explosion]\label{thm:bounded-explosion}
For every fixed integer $N\ge3$, there is a bounded connected open set $\Omega_N\subset\R^3$ such that, for every $z\in\Omega_N^N$ and every $T>0$, the classical Brownian Fleming--Viot system satisfies
\begin{equation}\label{eq:main-event}
 \Prob_z\left(\tau_\infty<T,\quad
 \lim_{t\uparrow\tau_\infty}Z_i(t)=0\text{ for all }1\le i\le N\right)>0.
\end{equation}
The domain is an open ball with countably many pairwise disjoint closed balls removed. The removed balls accumulate only at $0\in\partial\Omega_N$, and the boundary is $C^\infty$ away from $0$. Moreover, $\Omega_N$ has finite perimeter and
\begin{equation}\label{eq:main-area}
 \HH^2(\partial\Omega_N)<\infty,\qquad
 c_N^{\rm geom}r\le\HH^2(\partial\Omega_N\cap B(0,r))
 \le C_N^{\rm geom}r\quad(0<r<r_N^{\rm geom})
\end{equation}
for positive constants $c_N^{\rm geom},C_N^{\rm geom},r_N^{\rm geom}$.
\end{theorem}

Three particles are the exact threshold. Kwa\'snicki's two-particle theorem~\cite{Kwasnicki2024}, through its Corollary 1.4, gives non-extinction in every bounded open domain from every interior configuration, with no boundary regularity assumption. The counterexamples above already occur in three dimensions, with the entire singular behavior concentrated at one point and a smooth spherical boundary everywhere else. For each population, one domain works for all initial configurations and all positive target times.

The geometry follows the transport. We first construct a periodic array of absorbing spherical holes in $\R^3$ in which copying drives the population to infinity. Retain the part above a plane $y^{(1)}=R$ missing the closed holes, where $y^{(1)}$ is the first coordinate. Kelvin inversion,
\[
 \mathcal I(y)=y/|y|^2,
\]
maps this half-space to a ball tangent to zero and sends the holes to spheres accumulating there. Figure~\ref{fig:inversion-geometry} shows the resulting domain. The proof must preserve the signed transport through absorption and preserve the Brownian copying law through inversion.
\begin{figure}[htbp]
\centering
\includegraphics{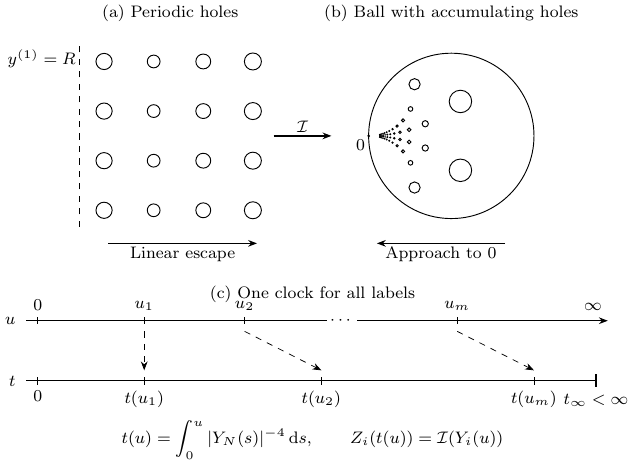}
\caption{Kelvin inversion and the common clock in a schematic section
of the construction in $\R^3$. The cut periodic domain in (a) becomes
the ball in (b), tangent to $0$. For operational positions $Y_i(u)$ and
copy times $u_m$, all labels share the physical times $t(u_m)$ in (c).
With the adapted operational coefficients, stable linear escape gives
$t_\infty=\lim_{u\to\infty}t(u)<\infty$ and
$Z_i(t(u))\to0$ for every label}
\label{fig:inversion-geometry}
\end{figure}

\subsection{The copying current}
Uniform copying turns spatial asymmetry into motion. Each Brownian segment has zero drift, and the generator of a single particle killed at a periodic rate is self-adjoint. In the population, the killing rate selects the victim; its jump samples the survivor cloud relative to that victim's phase. We identify the stationary correlation responsible for the resulting current on the Euclidean cover of the periodic system.

Normalize the neutral copying rate per ordered victim and donor pair to $1/2$. Let the killing rate depend on the first coordinate $s$ through
\begin{equation}\label{eq:intro-rate}
 \lambda_N=(N-1)/2,\qquad
 k_{N,\epsilon}(s)=\lambda_N+\epsilon\phi_*(s),\qquad
 \phi_*(s)=\cos s+\sin 2s.
\end{equation}
Let $v_{N,\epsilon}$ be the stationary longitudinal velocity. For $|\epsilon|\le\lambda_N/4$, we prove
\begin{equation}\label{eq:intro-current}
 v_{N,\epsilon}=\mathfrak c_N\epsilon^3+\epsilon^4r_{N,\epsilon},
 \qquad \mathfrak c_N>0\quad(N\ge3),\qquad
 \sup_{|\epsilon|\le\lambda_N/4}|r_{N,\epsilon}|<\infty.
\end{equation}
The coefficient has a factor $N-2$ and an explicit positive remaining factor. A single profile therefore supplies a positive current for every population in the theorem.

The first surviving response is chiral. For a real two-harmonic profile $\phi(s)=\sum_{\ell\in\{-2,-1,1,2\}}\widehat\phi_\ell e^{i\ell s}$, its orientation is encoded by
\[
 \mathcal C(\phi)=8\operatorname{Im}
       (\widehat\phi_1^2\overline{\widehat\phi_2}).
\]
This quantity is invariant under translation and changes sign under reflection; $\mathcal C(\phi_*)=1$. Translation removes the linear response, and reflection cancels the conjugate pairs at second order. The cubic resonance $1+1=2$ retains the relative phase of the two harmonics.

We compute that response through a finite Fourier elimination. Copying merges observed frequencies, and each perturbation adds at most one observed label, so the cubic calculation closes at rank four. A formal exponent records the Euclidean displacement through differentiation of the remaining root term. Successive Poisson equations turn this finite algebra into an exact stationary current identity. The correctors grow at most linearly in the real relative positions; genealogical moment bounds control the fourth-order remainder. Thus the calculation selects a fixed positive killing rate with a strictly positive current, as proved in Theorem~\ref{thm:copying-current}.

\subsection{Genealogy and signed displacement}
The passage from killing rates to absorbing holes must retain the displacement inherited at each copy. If two donor positions have the same periodic phase, their copying displacements can differ by an integer number of periods. We construct a transfer theorem for marked genealogical records that keeps this information through the absorption limit.

For a smooth positive periodic rate $k$ in three dimensions, place a hole of radius
\[
 a_z=\frac{h^3k(z)}{2\pi}
\]
at the center $z$ of each mesh cell of side $h$. Its Brownian capacity for the generator $\tfrac12\Delta$ is $2\pi a_z$, so capacity per cell volume equals $k(z)$. The one-particle form limit identifies the killing rate. Our marked comparison follows every victim, donor and survivor endpoint, and reconstructs complete phase paths from fine grids.

Write $\operatorname{ph}$ for reduction modulo $L=2\pi$ in each coordinate. For a configuration $x=(x_1,\ldots,x_N)$, the hidden state
\[
 H(x)=(\operatorname{ph}x_N,(x_i-x_N)_{i<N})
\]
retains the actual relative positions. For the periodic populations $X=(X_i)_{i\le N}$, put $H_t=H(X(t))$ and $\Delta_T=X_N(T)-X_N(0)$. Principal phase increments recover the Brownian increments on a sufficiently fine grid. Every initial lattice offset follows its ancestral label through subsequent copies. Once the backward genealogy has one ancestor, the reconstructed positions differ from the true positions by a single common lattice vector. That vector cancels from $H$ and from the reference displacement $\Delta_T$.

Let $\nu_h,\nu_k$ be the hard and soft stationary hidden laws, and define their vector currents by $\mathbf v_h(k)=\E_{\nu_h}\Delta_1$ and $\mathbf v(k)=\E_{\nu_k}\Delta_1$. Uniform death-count moments and an exponential ancestor tail turn the reconstruction into
\[
 \|\Law_{\nu_h}(H_0,\Delta_T)
        -\Law_{\nu_k}(H_0,\Delta_T)\|_\TV\longrightarrow0,
 \qquad \mathbf v_h(k)\longrightarrow\mathbf v(k),
 \quad h\downarrow0,
\]
for every fixed $T>0$. The first convergence is a joint law statement about the Euclidean lift, and uniform displacement moments give the signed mean in the second. Theorem~\ref{thm:marked-transfer} provides this transfer under uniform count and ancestry bounds. A capacity barrier verifies those bounds for small-contrast rates at every $N\ge3$.

For three particles, a lifetime potential supplies an energy estimate that treats holes of arbitrary relative radii. It controls the gradient energy between deaths and the donor squares added at copies. A microscopic rescue construction then produces a uniformly regular root density from every initial configuration, and finite-record comparison gives a common-ancestor bound. This proves transfer for every smooth positive periodic killing rate when $N=3$.

\subsection{Ballistic escape and finite lifetime}
The stationary current must yield escape of the whole population for all times under the coefficient changes used in inversion. We construct a Poisson correction with constant baseline mean increment and add a Foster penalty to keep the drift positive over all hidden states. On Foster sublevel sets, at least one label remains interior. The wall-approaching labels restart through a rapid copying cascade with P\'olya weights unchanged by adapted orders of deaths. This law gives the boundary continuity needed for uniform perturbation comparison, including nonperiodic coefficients depending on the full configuration.

From high translated seeds, with probability tending to one as the seed height increases, the reference coordinate has a positive linear lower bound and the cloud diameter stays small relative to it throughout the path. Block martingale bounds and excursion estimates carry the positive drift between observations. These are the escape properties needed to make inversion a small coefficient perturbation.

At a copy, victim and donor positions must be evaluated at the same physical time. We choose one clock for the entire population. With operational positions $Y_i(u)$, set
\begin{equation}\label{eq:intro-clock}
 t(u)=\int_0^u |Y_N(s)|^{-4}\,\dd s,\qquad
 Z_i(t(u))=\mathcal I(Y_i(u)).
\end{equation}
The operational diffusion and drift coefficients in the high, narrow region are
\[
 \sigma_i^K(Y)=\frac{|Y_i|^2}{|Y_N|^2}I_3,
 \qquad b_i^K(Y)=-\frac{|Y_i|^2Y_i}{|Y_N|^4},
\]
where $I_3$ is the identity matrix. They cancel the inversion drift and give every inverted label the same scalar covariance clock. The joint martingale identification also preserves independence of the Brownian drivers and uniform donor marks. Smooth extensions of these coefficients fall within the perturbation class controlled by stable escape.

On the escape event all $|Y_i(u)|$ grow at least linearly, so $t(\infty)<\infty$ and every $Z_i$ tends to zero. Point polarity in three dimensions forces infinitely many replacements before this clock endpoint. Raising the seed makes the remaining physical lifetime arbitrarily short. Positive Brownian tube probabilities in the connected inverted domain then transfer explosion to every interior initial configuration and every positive target time. The common-clock conversion is proved in Theorem~\ref{thm:kelvin-explosion}.

\subsection{Related work}
Non-extinction and stationary behavior have been studied under boundary regularity assumptions. L\"obus~\cite{Lobus2009} gave a structural formula for the invariant measure of the stationary Brownian Fleming--Viot system. Bieniek, Burdzy and Finch~\cite{BBF2012} proved non-extinction in bounded Lipschitz domains with a sufficiently small Lipschitz constant, and for two particles in bounded polyhedral domains. Grigorescu and Kang~\cite{GK2012} obtained criteria for smooth diffusions, including Brownian systems in nonsmooth domains with an integrable Martin kernel, and proved that a single ancestry line survives. Villemonais~\cite{Villemonais2014} established a general approximation method for conditioned Markov processes and non-explosion criteria for diffusion particle systems. Bieniek, Burdzy and Pal~\cite{BBP2012} proved finite-time extinction for two particles driven by Bessel processes of negative dimension, and for every population size in Brownian diffusions with drift stronger than the Bessel drift.

The genealogy carries its own spatial dynamics. Bieniek and Burdzy~\cite{BB2018} proved uniqueness of the spine under general assumptions and obtained its large-population limit for finite-state driving chains. Burdzy and Engl\"ander~\cite{BE2024} proved convergence of the Brownian Fleming--Viot spine, as the population tends to infinity, to Brownian motion conditioned to stay forever in a bounded Lipschitz domain with Lipschitz constant less than one. For the two-particle half-line system, Burdzy and Tadi{\'c}~\cite{BT2024} used a logarithmic map of the joint configuration and a common time change to analyze the spine. Burdzy, Engl\"ander and Marshall~\cite{BEM2025} showed that the two-particle spine in an interval has a path law mutually singular to that of this conditioned motion. Tough~\cite{Tough2023} obtained a Wright--Fisher limit for inherited colours on the accelerated time scale $t\mapsto Nt$, with a hard-killing extension for Brownian systems in smooth bounded domains.

Lin and Lu~\cite{LL2026} proved universal two-particle ancestral balance for finite reversible chains and generic circulation on a fixed three-state graph for every $N\ge3$. For $N\ge4$, their bounded-degree perturbation calculation gives a cubic ancestral current. The present construction computes an inherited Euclidean displacement current and carries it through spherical absorption to Brownian branching explosion.

For empirical measures, Grigorescu and Kang~\cite{GK2004} established the hydrodynamic limit jointly with the normalized boundary-visit count. Their tagged-particle analysis~\cite{GK2006} identified the limiting rebirth law and proved propagation of chaos. For non-explosive systems, C\'erou, Delyon, Guyader and Rousset~\cite{CDGR2020} established fixed-time central limit theorems for empirical distributions and survival-probability estimators. Journel and Monmarch\'e~\cite{JM2025} obtained exponential mixing with a decay rate independent of the population size and propagation of chaos uniformly in time for low-temperature metastable gradient diffusions with hard killing. Cardaliaguet, Cirant, Jackson and Souganidis~\cite{CCJS2026} obtained a weak error of order $N^{-1}$, uniformly in time, for smooth functionals of the empirical measure of hard-killing Brownian systems in smooth bounded domains, with independent initial particles drawn from a compactly supported law in the domain.

The capacity scaling used for spherical absorption is the critical perforation regime of Cioranescu and Murat~\cite{CM1997}. Cherednichenko, Dondl and R\"osler~\cite{CDR2018} proved norm-resolvent convergence for periodic spherical perforations with a constant effective absorption term. The capacity and kernel limits for the spatially varying hole radii used in our marked transfer theorem are proved in Appendix~\ref{app:capacity}.

\subsection{Organization}
The proof follows the positive longitudinal current through the construction:
\[
 \underbrace{\mathbf v(k)^{(1)}>0}_{\text{soft current}}
 \longrightarrow
 \underbrace{\mathbf v_h(k)^{(1)}>0}_{\text{one fixed mesh}}
 \longrightarrow
 \underbrace{\text{stable escape}}_{\text{Theorem~\ref{thm:stable-current}}}
 \longrightarrow
 \underbrace{\Prob_z(\tau_\infty<T)>0}_{\text{Theorem~\ref{thm:bounded-explosion}}}.
\]

The mesh limit selects one perforated domain. Stability and Kelvin inversion
then use that fixed mesh. After the dynamics and records, Sections 3 to 5
establish the current and its transfer; Sections 6 and 7 carry out escape and
inversion. The appendices contain the exact algebra, capacity and kernel
limits, local boundary arguments, and the arbitrary rate estimates for
three particles.

\section{Dynamics and genealogical records}\label{sec:dynamics}
Throughout the construction $N\ge3$ is fixed, the space dimension is three, and the period is $L=2\pi$. The current calculation uses the first spatial coordinate. For a configuration $x=(x_1,\ldots,x_N)$, $C_{ij}x$ replaces $x_i$ by $x_j$, where $i\ne j$. Particle paths are right-continuous, with the copied position recorded at each replacement time. For a spatial vector $y$, $y^{(1)}$ denotes its first coordinate. We write
\begin{equation}\label{eq:cloud-diameter}
 \mathfrak D(x)=\max_{i,j}|x_i-x_j|.
\end{equation}
The vector $e_1$ is the first coordinate unit vector. For a measure $\mu$, write $\mu f=\int f\dd\mu$. Our total variation convention is $\|\mu-\nu\|_{\TV}=\sup_{|f|\le1}|\mu f-\nu f|$.

\begin{lemma}[Brownian finite prefixes]\label{lem:finite-prefix}
For Brownian motion started at an interior point of an open set $D\subset\R^d$, the exit time is positive almost surely and its finite part has no atoms. Every finite prefix of the classical Brownian Fleming--Viot system in $D$ has a unique exiting label and is well defined.
\end{lemma}
\begin{proof}
For $x\in D$, choose $a>0$ with $\overline B(x,a)\subset D$. Let $W$ be
Brownian motion started at $x$, and let $\eta,U$ be its exit time and
exit point of this ball. Continuity and
the Gaussian law give
\[
 \Prob_x(\eta>0)=1,\qquad
 \Prob_x(\eta=t)\le\Prob_x(|W_t-x|=a)=0\quad(t>0).
\]
Rotational invariance, or reflection when $d=1$, gives
\[
 \Law_x(\eta,U)=\Law_x(\eta)\otimes
                  \mathrm{Unif}(\partial B(x,a)).
\]
For the domain exit time $\zeta=\eta+R$, the strong Markov property gives
\[
 \Law_x(R\mid\eta,U)=\Law_U(\zeta),\qquad
 \Law_x(\eta,R)=\Law_x(\eta)\otimes\Law_x(R).
\]
Allowing $R=\infty$, we obtain
\[
 \Prob_x(\zeta=t)
 =\int_{[0,\infty)}\Prob_x(\eta=t-r)\,\Law_x(R)(\dd r)=0.
\]
Set $\tau_0=0$. Conditional on the post-copy filtration $\mathcal F_{\tau_n}$, the next
exit waits $\zeta_1,\ldots,\zeta_N$ are independent, positive and
atomless on their finite parts. Consequently
\[
 \Prob\!\left(\exists i\ne j:
       \zeta_i=\zeta_j<\infty\mid\mathcal F_{\tau_n}\right)=0.
\]
At the finite minimum all survivors are interior. Induction on $n$
completes the construction.
\end{proof}

\subsection{Soft killing and the Euclidean cover}
For a smooth positive rate $k$ on $\R^d$ satisfying $k(y+Lm)=k(y)$ for every $m\in\Z^d$, the soft system has generator
\begin{equation}\label{eq:soft-generator}
 \cL_kf(x)=\tfrac12\sum_{i=1}^N\Delta_i f(x)
 +\frac1{N-1}\sum_{i\ne j}k(x_i)[f(C_{ij}x)-f(x)].
\end{equation}
Use independent Brownian drivers and candidate clocks of rate
$k_{\max}=\|k\|_\infty$ for each label. At a candidate for label $i$,
accept with probability $k(X_i)/k_{\max}$ and choose a uniform donor.
The victim's position is evaluated before copying. For every $T<\infty$,
\[
 \begin{gathered}
 K_T^{\rm cand}\sim\mathrm{Poisson}(Nk_{\max}T),\\
 \#\{n\ge1:\tau_n\le T\}\le K_T^{\rm cand}<\infty
 \quad\text{almost surely}.
 \end{gathered}
\]

Let $M_d=(\R/L\Z)^d$ and let $\operatorname{ph}$ be componentwise reduction modulo $L$. The phase process is $(\operatorname{ph}X_i)_{i\le N}$. Its hidden extension is
\begin{equation}\label{eq:hidden-state}
 H(X)=\bigl(\operatorname{ph}X_N,\ X_1-X_N,\ldots,X_{N-1}-X_N\bigr).
\end{equation}
The reference phase and real relative coordinates determine the configuration
up to one common lattice translation. In particular, they determine every
copying displacement $X_j-X_i$, including the lattice offsets inherited
from donor positions with identical phases. Write $H_t=H(X(t))$. For a periodic hard domain $D$, the hidden space is $H(D^N)$. We denote the soft and hard invariant hidden laws by $\nu_k$ and $\nu_h$, respectively, and their phase projections by $\pi_k$ and $\pi_h$. Their existence and moments are proved below. The reference increment is
\begin{equation}\label{eq:reference-reward}
 \Delta_T=X_N(T)-X_N(0).
\end{equation}
The relative path
\[
 \mathsf Y_T=(X_i(t)-X_N(t))_{i<N,\ 0\le t\le T}
 \in D([0,T],\R^{3(N-1)})
\]
is right-continuous with left limits. Joint path convergence uses the
Skorokhod $J_1$ topology, allowing a common reparameterization of copying times.
The block excursion is
\[
 \mathfrak X_T=\sup_{s\le T}|X_N(s)-X_N(0)|
                  +\sup_{s\le T}\mathfrak D(X_s).
\]
Its law depends on the initial configuration through $H_0$. Once
integrability is established, hidden stationarity gives
\[
 \mathbf v=\E_\nu\Delta_1,\qquad
 \E_\nu\Delta_{s+t}=\E_\nu\Delta_s+\E_\nu\Delta_t,
 \qquad \E_\nu\Delta_T=T\mathbf v.
\]

\subsection{Genealogical records}
Write $\mathcal M_N=\{(i,j):1\le i,j\le N,\ i\ne j\}$ with the discrete topology. One complete copying record belongs to
\begin{equation}\label{eq:complete-single-record}
 \mathcal R=(0,\infty)\times\mathcal M_N\times C([0,1],M_d)^N.
\end{equation}
Write a record as $(u,(i,j),\eta)$: $u$ is the wait and $\eta$ contains
the phase paths up to death, rescaled to $[0,1]$. For a horizon $T$,
\[
 \begin{split}
 \mathcal R_{\rm path}(T)=\coprod_{m\ge0}\Bigl\{
 &((u_r,(i_r,j_r),\eta_r)_{r=1}^m,u_{m+1},\eta_{m+1}):\\
 &u_1,\ldots,u_m>0,\quad u_{m+1}\ge0,\quad
                    \sum_{r=1}^{m+1}u_r=T,\\
 &(i_r,j_r)\in\mathcal M_N,\quad
             \eta_r\in C([0,1],M_d)^N,\\
 &u_{m+1}=0\ \Longrightarrow\ \eta_{m+1}\text{ is constant}
 \Bigr\}.
 \end{split}
\]
Use the product of Euclidean, discrete and uniform path topologies on
each component. This ambient space is Polish; the first $m$ complete
records lie in $\mathcal R^m$.

To reconstruct a realized record, choose a representative $X(0)$ of $H_0$ and
put $\tau_r=\sum_{a\le r}u_a$. Let $\widetilde\eta_r$ be the unique
continuous lift starting at $X(\tau_{r-1})$. Then
\[
 \begin{aligned}
 X(t)&=\widetilde\eta_r\bigl((t-\tau_{r-1})/u_r\bigr),
                   &&\tau_{r-1}\le t<\tau_r,\\
 X(\tau_r)&=C_{i_rj_r}\widetilde\eta_r(1),
                   &&1\le r\le m.
 \end{aligned}
\]
The terminal piece uses the same lift rule. A common lattice change in
$X(0)$ leaves $H_t$ and $\Delta_T$ unchanged.

For a finite grid $G\subset[0,T]$ containing $0,T$, a record in
$\mathcal R_{\rm grid}$ has the form
\[
 \left(
   (\tau_r,i_r,j_r,(\operatorname{ph}X_\ell(\tau_r-))_{\ell\ne i_r})_{r\le K_T},
   (\operatorname{ph}X_i(t))_{i\le N,\ t\in G}
 \right),
\]
where $K_T$ is the death count. The donor determines the victim's
poststate. The compressed record retains the survivor endpoints and omits
the victim's pre-death endpoint. That endpoint lies on a spherical boundary
in the hard system and has a volume density in the soft system.
The survivor record admits total variation comparison; complete phase
pieces are recovered from fine grids using Brownian moduli of continuity.
The stationary target variable is
\begin{equation}\label{eq:target-record}
 \mathcal Z_T=(H_0,\Delta_T).
\end{equation}
The transfer theorem gives total variation convergence for
$\mathcal R_{\rm grid}$ and $\mathcal Z_T$, and weak convergence for
$\mathcal R_{\rm path}$.

For positive integers $m$, put $[m]=\{1,\ldots,m\}$. The ancestral label map
$\mathfrak a_{s,t}:[N]\to[N]$ starts at $\mathfrak a_{s,s}=\id$ and,
at a copy $i\leftarrow j$ at time $\tau>s$, obeys
\[
 \mathfrak a_{s,\tau}(\ell)=
 \begin{cases}
   \mathfrak a_{s,\tau-}(j),&\ell=i,\\
   \mathfrak a_{s,\tau-}(\ell),&\ell\ne i.
 \end{cases}
 \qquad
 \mathsf A_{s,t}=\{|\mathfrak a_{s,t}([N])|=1\}.
\]
Thus $\mathsf A_{s,t}$ is the event of one common ancestor at time $s$,
and $\mathfrak a_{r,t}=\mathfrak a_{r,s}\circ\mathfrak a_{s,t}$ for
$r\le s\le t$. These maps act on labels, while $C_{ij}$ acts on positions.

A partition-root initial law has $\operatorname{ph}X_i(0)=Y_{b(i)}$, where
$b:[N]\twoheadrightarrow[r]$ specifies the partition blocks and $(Y_1,\ldots,Y_r)$
has a density on $M_d^r$. Labels use independent Brownian drivers,
and their initial ancestral identities are the distinct labels in $[N]$.

\subsection{Filtrations and perturbations}
Let $s_i$ enumerate the surviving labels in increasing order. At death
$\tau_n$ of label $i_n$, use independent donor marks $U_n$ with
\[
 j_n=s_{i_n}(U_n),\qquad
 \Law(U_n\mid\mathcal F_{\tau_n-},i_n)
       =\mathrm{Unif}\{1,\ldots,N-1\}.
\]
Here $\mathcal F_t$ is the natural filtration, revealing $U_n$ at
$\tau_n$. Count estimates and restarts use this filtration; the strong
Markov property supplies fresh Brownian primitives at each stopping time.
Auxiliary hard systems have the smooth configuration-dependent
coefficients and independent label noises in~\eqref{eq:perturbation}.
Lemmas~\ref{lem:simultaneous-wall-polarity} and
\ref{lem:all-death-budget} give unique finite prefixes and operational
nonexplosion. The Kelvin identification uses
$\widehat{\mathcal F}_t=\mathcal F_t\vee\sigma(U_n:n\ge1)$ to establish the joint
law of the Brownian drivers and marks.

\section{The cubic copying current}\label{sec:copying-current}
Fix $N\ge3$, set $\lambda_N=(N-1)/2$, and use period $L=2\pi$.
Particles move as standard Brownian motions in $\R^3$ and label $i$ dies
at rate
\[
 k_{N,\epsilon}(x)=\lambda_N+\epsilon\phi_*(x^{(1)}),
 \qquad \phi_*(s)=\cos s+\sin2s,
\]
as in~\eqref{eq:intro-rate}, with the copying rule of
Section~\ref{sec:dynamics}.
For $|\epsilon|\le\lambda_N/4$, the bound $\|\phi_*\|_\infty\le2$ places
the rate in $[\lambda_N/2,3\lambda_N/2]$.

Copying supplies the drift of the centre of mass
$\bar X=N^{-1}\sum_iX_i$. For a periodic rate $k$, this drift is the
hidden-state function
\begin{align}\label{eq:soft-com-drift}
 j_k(X)&=\frac1{N(N-1)}\sum_{i\ne j}k(X_i)(X_j-X_i)\notag\\
       &=\frac1{N-1}\sum_i k(X_i)(\bar X-X_i).
\end{align}
For a constant rate, $j_k=0$. A spatial perturbation weights each
victim's displacement towards its donor, and stationary correlations
can give the weighted sum a nonzero mean. For $\phi_*$, the response
first appears at cubic order: two first-harmonic factors resonate with
one second-harmonic factor. The calculation below determines its sign
for every $N\ge3$.

\begin{lemma}\label{lem:soft-hidden-moments}
The hidden process has a unique invariant law $\nu_{N,\epsilon}$, and, for
each finite $p\ge1$,
\begin{equation}\label{eq:soft-moments}
 \sup_{|\epsilon|\le\lambda_N/4}\nu_{N,\epsilon} \mathfrak D^p<\infty.
\end{equation}
There are constants $\beta_N>0$ and $A_N<\infty$, defined below, such that
\begin{equation}\label{eq:explicit-moment}
 \nu_{N,\epsilon}\mathfrak D\le A_N/\beta_N.
\end{equation}
\end{lemma}
\begin{proof}
Use independent base clocks of rate $\lambda_N/2$ for each label and extra
candidate clocks of rate $\lambda_N$. At an extra event for label $i$,
accept with probability
$[k_{N,\epsilon}(X_i)-\lambda_N/2]/\lambda_N$.
In the first $N-1$ unit windows require
exactly one base event, for label $a$ in window $a$, with donor $N$.
Forbid all other candidates, and leave the last unit window quiet.
Call this event $E$ and put $T_N=N$, $r_N=3N\lambda_N/2$. The Poisson
construction gives, independently of the Brownian drivers,
\[
 \Prob(E)=\beta_N
 =\left(\frac{\lambda_N}{2(N-1)}\right)^{N-1}e^{-r_NT_N}>0,
 \qquad E\subset\mathsf A_{0,T_N}.
\]
Conditional on $E$ and its copying times $\mathbf t$, all endpoints
have the initial root as their mean:
\[
 \begin{gathered}
 \Law(X(T_N)\mid E,\mathbf t)
 =\mathcal N_{3N}((X_N(0),\ldots,X_N(0)),\Sigma_{\mathbf t}),\\
 I\le\Sigma_{\mathbf t}\le3NT_N I.
 \end{gathered}
\]
The lower bound comes from the final quiet unit interval; the upper
bound follows from the trace $3NT_N$. Modulo common lattice
translations, the mean lies in the root torus and the relative mean is
zero. The wrapped Gaussian density is therefore uniformly positive on
a fixed compact hidden set.
Writing $P_t^{N,\epsilon}$ for the hidden transition kernel, for some
probability $m_N$ and $\eta_N>0$ we obtain
\[
 P_{T_N}^{N,\epsilon}(h,\cdot)\ge\eta_Nm_N(\cdot)
 \quad\text{for every }h\text{ and }|\epsilon|\le\lambda_N/4.
\]
This global minorization gives a unique invariant probability and
geometric convergence of the block chain.

Let $K_T$ be the total candidate count, let $W_i$ be the independent
Brownian drivers, and put $U_T=\max_i\sup_{s\le T}|W_i(s)|$.
Ancestral paths consist of at most
$K_T+1$ Brownian pieces. On the common-ancestor event $E$, their initial
relative positions disappear, so
\[
 \mathfrak D(X_{T_N})\le \mathfrak D(X_0)\one_{E^c}+4(K_{T_N}+1)U_{T_N}.
\]
Candidate counts and Brownian drivers are independent. Put
\[
 A_N=4(1+r_NT_N)\E U_{T_N}<\infty,
 \qquad \E U_T\le2\sqrt{3NT}.
\]
The first-moment contraction is
\[
 \E_x\mathfrak D(X_{T_N})\le(1-\beta_N)\mathfrak D(x)+A_N.
\]
For $p\ge1$, choose $\delta>0$ with
$\rho_p=(1+\delta)(1-\beta_N)<1$. The inequality
$(a+c)^p\le(1+\delta)a^p+C_{p,\delta}c^p$ gives
\[
 \begin{aligned}
 P_{T_N}^{N,\epsilon}\mathfrak D^p(x)
 &\le\rho_p\mathfrak D(x)^p
       +C_{p,\delta}\E(K_{T_N}+1)^p\,\E U_{T_N}^p
 \le\rho_p\mathfrak D(x)^p+C_{p,N},\\
 P_{nT_N}^{N,\epsilon}\mathfrak D^p(x)
 &\le\rho_p^n\mathfrak D(x)^p+C_{p,N}/(1-\rho_p).
 \end{aligned}
\]
Poisson and Brownian maximal moments make $C_{p,N}$ finite. Truncated
Fatou along convergence to $\nu_{N,\epsilon}$ proves
\eqref{eq:soft-moments}; the first-moment bound gives
\eqref{eq:explicit-moment}.
\end{proof}

For a fixed class $0<k_-\le k\le k_+$, use base rate $k_-$ and extra
candidate rate $k_+-k_-$. The same event has probability
$[k_-/(N-1)]^{N-1}e^{-Nk_+T_N}>0$ and the same conditional Gaussian
bounds. Thus the minorization and all moment constants depend only on
$N,k_-,k_+$, supplying the stationary laws for arbitrary-rate transfer.

\subsection{Translation symmetry and the phase of the profile}
For a periodic rate, write $v_N(k)$ for the first component of its
stationary velocity, and set $v_{N,\epsilon}=v_N(k_{N,\epsilon})$.
Common longitudinal translation and reflection commute with copying.
For longitudinal rates, uniqueness of the hidden invariant law gives
\begin{equation}\label{eq:current-symmetries}
 v_N(k(\cdot+\theta))=v_N(k),\qquad v_N(k(-\cdot))=-v_N(k).
\end{equation}
In particular, a rate that is even after a shift has zero velocity.

For a real two-harmonic profile, use the Fourier convention
$\phi(s)=\sum_\ell\widehat\phi_\ell e^{i\ell s}$ and define
\begin{equation}\label{eq:chiral-invariant}
 \mathcal C(\phi)
 =8\operatorname{Im}
       (\widehat\phi_1^2\overline{\widehat\phi_2}).
\end{equation}
For real $\phi$, translation and reflection act by
\[
 \begin{aligned}
 \widehat{\phi(\cdot+\theta)}_\ell
     &=e^{i\ell\theta}\widehat\phi_\ell,
 &\mathcal C(\phi(\cdot+\theta))&=\mathcal C(\phi),\\
 \widehat{\phi(-\cdot)}_\ell
     &=\overline{\widehat\phi_\ell},
 &\mathcal C(\phi(-\cdot))&=-\mathcal C(\phi).
 \end{aligned}
\]
In particular,
\[
 \begin{gathered}
 \phi(s)=a\cos(s+\alpha)+b\cos(2s+\beta),\\
 \mathcal C(\phi)=a^2b\sin(2\alpha-\beta).
 \end{gathered}
\]
for $a,b,\alpha,\beta\in\R$. The chosen profile has $\mathcal C(\phi_*)=1$.
Figure~\ref{fig:copying-current} connects this orientation with the
uniform-copying drift~\eqref{eq:soft-com-drift}.
\begin{figure}[tbp]
\centering
\includegraphics{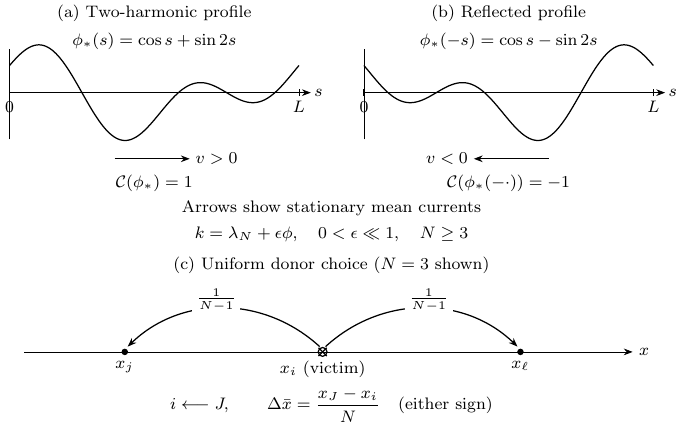}
\caption{Copying and the directed stationary current. For $N\ge3$ and
sufficiently small positive $\epsilon$, the profile in (a) and its
reflection in (b) have opposite chiralities and longitudinal stationary
mean currents, labelled $v$.
Panel (c) shows three labels, with $x_i=X_i^{(1)}$,
$\bar x=N^{-1}\sum_i x_i$ and a donor $J$ chosen uniformly among the
survivors. The centre of mass jump may have either sign}
\label{fig:copying-current}
\end{figure}

\subsection{Finite Fourier elimination}
The neutral generator is triangular in the number of observed labels:
copying merges frequencies. Each perturbation adds at most one observed
label, so the cubic calculation closes at rank four. We retain one
formal exponent to recover the displacement from the surviving root term.

Only first coordinates are used here; write $x_i=X_i^{(1)}$. For $n\le N$
and an exponent list $k=(k_1,\ldots,k_n)$ put
\begin{equation}\label{eq:En}
 E_n(k;x)=\frac1{N^{\underline n}}
  \sum_{i_1,\ldots,i_n\ \mathrm{distinct}}
      \exp\left(i\sum_{a=1}^nk_ax_{i_a}\right).
\end{equation}
Here $N^{\underline n}=N(N-1)\cdots(N-n+1)$. The neutral and perturbation generators on these first coordinates are
\begin{equation}\label{eq:normalized-generator}
 \cL_0f=\tfrac12\Delta f+\tfrac12\sum_{i\ne j}[f(C_{ij}x)-f(x)],
 \quad \cA f=\frac1{N-1}\sum_{i\ne j}\phi_*(x_i)[f(C_{ij}x)-f(x)].
\end{equation}
Lists can be sorted and ordinary zero exponents deleted. Exactly one
exponent contains a formal parameter $q$; that exponent is retained before
evaluation at $q=0$. Write $\cL_\epsilon=\cL_0+\epsilon\cA$.
Neutral ordered victim--donor pairs have rate $1/2$. Direct averaging gives
\begin{equation}\label{eq:neutral}
 \cL_0E_n=-\left(\frac12\sum_a k_a^2+\tfrac12 n(n-1)\right)E_n
       +\tfrac12\sum_{a\ne c}E_{n-1}(k^{a\to c}),
\end{equation}
where $k^{a\to c}$ deletes $k_a$ and replaces $k_c$ by $k_c+k_a$.
Copies to unobserved labels cancel in the symmetric average.

Write $k[a\leftarrow u,c\leftarrow v]$ for the list obtained by
replacing entries $a,c$ of $k$ by $u,v$, and use the analogous notation
for one replacement. For a rate perturbation $e^{i\ell x}$,
\begin{align}\label{eq:mode}
 \cA_\ell E_n=\sum_{a=1}^n\Bigg[&\frac1{N-1}\sum_{c\ne a}
 E_n(k[a\leftarrow\ell,c\leftarrow k_c+k_a])\notag\\
 &+\frac{N-n}{N-1}E_{n+1}(k[a\leftarrow\ell],k_a)\notag\\
 &-E_n(k[a\leftarrow k_a+\ell])\Bigg].
\end{align}
In the second term, $k_a$ is appended to the modified list; that term
is omitted when $n=N$. All entries on the right use their original
values. In each gain term the inserted
frequency $\ell$ remains on the victim slot, since the rate is evaluated
before copying.
The Fourier coefficients of $\phi_*(s)=\cos s+\sin2s$ are
\begin{equation}\label{eq:fourier}
 \widehat\phi_{\pm1}=\tfrac12,\quad \widehat\phi_2=-\tfrac i2,\quad
 \widehat\phi_{-2}=\tfrac i2,\qquad \cA=\sum_\ell \widehat\phi_\ell\cA_\ell.
\end{equation}

Let $s_0(q)=-q^2/2$, and call $E_1(q)$ the root.
Its derivative records the centre of mass:
\[
 i^{-1}\partial_qE_1(q;x)\big|_{q=0}
       =N^{-1}\sum_i x_i.
\]
At $q=0$, the root is the constant function.
For a non-root exponent list set
\[
 d_k(q)=\tfrac12\big(\textstyle\sum_a k_a^2-q^2\big)+\tfrac12 n(n-1).
\]
Its value at $q=0$ is positive. Resolve a finite expression $g$ in
decreasing rank: divide the coefficient of each non-root term by $d_k$,
retain the resulting term in $\mathcal Rg$, and pass half that coefficient to
every merger in \eqref{eq:neutral}. Leave the root coefficient
$\mathcal Sg$ undivided. Rank strictly decreases in a merger, so
\begin{equation}\label{eq:resolvent}
 (s_0-\cL_0)\mathcal Rg=g-\mathcal Sg\,E_1(q).
\end{equation}
Define
\begin{equation}\label{eq:finite-recursion}
 F_0=E_1(q),\qquad F_j=\mathcal R(\cA F_{j-1}),\qquad
 s_j=\mathcal S(\cA F_{j-1}),\quad j=1,2,3.
\end{equation}
Only first-order Taylor data at $q=0$ are needed. With $y=(N-1)^{-1}$
the coefficients in \eqref{eq:mode} are $y$ and $1-(n-1)y$.
For dual coefficients the division rule is
\begin{equation}\label{eq:dual}
 \frac{c_0+c_1q}{d_0+d_1q}
 =\frac{c_0}{d_0}+q\left(\frac{c_1}{d_0}
                  -\frac{c_0d_1}{d_0^2}\right)+O(q^2).
\end{equation}
All operations are finite rational polynomial operations in $y$.
The largest rank is four. Every rank-four creation coefficient contains
$(1-y)(1-2y)$. At $N=3$, substituting $y=1/2$ and deleting zero terms
reduces the recurrence to functions of rank at most three.

\begin{lemma}\label{lem:cubic}
The recursion satisfies $s_j(0)=0$ for $j\le3$,
$s_1'(0)=s_2'(0)=0$, and
\begin{equation}\label{eq:cubic}
 s_3'(0)=i\mathfrak c_N,\qquad \mathfrak c_N>0.
\end{equation}
The exact value of $\mathfrak c_N$ is given
in~\eqref{eq:population-current}.
\end{lemma}
\begin{proof}
First evaluate the recurrence at $q=0$ as an identity of phase functions.
The root is $F_0(0)=1$, and $\cA1=0$. Integrating
\eqref{eq:resolvent} for $F_1$ against the neutral phase invariant law
gives $s_1(0)=0$, hence $\cL_0F_1(0)=0$. The function $F_1(0)$ is
bounded, so neutral phase mixing makes it constant. Every non-root
Fourier term at $q=0$ has a nonzero integer exponent and zero average
under product Lebesgue measure. Its constant value is consequently zero.
The same argument applies inductively: if $F_{j-1}(0)=0$, then
$\cA F_{j-1}(0)=0$, and
\[
 s_j(0)=0,\qquad F_j(0)=0,\qquad j=1,2,3.
\]

For an inserted word $w=(\ell_1,\ldots,\ell_r)$, define its unweighted
root derivative by
\[
 \rho_N(w)=\left.\partial_q\!\left[
   \mathcal S\cA_{\ell_r}\mathcal R\cA_{\ell_{r-1}}\cdots
   \mathcal R\cA_{\ell_1}E_1(q)\right]\right|_{q=0},
 \qquad r=2,3.
\]
Insertion adds the mode frequency and neutral mergers preserve it.
Thus all resulting terms have total exponent $q+\sum_a\ell_a$, and
$\rho_N(w)=0$ unless $\sum_a\ell_a=0$.
The coefficients of these operations are real. Simultaneously reversing
$q$ and all inserted frequencies leaves the operations unchanged and
gives
\begin{equation}\label{eq:word-reflection}
 \rho_N(-\ell_1,\ldots,-\ell_r)
       =-\rho_N(\ell_1,\ldots,\ell_r).
\end{equation}

There is no single frequency equal to zero in~\eqref{eq:fourier}, so the
linear derivative vanishes. At order two, each resonant word
$(\ell,-\ell)$ is paired with $(-\ell,\ell)$. Their Fourier products
both equal $|\widehat\phi_\ell|^2$, and their derivatives cancel
by~\eqref{eq:word-reflection}. This gives $s_2'(0)=0$.

At order three the only resonances are the three distinct permutations
of $(1,1,-2)$ and the three permutations with all signs reversed. Put
\[
 W_N=\rho_N(1,1,-2)+\rho_N(1,-2,1)+\rho_N(-2,1,1).
\]
For $\phi_*$, the first three words have common Fourier product $i/8$.
Their negatives have product $-i/8$ and the opposite root derivatives.
Thus
\[
 s_3'(0)=\frac{i}{4}W_N.
\]
The finite rational calculation in
Appendix~\ref{app:general-resonance} gives
$W_N=4\mathfrak c_N$, with the positive factorization
in~\eqref{eq:population-current}. It supplies the individual ordered
contributions and their exact sum. This proves~\eqref{eq:cubic}.
\end{proof}

\subsection{Correctors on the Euclidean hidden space}
Differentiating the finite elimination produces Poisson identities
on the Euclidean cover. The hidden stationary moments then control
their fourth-order remainder.

The scalar functions below are evaluated on three-dimensional
configurations through $x_i=X_i^{(1)}$.
Put $h_0=\bar x=N^{-1}\sum_i x_i$ and
$h_j=i^{-1}\partial_qF_j(0)$ for $j=1,2,3$.
Conjugation and simultaneous reversal of exponents preserve the finite
recurrence for real $\phi_*$, so $F_j(-q)=\overline{F_j(q)}$ and $h_j$ is real.
Since $s_0'(0)=0$ and $F_j(0)=s_j(0)=0$ for $j\ge1$, differentiating
\eqref{eq:resolvent} leaves the exact identities
\begin{equation}\label{eq:correctors}
 \cL_0h_1=-\cA h_0,\quad \cL_0h_2=-\cA h_1,\quad
 \cL_0h_3=\mathfrak c_N-\cA h_2.
\end{equation}
Each corrector is a finite sum of Fourier terms and Fourier terms times
linear coordinates. A common lattice translation adds a multiple of
$F_j(0)=0$; hence the correctors are hidden-state functions. Subtracting
$x_N$ from every linear coordinate gives finite constants $B_{j,N}>0$ with
\begin{equation}\label{eq:corrector-growth}
 |h_j|+|\cA h_j|+\sum_i(|\nabla_i h_j|+|\nabla_i^2h_j|)
       \le B_{j,N}(1+\mathfrak D).
\end{equation}
The constants are finite sums of absolute values of the rational
coefficients, including those of the pure Fourier terms and their
derivatives.

For an explicit first-order formula, put
\[
 e_\ell=N^{-1}\sum_a e^{i\ell x_a},\qquad
 B_\ell=\frac1{N(N-1)}\sum_{a\ne c}(x_c-x_a)e^{i\ell x_a}.
\]
Direct differentiation and copying give
\[
 \cL_0e_\ell=-\tfrac12\ell^2e_\ell,\qquad
 \cL_0B_\ell=-(\ell^2/2+1)B_\ell-i\ell e_\ell.
\]
The corrector for the individual mode $e^{i\ell x}$ is therefore
\[
 h_{1,\ell}=\frac{B_\ell-(2i/\ell)e_\ell}{\ell^2/2+1}.
\]
For a real function $\psi$, put
\[
 B_\psi=\frac1{N(N-1)}\sum_{a\ne c}(x_c-x_a)\psi(x_a).
\]
In real notation,
\begin{equation}\label{eq:real-h1}
 h_1=\tfrac23B_{\cos}+\tfrac13B_{\sin2}
       +\tfrac43N^{-1}\sum_a\sin x_a
       -\tfrac13N^{-1}\sum_a\cos2x_a.
\end{equation}

\begin{theorem}[A copying-induced cubic current]\label{thm:copying-current}
For $|\epsilon|\le\lambda_N/4$, the longitudinal stationary reference
velocity satisfies
\begin{equation}\label{eq:exact-current}
 v_{N,\epsilon}=\mathfrak c_N\epsilon^3+\epsilon^4\nu_{N,\epsilon}(\cA h_3),
 \qquad |\nu_{N,\epsilon}(\cA h_3)|\le M_N<\infty.
\end{equation}
The cubic coefficient is
\begin{equation}\label{eq:population-current}
 \mathfrak c_N
 =\frac{(N-2)[859+2282(N-3)]}{323400(N-1)^2}>0.
\end{equation}
In particular, a strictly positive velocity can be selected analytically
for every fixed $N\ge3$.
\end{theorem}
\begin{proof}
The function $G_\epsilon=\bar x+\epsilon h_1+\epsilon^2h_2+\epsilon^3h_3$
satisfies, by cancellation of the first three perturbation terms,
\[
 \cL_\epsilon G_\epsilon=\mathfrak c_N\epsilon^3+\epsilon^4\cA h_3.
\]
The difference
\[
 G_\epsilon-x_N=(\bar x-x_N)
                   +\epsilon h_1+\epsilon^2h_2+\epsilon^3h_3
\]
is an integrable function of the hidden state. Its mean increment under
$\nu_{N,\epsilon}$ is zero. Consequently, for every $t>0$,
\[
 \E_{\nu_{N,\epsilon}}[G_\epsilon(X_t)-G_\epsilon(X_0)]
 =\E_{\nu_{N,\epsilon}}[X_N^{(1)}(t)-X_N^{(1)}(0)]
 =t v_{N,\epsilon}.
\]

To apply Dynkin's formula to $G_\epsilon$, stop when the relative
diameter exceeds $m$. The derivative bound~\eqref{eq:corrector-growth}
controls its Brownian quadratic variation by $C_{N,\epsilon}(1+\mathfrak D^2)$.
Each copying increment obeys the same squared bound, since
$\mathfrak D(C_{ij}X)\le\mathfrak D(X)$. The second moment
in~\eqref{eq:soft-moments} makes both martingales square integrable on
every bounded time interval under the stationary law. The same bounds
and the fixed-window Brownian and Poisson maximal moments allow the
stopping to be removed. Taking expectations in Dynkin's formula now
gives
\[
 t v_{N,\epsilon}
 =t\mathfrak c_N\epsilon^3
      +t\epsilon^4\nu_{N,\epsilon}(\cA h_3).
\]
This gives~\eqref{eq:exact-current}. Equations~\eqref{eq:corrector-growth}
and~\eqref{eq:explicit-moment} give the uniform choice
\[
 M_N=B_{3,N}(1+A_N/\beta_N),
\]
after enlarging $B_{3,N}$ if needed. The displayed population
factorization is strictly positive for every $N\ge3$.

Fix once and for all a contrast parameter
\[
 1<\Theta<\frac2{\sqrt3}.
\]
Choose
\begin{equation}\label{eq:epsilon-choice}
 \begin{gathered}
 0<\epsilon_N<\min\left\{\frac{\lambda_N}{4},
                  \frac{\lambda_N(\Theta-1)}{2(\Theta+1)},
                   \frac{\mathfrak c_N}{2(M_N+1)}\right\},\\
 k_N(x)=\lambda_N+\epsilon_N\phi_*(x^{(1)}).
 \end{gathered}
\end{equation}
The remainder bound gives
$v_{N,\epsilon_N}\ge\mathfrak c_N\epsilon_N^3/2>0$. The second amplitude
bound ensures, on writing $k_{\min}=\min_x k_N(x)$ and
$k_{\max}=\max_x k_N(x)$, that
\begin{equation}\label{eq:ratio}
 \frac{k_{\max}}{k_{\min}}
 \le\frac{\lambda_N+2\epsilon_N}{\lambda_N-2\epsilon_N}
 <\Theta.
\end{equation}
The rate $k_N$ supplies the positive current and the contrast margin
used in the geometric construction.
\end{proof}

\begin{corollary}[A quantitative three-particle window]\label{cor:quantitative-three-current}
For $N=3$ and $|\epsilon|\le1/12$,
\begin{equation}\label{eq:quantitative-three-current}
 \begin{gathered}
 \left|v_{3,\epsilon}-\mathfrak c_3\epsilon^3\right|
      \le9|\epsilon|^4,\qquad \mathfrak c_3=\frac{859}{1293600},\\
 0<\epsilon\le2^{-16}\quad\Longrightarrow\quad
 v_{3,\epsilon}>\tfrac12\mathfrak c_3\epsilon^3>0.
 \end{gathered}
\end{equation}
\end{corollary}
\begin{proof}
Appendix~\ref{app:quantitative-three-certificate} gives
$|\nu_{3,\epsilon}\cA h_3|<9$ in this range.
Apply~\eqref{eq:exact-current} and the rational inequality
$9\cdot2^{-16}<\mathfrak c_3/2$.
The contrast condition for the geometric construction remains
\eqref{eq:epsilon-choice}.
\end{proof}

\begin{corollary}[Stationary displacement and speed]\label{cor:soft-speed}
For the soft system with any fixed smooth positive periodic rate, write $\mathbf v=\nu_k j_k$. For every $T>0$,
\begin{equation}\label{eq:soft-stationary-increments}
 \E_{\nu_k}\Delta_T=T\mathbf v,\qquad
 \E_{\nu_k}(J_T-J_0)=T\mathbf v^{(1)}/L,
 \quad J_t=\lfloor X_N^{(1)}(t)/L\rfloor.
\end{equation}
For every finite $p\ge1$, $\E_{\nu_k}\mathfrak X_T^p<\infty$.
Under the hidden stationary law, the centre of mass satisfies
\begin{equation}\label{eq:soft-almost-sure-speed}
 \lim_{t\to\infty}\frac{\bar X_t-\bar X_0}{t}=\mathbf v
 \quad\text{almost surely},\qquad
 \bar X_t=N^{-1}\sum_iX_i(t).
\end{equation}
\end{corollary}
\begin{proof}
The integrable differences $\bar X-X_N$ and $X_N^{(1)}-LJ$ are stationary
hidden functions. Dynkin's equation for $\bar X$ gives the first identity;
cancellation of the fractional phase gives the second. With candidate
count $K_T$ and $U_T=\max_i\sup_{s\le T}|W_i(s)|$, the ancestral envelope gives
\[
 \mathfrak X_T\le2\mathfrak D(X_0)+6(K_T+1)U_T,\qquad
 \E_{\nu_k}\mathfrak X_T^p\le C_{p,N,k,T}<\infty,
\]
by Poisson and Brownian maximal moments and
Lemma~\ref{lem:soft-hidden-moments}.

Write $\bar X_t-\bar X_0=\int_0^t j_k(H_s)\,ds+M_t$. The hidden process is ergodic by its global minorization, so the time average of the integrable drift converges to $\nu_kj_k$. The expected quadratic variation rate of each component of $M$ is bounded by $1/N+C_Nk_{\max}\nu_k\mathfrak D^2$. Hence $\int_0^\infty(1+s)^{-1}\,dM_s$ has finite expected quadratic variation and converges almost surely. The martingale version of Kronecker's lemma gives $M_t/t\to0$, proving~\eqref{eq:soft-almost-sure-speed}.
\end{proof}

\subsection{The chiral response and continuity of the current}
For the two-harmonic profiles of~\eqref{eq:chiral-invariant}, the same
finite recurrence gives
\begin{equation}\label{eq:general-chirality}
 v_{N,\epsilon,\phi}
 =\mathfrak c_N\mathcal C(\phi)\epsilon^3
 +O_{N,\phi}(\epsilon^4).
\end{equation}
Here $v_{N,\epsilon,\phi}$ is the longitudinal stationary velocity for
$k(x)=\lambda_N+\epsilon\phi(x^{(1)})$, with $|\epsilon|$ small enough
that the rate lies in the fixed bounded class of
Lemma~\ref{lem:soft-hidden-moments}. The first two response terms vanish
for every such profile: the frequency selection and paired cancellation
in Lemma~\ref{lem:cubic} apply before its Fourier weights are chosen.

At order three, put
$z=\widehat\phi_1^2\overline{\widehat\phi_2}$.
The words $(1,1,-2)$ and their permutations have common weight $z$,
while their negatives have weight $\overline z$. Using
$W_N=4\mathfrak c_N$ and~\eqref{eq:word-reflection}, their combined
root derivative is
\[
 W_Nz-W_N\overline z
 =8i\mathfrak c_N\operatorname{Im}z
 =i\mathfrak c_N\mathcal C(\phi).
\]
Differentiation gives the corresponding corrector identities, and the
growth and stationary moment bounds give a fourth-order remainder as in
Theorem~\ref{thm:copying-current}. This proves~\eqref{eq:general-chirality}.
The population enters through the positive factor $\mathfrak c_N$;
the sign of $\mathcal C(\phi)$ determines the sign of the cubic coefficient.

\begin{corollary}[Continuity and positive rates]\label{cor:soft-open-family}
For fixed $N$, velocity is continuous in the uniform norm on a neighbourhood of $k_N$. In particular, the smooth positive periodic rates with positive longitudinal velocity contain an open neighbourhood of $k_N$.
\end{corollary}
\begin{proof}
Work in $\lambda_N/2\le k,\widetilde k\le3\lambda_N/2$ and put
$d=\|k-\widetilde k\|_\infty\le1$. Lemma~\ref{lem:soft-hidden-moments}
gives common minorization and moment constants. With the same drivers
and candidate marks, the paths agree until a discrepant acceptance,
whose probability before $T_N$ is at most $NT_Nd$. For the block
kernels $P_k,P_{\widetilde k}$,
\[
 \begin{aligned}
 \sup_h\|P_k(h,\cdot)-P_{\widetilde k}(h,\cdot)\|_{\TV}
       &\le2NT_Nd,\\
 \|\nu_k-\nu_{\widetilde k}\|_{\TV}
       &\le\eta_N^{-1}\sup_h
            \|P_k(h,\cdot)-P_{\widetilde k}(h,\cdot)\|_{\TV}
        \le C_Nd.
 \end{aligned}
\]
Corollary~\ref{cor:soft-speed} identifies
$v_N(k)=\nu_kj_k^{(1)}$. By~\eqref{eq:soft-com-drift},
$|j_k|\le C_N\mathfrak D$ and
$|j_k-j_{\widetilde k}|\le C_Nd\mathfrak D$. For $R\ge1$,
\[
 \begin{aligned}
 |v_N(k)-v_N(\widetilde k)|
 &\le \nu_k|j_k-j_{\widetilde k}|
       +\bigl| (\nu_k-\nu_{\widetilde k})
               (j_{\widetilde k}^{(1)}\one_{\{\mathfrak D\le R\}})\bigr|\\
 &\quad +(\nu_k+\nu_{\widetilde k})
              (|j_{\widetilde k}|\one_{\{\mathfrak D>R\}})\\
 &\le C_N(d+Rd+R^{-1}).
 \end{aligned}
\]
Taking $R=d^{-1/2}$ for $d>0$ gives
$|v_N(k)-v_N(\widetilde k)|\le C_N(d+\sqrt d)$.
Positivity at $k_N$ and its strict contrast margin survive in a sufficiently small neighbourhood.
\end{proof}

\section{Periodic perforations and moment bounds}\label{sec:fixed}
Before inversion, the particles move in periodic arrays of spherical
holes. A periodic perforated domain is
\begin{equation}\label{eq:template}
 D=\R^3\setminus\bigcup_{k\in\Z^3}\bigcup_{a=1}^J
             \overline B(z_a+Lk,r_a),\qquad J\in\{1,2,\ldots\},
\end{equation}
with positive radii and a positive minimum separation between distinct
closed balls. It is connected: a compact polygonal path meets finitely
many balls, each of which can be bypassed in a disjoint exterior collar.
The separation also gives the boundary a uniform tubular neighborhood.

We also consider auxiliary systems
\begin{equation}\label{eq:perturbation}
 dX_i=\sigma_i(X)\,dW_i+b_i(X)\,du,\qquad
 \sup_x\max_i(\|\sigma_i(x)-I\|+|b_i(x)|)\le\eta.
\end{equation}
The $W_i$ are independent. Coefficients are smooth on all of $\R^{3N}$,
can depend on the entire configuration, and need not be periodic.
Estimate constants in this section depend on $N,D$ and the displayed time or
moment parameters. They are uniform for sufficiently small $\eta$,
independently of coefficient derivatives; smoothness is used locally.
The periodic baseline induces a Markov hidden process; conditional
estimates for a nonperiodic perturbation use the full real configuration.

\subsection{Death-count estimates from the boundary distance}
A collective boundary-distance function controls wall clustering and
the entire death count. Its continuous drift is negative near simultaneous
boundary contact, and uniform copying contracts it in conditional
expectation.
Choose a smooth periodic $\rho:\overline D\to[0,R_0]$, equal to boundary
distance in a fixed collar, positive in the interior, and saturated away
from the collar. Write $\rho_i=\rho(x_i)$ for a configuration $x$, and put
\begin{equation}\label{eq:Phi}
 \gamma=(N-2)/16,\quad S=\sum_i\rho_i^2,\quad
 \Phi=(NR_0^2/S)^{\gamma/2}\ge1.
\end{equation}
These functions are invariant under common lattice translations, so
we also write $\Phi(H(x))=\Phi(x)$.
For $s=1,2,4$ let $\beta=s\gamma$ and $\varphi=\Phi^s$.
Write $a_i=\sigma_i\sigma_i^{\mathsf T}$ for the covariance matrix, and set
\[
 \begin{aligned}
 g_i&=\nabla\rho(x_i)^{\mathsf T}a_i\nabla\rho(x_i),\\
 \cL_i\rho&=\tfrac12 a_i:D^2\rho(x_i)+b_i\cdot\nabla\rho(x_i).
 \end{aligned}
\]
The continuous generator is
\begin{align}\label{eq:Phi-generator}
 \cL\varphi={}&-\frac\beta2\frac\varphi S\sum_i g_i
 -\beta\frac\varphi S\sum_i\rho_i\cL_i\rho
 +\frac{\beta(\beta+2)}2\frac\varphi{S^2}\sum_i\rho_i^2g_i.
\end{align}
Cross-label quadratic variations vanish. For small $S$, every label is
in the distance collar: $|\nabla\rho(x_i)|=1$, $|\cL_i\rho|\le C$.
If $a_-I\le a_i\le a_+I$, then
\[
 \cL\varphi\le-\frac\beta2\frac\varphi S
 [Na_--(\beta+2)a_+]+C\varphi/\sqrt S.
\]
Choose $\eta$ small enough that
\[
 4\gamma+2=(N+6)/4<N,\qquad
 Na_--(4\gamma+2)a_+\ge c_*>0.
\]
Absorb $C\varphi/\sqrt S$ for small $S$; on the complement the first
two derivatives are bounded. Using
$\varphi/S=(NR_0^2)^{-1}\varphi^{1+2/\beta}$ gives
\begin{equation}\label{eq:Phi-drift}
 \cL\Phi^s\le C_s-c_s(\Phi^s)^{1+2/(s\gamma)},\qquad s=1,2,4.
\end{equation}

At death $\tau$ of label $i$, put $z_j=\rho_j^2/S$. The jump satisfies
\[
 \rho_i=0,\quad \sum_{j\ne i}z_j=1,\qquad
 \frac{(\Phi^s)^+}{(\Phi^s)^-}=(1+z_j)^{-s\gamma/2}\le1.
\]
Average the convex chord bound
$(1+z)^{-\gamma/2}\le1-(1-2^{-\gamma/2})z$, $0\le z\le1$, over
the uniform donor:
\begin{equation}\label{eq:copy-contraction}
 \E(\Phi^+\mid\mathcal F_{\tau-})\le q\Phi^-,
 \qquad q=1-\frac{1-2^{-\gamma/2}}{N-1}<1.
\end{equation}

\begin{lemma}\label{lem:all-death-budget}
For a fixed domain and all sufficiently small perturbations, the recursive
system is nonexplosive. There are uniform constants $\alpha,C>0$ such
that, if $K_t$ is its number of deaths, then, for every $p>0$,
\begin{align}
 \E_x e^{\alpha K_t}&\le e^{Ct}\Phi(x),\label{eq:count}\\
 \E_x(K_t+1)^p&\le C_{p,t}(1+\log\Phi(x))^p,\label{eq:count-moment}\\
 \sup_x\E_x\Phi(H_t)^s&\le C_s(1+t^{-s\gamma/2}),
       \qquad t>0,\ s=1,2,4.\label{eq:landing}
\end{align}
\end{lemma}
\begin{proof}
Lemma~\ref{lem:simultaneous-wall-polarity} first gives unique victims for every finite
continuous segment. Choose $\alpha>0$ with $e^\alpha q\le1$.
Equations~\eqref{eq:Phi-drift} and \eqref{eq:copy-contraction} make
$e^{\alpha K_t-Ct}\Phi(H_t)$ a localized supermartingale. Write
\[
 \begin{aligned}
 \tau_m&=\inf\{t:K_t\ge m\},\qquad
 \sigma_R=\inf\{t:\Phi(H_t)\ge R\},\\
 \chi_\ell&=\inf\{t:\max_i|X_i(t)|\ge\ell\}.
 \end{aligned}
\]
For $R>\Phi(x)$ and $\ell>\max_i|x_i|$, optional stopping gives
\[
 \begin{aligned}
 \E_x\!\left[e^{\alpha K_\theta-C\theta}\Phi(H_\theta)\right]
       &\le\Phi(x),\qquad
       \theta=T\wedge\tau_m\wedge\sigma_R\wedge\chi_\ell,\\
 \Prob_x(\tau_m\le T\wedge\sigma_R\wedge\chi_\ell)
       &\le e^{CT}\Phi(x)e^{-\alpha m},\\
 \Prob_x(\sigma_R\le T\wedge\tau_m\wedge\chi_\ell)
       &\le e^{CT}\Phi(x)/R.
 \end{aligned}
\]
At fixed $m$, ancestral paths contain finitely many continuous
martingale pieces with bounded quadratic variation and bounded drift.
Their positions remain finite on $[0,T\wedge\tau_m]$, so first
$\ell\to\infty$ removes the spatial stop. Letting $R\to\infty$ then gives
\[
 \Prob_x(\tau_m\le T)\le e^{CT}\Phi(x)e^{-\alpha m}
       \longrightarrow0\quad(m\to\infty).
\]
Thus $\tau_m\uparrow\infty$ almost surely. Fatou's lemma for the stopped
supermartingale yields~\eqref{eq:count}. For the moment bound, put
$u_0=(Ct+\log\Phi(x))/\alpha$. The exponential tail gives
\[
 \begin{aligned}
 \E_x(K_t+1)^p
 &=1+p\int_0^\infty(u+1)^{p-1}\Prob_x(K_t>u)\,du\\
 &\le(1+u_0)^p
       +p\int_{u_0}^\infty(u+1)^{p-1}e^{-\alpha(u-u_0)}\,du\\
 &\le C_{p,t}(1+\log\Phi(x))^p.
 \end{aligned}
\]

Put $m_s(t)=\E_x\Phi(H_t)^s$. Localized It\^o and the downward copy jumps give
\[
 \E_x\Phi(H_t)^s+c_s\int_0^t
   \E_x(\Phi(H_r)^s)^{1+2/(s\gamma)}\,dr
 \le \Phi(x)^s+C_st.
\]
Fatou, the full-configuration Markov property and Jensen's inequality give
\[
 m_s(b)-m_s(a)\le C_s(b-a)-c_s
             \int_a^b m_s(r)^{1+2/(s\gamma)}\,dr.
\]
Set $p_s=2/(s\gamma)$ and
\[
 A_s=(2C_s/c_s)^{1/(1+p_s)},\quad
 B_s=(2/(c_sp_s))^{1/p_s},\quad
 y_s(t)=A_s+B_st^{-1/p_s}.
\]
For all sufficiently small $\varepsilon>0$,
\[
 y_s'\ge C_s-c_sy_s^{1+p_s},\qquad
 m_s(\varepsilon)\le\Phi(x)^s+C_s\varepsilon\le y_s(\varepsilon).
\]
The integral inequality permits only downward singular variation;
positive-part comparison on $[\varepsilon,T]$ gives, with
$L_{\varepsilon,T}<\infty$,
\[
 (m_s(t)-y_s(t))^+
 \le L_{\varepsilon,T}\int_\varepsilon^t
                   (m_s(r)-y_s(r))^+\,dr.
\]
Gronwall and $\varepsilon\downarrow0$ give $m_s\le y_s$, proving
\eqref{eq:landing}.
\end{proof}

\subsection{Common ancestry from arbitrary configurations}
A protected survivor recruits every other label through that label's
first copying event. Periodicity gives a uniform chance that each original
free path enters a hole; copying places that label at the protected
particle's actual Euclidean position. This construction removes the
cloud's initial lattice separation.

\begin{lemma}\label{lem:common-ancestor}
There are $T_*>0$, $\beta_*>0$ and an interior compact hidden set $C_*$ such
that, uniformly over all starting real configurations and small
perturbations,
\begin{equation}\label{eq:fixed-ancestor}
 \Prob_x\{H_{T_*}\in C_*,\ \mathsf A_{0,T_*}\}
       \ge\beta_*.
\end{equation}
\end{lemma}
\begin{proof}
Choose $R>4\sup_x\E_x\Phi(H_1)$, which is finite by~\eqref{eq:landing}.
Then
\[
 \Prob_x\{\Phi(H_1)\le R\}\ge3/4,
 \qquad \Phi\le R\ \Longrightarrow
 \max_i\rho_i\ge R_0R^{-1/\gamma}.
\]
By periodicity, choose $d_*>0$ with
$\rho(y)\ge R_0R^{-1/\gamma}\Rightarrow\dist(y,\partial D)\ge d_*$,
and select a protected label $p$ measurably. Fix a phase ball $Q$ strictly
inside the projection of one hole and an interval of length $T_0$. Let $B_i$ be the
original free Brownian primitives from time $1$, and $B_i'$ the fresh
primitives after their first copies. With suprema over $[0,T_0]$, require
\[
 \|B_p\|_\infty\le d_*/16,\qquad
 \bigcap_{i\ne p}\left\{
    \operatorname{ph}(X_i(1)+B_i(T_0))\in Q,\quad
    \|B_i'\|_\infty\le d_*/16
 \right\},
\]
and prescribe $p$ as the first donor mark of each label $i\ne p$.
These independent primitive requirements
define $E_{\rm prim}$. Heat-kernel positivity, Brownian small-ball
probabilities and uniform donors give
\[
 \begin{split}
 \Prob(E_{\rm prim}\mid\mathcal F_1)
 &\ge (N-1)^{-(N-1)}
       \Prob(\|B\|_\infty\le d_*/16)^N\\
 &\quad\times\left[
    \inf_{z\in M_3}\Prob\{\operatorname{ph}(z+B(T_0))\in Q\}
               \right]^{N-1}
 =:p_*>0,
 \end{split}
\]
where $B$ is standard Brownian motion; only the phases of $X_i(1)$ enter.

Let $E_{\rm err}$ bound all errors between actual and free paths by
$\varepsilon_*$ on the first $N$ continuous segments. BDG gives
\[
 \eta T_0\le\varepsilon_*/2
 \quad\Longrightarrow\quad
 \Prob(E_{\rm err}^c\mid\mathcal F_1)
       \le C_N\eta^2T_0/\varepsilon_*^2.
\]
Take $\varepsilon_*$ smaller than both $d_*/16$ and half the prescribed
depth inside the hole. On $E_{\rm prim}\cap E_{\rm err}$, the protected
label remains interior. Every other original path must hit a wall,
because its free endpoint is inside a hole; after copying $p$, its
bounded fresh primitive keeps it interior. Hence, with $T_*=1+T_0$,
\[
 E_{\rm prim}\cap E_{\rm err}\subset
 \{K_{T_*}-K_1=N-1,\ \mathsf A_{0,T_*},\ H_{T_*}\in C_*\},
\]
for a compact interior hidden set $C_*$. Choose $\eta$ with
$C_N\eta^2T_0/\varepsilon_*^2\le p_*/2$. Then
\[
 \Prob(E_{\rm prim}\cap E_{\rm err}\mid\mathcal F_1)
 \ge p_*-C_N\eta^2T_0/\varepsilon_*^2\ge p_*/2.
\]
Multiplying by the landing probability gives
\eqref{eq:fixed-ancestor} with $\beta_*=3p_*/8$.
\end{proof}

\subsection{Foster estimates and stationary laws}
Common ancestry contracts the initial width; death counts control the
noise accumulated along ancestral pieces.
Define
\[
 M_i(t)=\int_0^t\sigma_i(X_s)\,dW_i(s),\qquad U_t=\max_i\sup_{s\le t}|M_i(s)|.
\]
For every terminal ancestral path,
\[
 X_i(t)=X_{\mathfrak a_{0,t}(i)}(0)+\xi_i(t),\qquad
 |\xi_i(t)|\le2(K_t+1)U_t+B t,
\]
where $B$ bounds the drift. Consequently
\begin{equation}\label{eq:ancestor-envelope}
 \mathfrak D(X_t)\le \mathfrak D(X_0)\one_{\mathsf A_{0,t}^c}
                         +4(K_t+1)U_t+2Bt
\end{equation}
For $T\ge T_*$, $\Prob_x(\mathsf A_{0,T})\ge\beta_*$.
Choose $\delta>0$ with $(1+\delta)(1-\beta_*)<1$. Hölder,
\eqref{eq:count-moment} and BDG give, for $p=2,4$,
\[
 \begin{aligned}
 \E_x\!\left[(4(K_T+1)U_T+2BT)^p\right]
 &\le C_p\bigl(\E_x(K_T+1)^{2p}\bigr)^{1/2}
            \bigl(\E_xU_T^{2p}\bigr)^{1/2}+C_{p,T}\\
 &\le C_{p,T}(1+\log\Phi(x))^p,\\
 \E_x\mathfrak D(X_T)^p
 &\le(1+\delta)(1-\beta_*)\mathfrak D(x)^p
       +C_{p,T,\delta}(1+\log\Phi(x))^p.
 \end{aligned}
\]
Put $\rho_p=(1+\delta)(1-\beta_*)<1$.
\begin{equation}\label{eq:Foster-functions}
 \mathcal V=1+\mathfrak D^2+\Phi,\qquad \mathcal W=1+\mathfrak D^4+\Phi^2,\qquad
 \mathcal V^2\le3\mathcal W.
\end{equation}
For $p=2,4$, the landing estimate and the bound
$(1+\log r)^p\le\varepsilon r^{p/2}+C_{p,\varepsilon}$ give, for any
fixed $0<\zeta<1$ after choosing $\varepsilon$ small enough,
\[
 \E_x[1+\mathfrak D(X_T)^p+\Phi(H_T)^{p/2}]
 \le\rho_p\mathfrak D(x)^p+\zeta\Phi(x)^{p/2}+C_{p,T,\zeta}.
\]
Taking $\rho_V=\max(\rho_2,\zeta)$ and
$\rho_W=\max(\rho_4,\zeta)$ yields
\begin{equation}\label{eq:Foster}
 \E_x\mathcal V(H_T)\le\rho_V\mathcal V(x)+C_V,\qquad
 \E_x\mathcal W(H_T)\le\rho_W\mathcal W(x)+C_W,\quad \rho_V,\rho_W<1.
\end{equation}
The same ancestral estimate within the block gives
\[
 \begin{aligned}
 \mathfrak X_T&\le2\mathfrak D(x)+6(K_T+1)U_T+3BT,\\
 \E_x\mathfrak X_T^2
 &\le C_T[1+\mathfrak D(x)^2+(1+\log\Phi(x))^2]
 \le C_T\mathcal V(x).
 \end{aligned}
\]
This proves
\begin{equation}\label{eq:block-envelope}
 \E_x\mathfrak X_T^2\le C\mathcal V(x).
\end{equation}

\begin{proposition}\label{prop:template-foster}
For some $T>0$, $\beta>0$ and probability law $\mu$, the baseline hidden
block kernel $P=P_T$ satisfies
$P(\xi,\cdot)\ge\beta\mu(\cdot)$ for every hidden state $\xi$. It has a unique invariant law $\nu$, with
$\nu\mathcal W<\infty$. The invariant law is also invariant for continuous
time, and its reference displacement satisfies
$\E_\nu\Delta_t=t\mathbf v$ for some $\mathbf v\in\R^3$.
The estimates~\eqref{eq:Foster} and~\eqref{eq:block-envelope} hold uniformly for the small perturbation class.
\end{proposition}
\begin{proof}
After the safe common-ancestor cloud in Lemma~\ref{lem:common-ancestor}, add a
fixed quiet interval of length $\delta>0$, and take $T=T_*+\delta$.
A common lattice translation represents every configuration in $C_*$
by a point of a compact interior set $K\subset D^N$ in finitely many
cells. Choose a ball $B$ with $\overline B\subset D$. The killed Brownian density
$p_\delta^D$ is positive and continuous on the connected domain, so
\[
 a_0:=\inf_{x\in K,\ y\in\overline B^N}
             \prod_{i=1}^Np_\delta^D(x_i,y_i)>0.
\]
Let $\mu$ be the hidden image of normalized volume on $B^N$.
The probability $\beta_*$ of reaching $C_*$ followed by the product
killed kernel gives
\[
 P_T(\xi,A)\ge\beta_*a_0|B|^N\mu(A)
 \quad\text{for every hidden state }\xi\text{ and measurable }A.
\]
This minorization includes the real relative coordinates. Its
contraction gives a unique invariant law. Iterating~\eqref{eq:Foster}
and applying truncated Fatou along $P_{nT}(\xi,\cdot)\to\nu$ gives
\[
 \E_\xi\mathcal W(H_{nT})
 \le\rho_W^n\mathcal W(\xi)+\frac{C_W}{1-\rho_W},
 \qquad \nu\mathcal W\le\frac{C_W}{1-\rho_W}.
\]
For $s\ge0$, commutation gives
\[
 (\nu P_s)P_T=(\nu P_T)P_s=\nu P_s.
\]
Uniqueness of the block invariant law therefore implies $\nu P_s=\nu$.

Set $a(t)=\E_\nu\Delta_t$. The conditional law of future
increments depends only on the hidden state, so $a(s+t)=a(s)+a(t)$.
The ancestral block envelope gives
$\sup_{0\le t\le T}|a(t)|\le(C\nu\mathcal V)^{1/2}<\infty$.
Additivity and this local bound imply $a(t)=t\mathbf v$.
\end{proof}

\section{Genealogical homogenization and the signed current}\label{sec:homogenization}
Fix $L=2\pi$ and a smooth positive killing rate $k$ on the torus
$M=(\R/L\Z)^3$, with $0<k_{\min}\le k\le k_{\max}<\infty$.
The main construction uses the rate $k_N$ selected in~\eqref{eq:epsilon-choice}.
Take $h=L/m$ and place a ball at each center $z$
of a mesh cube $Q_z$ with
\begin{equation}\label{eq:holes}
 a_z=\frac{h^3k(z)}{2\pi},\qquad
 M_h=M\setminus\bigcup_z\overline B(z,a_z).
\end{equation}
Extend the radii periodically by $a_{z+L\ell}=a_z$, $\ell\in\Z^3$.
For small $h$ the concentric balls of radius $h^2$ are
disjoint. Write $D_h$ for the periodic Euclidean lift. The hidden invariant
law $\nu_h$ exists for each fixed $h$ by Proposition~\ref{prop:template-foster};
write $\pi_h$ for its phase projection. The radius in~\eqref{eq:holes} makes
the Brownian capacity of one ball equal to $h^3k(z)$. Thus absorption by
the holes approaches killing at rate $k$. To recover the signed current,
we follow the genealogy of Euclidean lifts. On a sufficiently fine grid,
principal phase increments recover the continuous Brownian increments.
An initial lattice offset follows its ancestral label through every copy.
Once the backward genealogy has one ancestor, all reconstructed positions
have the same lattice error. This common translation cancels from the
real relative positions and the reference displacement.
Figure~\ref{fig:capacity-marked-transfer} summarizes this passage from
capacity to the signed current.
\begin{figure}[htbp]
\centering
\includegraphics{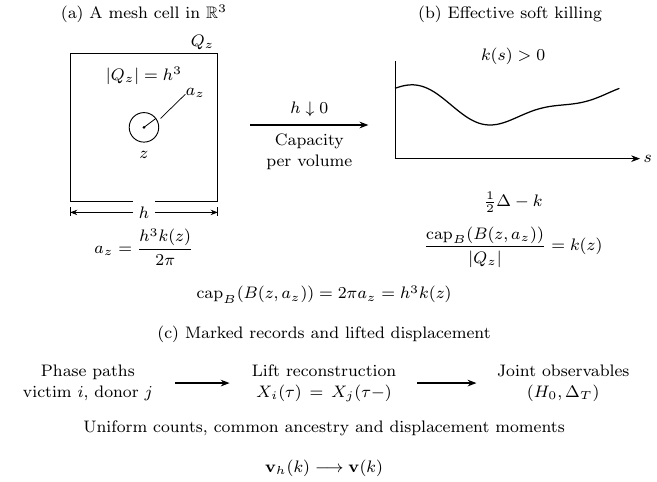}
\caption{Capacity scaling and marked transfer. Panel (a) is a section
of a cube of volume $h^3$, with an enlarged spherical hole. Brownian
capacity $\capB$ for $\tfrac12\Delta$ satisfies
$\capB(B(z,a_z))=2\pi a_z=h^3k(z)$, giving the soft killing rate in (b).
In (c), victim and donor marks guide reconstruction of the joint
observables $(H_0,\Delta_T)$ from phase records and the passage to the
stationary signed current}
\label{fig:capacity-marked-transfer}
\end{figure}

\begin{theorem}[Marked transfer principle]\label{thm:marked-transfer}
Fix $N\ge3$ and a smooth positive periodic rate $k$. Assume the uniform stationary death-count and common-ancestor estimates
\begin{equation}\label{eq:transfer-inputs}
 \begin{gathered}
 \E_{\pi_h}e^{\alpha K_t}\le Ce^{Ct},\\
 \sup_{x\in D_h^N}\Prob_x^h(\mathsf A_{0,nT_*}^{c})
 \le(1-\beta)^n.
 \end{gathered}
\end{equation}
for $t\ge0$, $n\ge1$ and all sufficiently small $h$, with
$\alpha,T_*>0$, $\beta\in(0,1)$ and $C<\infty$ independent of $h$.
Then the stationary vector currents satisfy
\begin{equation}\label{eq:velocity-transfer}
 \mathbf v_h(k)\longrightarrow\mathbf v(k).
\end{equation}
More precisely, the stationary phase laws satisfy $\|\pi_h-\pi_k\|_{\TV}\to0$. For every fixed $T>0$,
\begin{equation}\label{eq:stationary-target-transfer}
 \|\Law_{\nu_h}(H_0,\Delta_T)-\Law_{\nu_k}(H_0,\Delta_T)\|_{\TV}\to0.
\end{equation}
The complete stationary continuous record on $[0,T]$, jointly with
the initial hidden state and the real relative particle paths, converges
weakly. There is $h_0>0$ such that, for every $p>0$ and fixed $T>0$,
\[
 \sup_{h<h_0}\left(\nu_h\mathfrak D^p
                 +\E_{\nu_h}|\Delta_T|^p\right)<\infty.
\]

For every partition with $r$ roots and probability densities on $M_h^r$
whose zero extensions satisfy $f_h\to f$ strongly in $L^2(M^r)$,
the compressed records on any fixed finite grid converge in total
variation. Complete continuous
records on every fixed time window, and the first $m$ complete copying
records for every finite $m$, converge weakly.
\end{theorem}

\begin{corollary}[Transfer at small contrast]\label{cor:small-contrast-transfer}
Fix $N\ge3$ and $1<\Theta<2/\sqrt3$. For every smooth positive periodic
rate satisfying $k_{\max}/k_{\min}\le\Theta$, all conclusions of
Theorem~\ref{thm:marked-transfer} hold for the holes~\eqref{eq:holes}.
\end{corollary}
\begin{proof}
The barrier estimates in Section~\ref{sec:uniform-small-contrast} verify
both bounds in~\eqref{eq:transfer-inputs}, uniformly as $h\downarrow0$.
Theorem~\ref{thm:marked-transfer} applies.
\end{proof}

\begin{corollary}[Arbitrary rates for three particles]\label{cor:three-particle-transfer}
For $N=3$ and every smooth positive periodic rate $k$, all conclusions
of Theorem~\ref{thm:marked-transfer} hold for the holes~\eqref{eq:holes}.
If $\mathbf v(k)\ne0$, a signed permutation of the coordinate axes
and a sufficiently fine mesh give a perforated domain with positive
longitudinal current.
\end{corollary}
\begin{proof}
Lemmas~\ref{ah:corner} and~\ref{ah:coalescence} in
Appendix~\ref{app:three-particle-rates} verify~\eqref{eq:transfer-inputs}.
Choose a nonzero velocity component and orient its axis positively.
The vector limit preserves its sign; the signed permutation preserves
the cubic lattice.
\end{proof}

Finite-record comparison uses the one-particle kernels before the uniform
count and ancestry estimates enter. Those estimates then control
stationary lifts and the reconstruction error. This order also permits
finite-record comparison to establish common ancestry for arbitrary
three-particle rates.

\subsection{Forms, kernels and death densities}
Zero extension gives the limiting energy
\begin{equation}\label{eq:effective-forms}
 \mathcal E_k(u)=\tfrac12\int_M|\nabla u|^2+\int_Mk|u|^2,
 \qquad u\in H^1(M).
\end{equation}
The shell energy is $2\pi a_z/(1-a_z/h^2)=h^3k(z)(1+O(h))$.
The shells occupy vanishing volume and contribute the nonzero absorption
energy $\int_Mk|u|^2$. The energy outside them retains the full gradient
term. Appendix~\ref{app:capacity} proves the form limit and the associated
operator convergence.

Let $A_h=-\frac12\Delta$ on $L^2(M_h)$ with Dirichlet boundary conditions,
$A_k=-\frac12\Delta+k$, and let $J_h$ be zero extension. Put
$R_h=J_h(A_h+1)^{-1}J_h^*$ and $R_k=(A_k+1)^{-1}$.
The capacity calculation gives $\|R_h-R_k\|_{2\to2}\to0$.
Continuous functional calculus yields
\begin{equation}\label{eq:operator-limit}
 \|P_h(t)-P_k(t)\|_{2\to2}\to0,\qquad
 P_h(t)=J_he^{-tA_h}J_h^*,\quad P_k(t)=e^{-tA_k}.
\end{equation}
It also gives the time-derivative limit uniformly on each
$[\delta,T]$, $0<\delta<T$:
\[
 \sup_{\delta\le t\le T}
 \bigl(\|P_h(t)-P_k(t)\|_{2\to2}
       +\|\partial_tP_h(t)-\partial_tP_k(t)\|_{2\to2}\bigr)\to0.
\]

The kernels $p_h,p_k$ are dominated by the free torus kernel, giving
uniform Hilbert--Schmidt norms at positive times. With
$B_h=P_h(t/2)$ and $B_k=P_k(t/2)$,
\[
 \|B_h^2-B_k^2\|_{\rm HS}
 \le \|B_h-B_k\|_{2\to2}
       (\|B_h\|_{\rm HS}+\|B_k\|_{\rm HS}).
\]
This upgrades the operator limit to kernel $L^2(M^2)$ convergence.
For $\star\in\{h,k\}$, let $S_\star(t)=P_\star(t)1$ be the
one-particle survival function and $q_\star=-\partial_tS_\star$ its
death density. Then
\begin{equation}\label{eq:kernel-limit}
 \sup_{\delta\le t\le T}
 \big(\|p_h(t)-p_k(t)\|_{L^2(M^2)}
       +\|q_h(t)-q_k(t)\|_{L^2(M)}\big)=:a_h(\delta)\longrightarrow0.
\end{equation}
The derivative limit applied to $1$ gives the $q$ term. Moreover,
\[
 \begin{aligned}
 q_h(t)&=J_he^{-tA_h/2}A_he^{-tA_h/2}J_h^*1,\\
 \|q_h(t)\|_\infty
 &\le \|J_he^{-tA_h/2}\|_{2\to\infty}
       \|A_he^{-tA_h/2}\|_{2\to2}|M|^{1/2}\le C_\delta
       \quad(t\ge\delta).
 \end{aligned}
\]
Here heat-kernel domination bounds the first norm and the spectral
estimate $\|A_he^{-tA_h/2}\|_{2\to2}\le C/t$ bounds the second.
For the soft model,
\begin{equation}\label{eq:soft-density}
 q_k(t)=P_k(t)k,\qquad 0\le q_k(t,x)\le k_{\max},\quad t\ge0.
\end{equation}

\subsection{Uniform counts and common ancestry at small contrast}\label{sec:uniform-small-contrast}
The radial capacity profile also provides a boundary barrier whose
death-count constants are uniform in the shrinking-hole scale.
Its gradient weights decrease as its squared values increase.
Summing over labels exploits this opposite ordering and leaves a
coercive margin determined by the contrast of the radii.

Assume $k_{\max}/k_{\min}\le\Theta$, with $1<\Theta<2/\sqrt3$, and put
\begin{equation}\label{eq:contrast-margin}
 \delta_\Theta=1-\frac34\Theta^2>0.
\end{equation}
Choose a smooth cutoff $\chi$ taking values in $[0,1]$, equal to one on $[0,1/8]$ and zero on
$[1/4,\infty)$. On each disjoint annulus $a_z<|x-z|<h/4$ set
\[
 u_h(x)=1-\frac{a_z}{|x-z|}\chi(|x-z|/h),
\]
and set $u_h=1$ outside the radius-$h/4$ balls. This function vanishes
at the wall and lies in $(0,1]$ in the interior. In each core
$a_z<|x-z|\le h/8$,
\begin{equation}\label{eq:capacity-barrier}
 \Delta u_h=0,\qquad |\nabla u_h|^2=a_z^{-2}(1-u_h)^4.
\end{equation}
Outside the core, $u_h\ge1/2$, $|\nabla u_h|\le Ch$, and
$|\Delta u_h|\le C$, uniformly in small $h$. Thus
$\frac12\Delta u_h\ge-c_0u_h$. Let $\tau_h$ denote the exit time of
one Brownian particle from $D_h$. Stopped It\^o gives
\begin{equation}\label{eq:survival-lower}
 \Prob_x(\tau_h>T)\ge e^{-c_0T}u_h(x).
\end{equation}

Write $u_i=u_h(X_i)$, put $\mathsf S=\sum_i u_i^2$, and
$\Phi_h=(N/\mathsf S)^{\gamma/2}$, with $\gamma$ as in \eqref{eq:Phi}.
Formula~\eqref{eq:Phi-generator} applies with $u_h$ replacing $\rho$.
Let $\beta_s=s\gamma$, $s=1,2,4$, and $w_i=(1-u_i)^4$. Opposite monotonicity gives
\[
 N\sum_i u_i^2w_i-\Big(\sum_i u_i^2\Big)\Big(\sum_iw_i\Big)
       =\sum_{i<j}(u_i^2-u_j^2)(w_i-w_j)\le0.
\]
Put $a_{\min}=\min_z a_z$, $a_{\max}=\max_z a_z$ and
$R_a=a_{\max}/a_{\min}$. If all coordinates are in cores, then
\begin{equation}\label{eq:gradient-sort}
 \frac{\beta_s+2}{\mathsf S}\sum_i u_i^2|\nabla u_i|^2
       \le\frac{(\beta_s+2)R_a^2}{N}\sum_i|\nabla u_i|^2.
\end{equation}
Here $R_a\le k_{\max}/k_{\min}\le\Theta$. Since
$\gamma=(N-2)/16$ and $s\le4$,
\begin{equation}\label{eq:uniform-coercivity}
 \frac{\beta_s+2}{N}R_a^2
 \le\frac{N+6}{4N}\Theta^2
 \le\frac34\Theta^2=1-\delta_\Theta.
\end{equation}
Substitution into~\eqref{eq:Phi-generator} gives, in the cores,
\[
 \cL\Phi_h^s\le
 -\frac{\beta_s\delta_\Theta}{2}
       \frac{\Phi_h^s}{\mathsf S}\sum_i|\nabla u_i|^2.
\]
For mesh-cell centers $z_i$, set $g_i=a_{z_i}^{-2}(1-u_i)^4$.
The core identity and cutoff estimates give
\[
 \begin{aligned}
 X_i\text{ outside its core}
 &\ \Longrightarrow\ \mathsf S\ge\tfrac14,\quad
        |\nabla u_i|^2+g_i\le Ch^2,\quad |\Delta u_i|\le C,\\
 \cL\Phi_h^s
 &\le C_s-c_s\frac{\Phi_h^s}{\mathsf S}\sum_i g_i,\\
 \mathsf S\le\tfrac18
 &\ \Longrightarrow\ \sum_i g_i\ge ch^{-6}\ge c,\qquad
 \frac{\Phi_h^s}{\mathsf S}
       =\frac1N(\Phi_h^s)^{1+2/(s\gamma)}.
 \end{aligned}
\]
Hence
\[
 \cL\Phi_h^s\le C_s-c_s(\Phi_h^s)^{1+2/(s\gamma)},\qquad s=1,2,4.
\]
At a death, $\mathsf S\mapsto\mathsf S+u_j^2$;
\eqref{eq:copy-contraction} holds with the same $q<1$.
Lemma~\ref{lem:all-death-budget}, invariance and Tonelli give
\[
 \pi_h\Phi_h=\pi_hP_1^h\Phi_h\le C_0,\qquad
 \E_{\pi_h}e^{\alpha K_T}
 \le e^{CT}\pi_h\Phi_h\le C_0e^{CT}.
\]
Consequently,
\begin{equation}\label{eq:stationary-count}
 \begin{gathered}
 \pi_h\Phi_h\le C_0,\qquad
 \E_{\pi_h}e^{\alpha K_T}\le C_0e^{CT},\\
 \E_{\pi_h}K_t=\lambda_ht,\qquad \sup_h\lambda_h\le\Lambda<\infty.
 \end{gathered}
\end{equation}
The final identity follows from additivity of expected stationary counts;
the exponential estimate at $t=1$ bounds the slope.

By landing, choose a label $p$ measurably on
$\mathcal L=\{\max_i u_h(X_i(1))\ge c_*\}$, with
$\inf_x\Prob_x^h(\mathcal L)\ge1/2$. Enlarge all holes by
$\kappa=1+c_*/4$. Their barrier satisfies
\[
 u_h^\kappa=1-\kappa(1-u_h),\qquad
 u_h^\kappa(X_p(1))\ge3c_*/4,\qquad
 \tfrac12\Delta u_h^\kappa\ge-c_\kappa u_h^\kappa.
\]
The last inequality follows from harmonicity in the cores and uniform
positivity outside them. Let $\mathcal P$ require the label's free Brownian path to
avoid these enlarged holes on $[1,1+T]$. Stopped It\^o gives, on $\mathcal L$,
\[
 \begin{gathered}
 \Prob(\mathcal P\mid\mathcal F_1)\ge
 p_*:=\tfrac34c_*e^{-c_\kappa T}>0,\\
 \mathcal P\ \Longrightarrow\
 \inf_{1\le s\le1+T}u_h(X_p(s))\ge c':=1-\kappa^{-1}>0.
 \end{gathered}
\]

Factoring the semigroup at $T/2$, \eqref{eq:operator-limit} and the
uniform killed heat-row bound imply
\[
 \|P_h(T)1\|_\infty
 \le C(T/2)|M|^{1/2}e^{-k_{\min}T/2}+o_h(1).
\]
First choose $T$ large, then $h$ small, to make this at most $q_0<1$.
For each $i\ne p$, let $\mathcal E_i$ require its first death before
$1+T$, donor $p$, and survival thereafter. Conditional on the protected
path, the label primitives are independent and~\eqref{eq:survival-lower}
applies at every first-death time. Thus
\[
 \Prob\left(\bigcap_{i\ne p}\mathcal E_i
       \,\middle|\,\mathcal F_1,\text{protected path}\right)
 \ge\left[\frac{(1-q_0)p'}{N-1}\right]^{N-1},\qquad
 p'=c'e^{-c_0T},
\]
on $\mathcal L\cap\mathcal P$. Since
$\mathcal L\cap\mathcal P\cap\bigcap_{i\ne p}\mathcal E_i
\subset\mathsf A_{0,1+T}$, for $T_*=1+T$,
\begin{equation}\label{eq:uniform-ancestor}
 \inf_x\Prob_x^h(\mathsf A_{0,T_*})\ge
 \tfrac12p_*\left[\frac{(1-q_0)p'}{N-1}\right]^{N-1}=:\beta_*>0.
\end{equation}
The block construction and persistence of common ancestry give
\[
 \sup_x\Prob_x^h(\mathsf A_{0,nT_*}^c)\le(1-\beta_*)^n,
\]
uniformly in the starting offsets and small $h$.

\subsection{Stationary displacement moments}
Assume~\eqref{eq:transfer-inputs}. The soft process has the same type of
bounds: its count is dominated by a Poisson process of rate $Nk_{\max}$,
and prescribing exactly $N-1$ candidate events, with distinct victims
and one common donor, gives a block success probability at least
\[
 \beta_k=e^{-Nk_{\max}T_*}\frac1{(N-1)!}
                \left(\frac{k_{\min}T_*}{N-1}\right)^{N-1}>0.
\]
Each candidate is accepted with probability at least $k_{\min}/k_{\max}$.
Replacing $\beta$ by $\min\{\beta,\beta_k\}$ and enlarging the other
constants if necessary gives the same bounds for both models.
The bounded-rate construction following Lemma~\ref{lem:soft-hidden-moments}
provides $\nu_k$; Proposition~\ref{prop:template-foster} provides $\nu_h$.
All constants below may depend on $N,k$ and the constants
in~\eqref{eq:transfer-inputs}, and are independent of small $h$.

For either stationary hidden process, choose a representative at the
left endpoint of each finite window and generate its drivers and marks.
Common lattice translations preserve the relative paths and rewards,
making these window laws consistent. Set
\[
 \begin{gathered}
 E_m=\mathsf A_{-mT_*,0},
 \qquad q=1-\beta\in(0,1),\\
 A_m=4(K_{mT_*}+1)U_{mT_*},\qquad
 U_{mT_*}=\max_{i\le N}\sup_{-mT_*\le s\le0}
             |B_i(s)-B_i(-mT_*)|.
 \end{gathered}
\]
Here $K_{mT_*}$ counts copies on $[-mT_*,0]$, and $B_i$ is the
concatenated Brownian driver of label $i$. Put
$\mathfrak D_0=\mathfrak D(X_0)$ and $E_0=\varnothing$.
Ancestral identities are distinct at the left endpoint even when their
positions coincide. The ancestor estimate and persistence under copying
give
\[
 E_m\subset E_{m+1},\qquad \Prob(E_m^c)\le q^m,\qquad
 \Prob\Big(\bigcup_{m\ge1}E_m\Big)=1.
\]
On $E_m$, both ancestral paths start at the same position and use at
most $K_{mT_*}+1$ increments of norm at most $2U_{mT_*}$; hence
\[
 \mathfrak D_0\mathbf1_{E_m}\le A_m.
\]

For $t,u\ge0$ and $r>0$, the count bound and Brownian maximal inequality
imply
\[
 \begin{gathered}
 \Prob(K_t\ge u)\le\min\{1,Ce^{Ct-\alpha u}\},\qquad
 \E(K_t+1)^r\le C_r(1+t)^r,\\
 \E U_{mT_*}^r\le C_rm^{r/2}.
 \end{gathered}
\]
In particular, for every $p>0$,
\[
 \begin{aligned}
 \E A_m^{2p}
 &\le4^{2p}\bigl(\E(K_{mT_*}+1)^{4p}\bigr)^{1/2}
                 \bigl(\E U_{mT_*}^{4p}\bigr)^{1/2}\\
 &\le C_p(1+m)^{2p}m^p\le C_p(1+m)^{3p},\\
 \E\mathfrak D_0^p
 &=\sum_{m\ge1}\E[\mathfrak D_0^p;E_m\setminus E_{m-1}]\\
 &\le\sum_{m\ge1}(\E A_m^{2p})^{1/2}
                        \Prob(E_{m-1}^c)^{1/2}.
 \end{aligned}
\]
Consequently,
\begin{equation}\label{eq:uniform-lift-moments}
 \sup_h\nu_h\mathfrak D^p
 \le C_p\sum_{m\ge1}(1+m)^{3p/2}q^{(m-1)/2}<\infty
 \quad(0<p<\infty).
\end{equation}
The same bound holds for $\nu_k$. For the reference increment, tracing
its ancestry back to time zero gives
\begin{equation}\label{eq:reward-envelope}
 |\Delta_T|\le \mathfrak D_0+2(K_T+1)U_T,
\end{equation}
where $U_T=\max_i\sup_{0\le s\le T}|B_i(s)-B_i(0)|$. Hence
\[
 \E|\Delta_T|^p\le C_p\left[
  \E\mathfrak D_0^p+
  \bigl(\E(K_T+1)^{2p}\bigr)^{1/2}
  \bigl(\E U_T^{2p}\bigr)^{1/2}\right]\le C_{p,T},
\]
uniformly for the stationary hard and soft systems.

\subsection{Finite marked records in total variation}
Use the compressed record of Section~\ref{sec:dynamics}, with sampling
grid $G\supset\{0,T\}$ and post-copy states determined by the donor marks.

For phases $x=(x_1,\ldots,x_N)$ and $\star\in\{h,k\}$, define the
victim-$i$ next-death density and the quiet sampling kernel by
\begin{equation}\label{eq:event-kernel}
 \begin{aligned}
 \mathsf D_\star^i(w,x,dy)
 &=q_\star(w,x_i)\prod_{j\ne i}p_\star(w,x_j,y_j)\,dy,\\
 \mathsf K_\star^{\rm grid}(w,x,dy)
 &=\prod_jp_\star(w,x_j,y_j)\,dy.
 \end{aligned}
\end{equation}
In the first line $dy=\prod_{j\ne i}dy_j$. Let $R_{ij}$ insert the donor
coordinate in the victim slot. The marked death density is
\[
 \mathsf k_\star^{\rm death}(w,x)
 =\frac1{N-1}\sum_i\sum_{j\ne i}
       \delta_{(i,j)}\otimes(R_{ij})_\#\mathsf D_\star^i(w,x).
\]

\begin{lemma}\label{lem:compressed-record-tv}
For fixed $T$, a finite sampling grid and $B<\infty$, these record laws
converge in total variation, uniformly over probability densities $g$ on
$M_h^r$, extended by zero to $M^r$, with $\|g\|_2\le B$, for each
partition of the $N$ labels into $r$ roots. The soft system starts from
the same density.
\end{lemma}
\begin{proof}
Fix a count cutoff $m$. The truncated record stops at the $m$th death
and carries a flag if that death occurs before $T$; otherwise it records
the full window. Write $\tau_0=0$ and $w_j=\tau_j-\tau_{j-1}$ for the
complete inter-death waits, and let $G$ be the sampling grid. For positive
thresholds $\delta_1,\ldots,\delta_m,\varepsilon$, retain the event
\[
 \mathcal A=\left\{
 \begin{array}{l}
  w_j\ge\delta_j\text{ whenever }j\le m\text{ and }\tau_j\le T,\\
  \operatorname{dist}(\tau_j,G)\ge\varepsilon
       \text{ whenever }j\le m\text{ and }\tau_j\le T
 \end{array}\right\}.
\]
On $\mathcal A$, successive nodes have spacing at least
$\eta=\eta(\delta_1,\ldots,\delta_m,\varepsilon,G)>0$.
Constants below may depend on $k,m,T,N,G,\eta,B$, uniformly in $g,h$.

Put
$d_h(w,x)=\int|p_h(w,x,y)-p_k(w,x,y)|\,dy$.
Telescoping the products and using the subprobability masses of the
killed kernels gives
\begin{align}
 \|\mathsf K_h^{\rm grid}(w,x)-\mathsf K_k^{\rm grid}(w,x)\|_\TV
       &\le\sum_jd_h(w,x_j),\label{eq:grid-error}\\
 \|\mathsf k_h^{\rm death}(w,x)-\mathsf k_k^{\rm death}(w,x)\|_\TV
       &\le\sum_i|q_h(w,x_i)-q_k(w,x_i)|
             \notag\\
       &\hspace{1em}{}+(N-1)k_{\max}\sum_jd_h(w,x_j).\label{eq:death-error}
\end{align}
The second telescoping sum replaces $q_h$ first and then $p_h$;
$q_k\le k_{\max}$ by~\eqref{eq:soft-density}, and donor reset contracts $\TV$.

For a root partition $\mathcal P$, $x_i=x_{\mathcal P(i)}$.
Writing $g_a$ for its one-root marginals, we have
\[
 \max_a\|g_a\|_2\le |M|^{(r-1)/2}B,\qquad
 \sup_{\eta\le w\le T}\|d_h(w)\|_2
       \le |M|^{1/2}a_h(\eta).
\]
Thus the integrated first-transition error is at most
$C_{\eta,T,N}(B+1)a_h(\eta)$ by~\eqref{eq:grid-error}--\eqref{eq:death-error}.

At a subsequent death node, the hard prefix has a joint, unnormalized
density $r_{ij}(s,z)\,ds\,dz$, where $s$ is its absolute time and
$z\in M^{N-1}$ the survivors. Bounded $q_h,p_h$ at spacing $\eta$ and prefix mass
at most one give
\[
 0\le r_{ij}(s,z)\le C_{\eta,T,N,m}.
\]
Sampling-node densities satisfy the same bound. Each slot of $R_{ij}z$
projects a free coordinate, so for $\ell$ free coordinates and
$\eta\le w\le T$,
\begin{align*}
 \int_{M^\ell}|q_h(w,x_a)-q_k(w,x_a)|\,dz
       &\le |M|^{\ell-1/2}a_h(\eta),\\
 \int_{M^\ell}d_h(w,x_a)\,dz&\le |M|^\ell a_h(\eta).
\end{align*}
This includes $w=t_*-s$ for the next sampling time $t_*$; reset
duplicates input coordinates, while the output integration is unchanged.

For a prefix law $\mu(d\omega,dz)$ with marginal $\mu_z$, compared
transitions $K_h,K_k$, and a future subprobability kernel $Q$,
\[
 \begin{aligned}
 &\sup_{|F|\le1}\left|\int\mu(d\omega,dz)(K_h-K_k)(z,dz')QF(\omega,z,z')\right|\\
 &\quad\le\int\mu_z(dz)\|K_h(z,\cdot)-K_k(z,\cdot)\|_\TV.
 \end{aligned}
\]
Telescoping at most $m+|G|$ transitions gives the restricted-law estimate
\[
 \sup_{\|g\|_2\le B}
 \|\mathsf L_{h,g}^{m}|_{\mathcal A}
       -\mathsf L_{k,g}^{m}|_{\mathcal A}\|_\TV
 \le C_{\eta,T,N,m,G}(B+1)a_h(\eta)\longrightarrow0.
\]

The excluded mass is controlled in chronological order. Define
\[
 F_j=\{\tau_j\le T,\ w_j<\delta_j,
                  w_\ell\ge\delta_\ell\text{ for }\ell<j\}.
\]
For the first wait, independence of the label motions conditional on
the roots gives
\[
 \Prob_g^h(F_1)
 \le B\left\|1-\prod_iS_h(\delta_1,x_{\mathcal P(i)})
                          \right\|_{L^2(M^r)},\qquad
 \limsup_{h\to0}\sup_g\Prob_g^h(F_1)\le C B\delta_1.
\]
Use~\eqref{eq:operator-limit} and $1-S_k(t,x)\le k_{\max}t$.
For $j\ge2$, the preceding positive complete wait bounds its survivor
density by $B_j=B_j(T,N,k,\delta_1,\ldots,\delta_{j-1})<\infty$.
Integration over the free survivor coordinates, including repeated
slots via $1-S_h^a\le a(1-S_h)$, gives
\[
 \limsup_{h\to0}\sup_g\Prob_g^h(F_j)\le C B_j\delta_j,
 \qquad \sup_g\Prob_g^k(F_j)\le C B_j\delta_j.
\]
The soft first-wait bound is also $CB\delta_1$. Bounded death-time
densities after fixing the complete waits give
\[
 \sup_g\Prob_g^h\left\{
 \begin{array}{l}
  w_j\ge\delta_j\text{ for each death }j\le m,\\
  \operatorname{dist}(\tau_j,G)<\varepsilon
              \text{ for some death }j\le m
 \end{array}\right\}
 \le C_{T,N,m,G,\delta_1,\ldots,\delta_m}\varepsilon,
\]
and the same estimate holds in the soft system.

For $a>0$, choose thresholds recursively:
\[
 CB\delta_1<\frac am,\qquad
 CB_j(\delta_1,\ldots,\delta_{j-1})\delta_j<\frac am
 \ (2\le j\le m),\qquad C_{\boldsymbol\delta}\varepsilon<a.
\]
Since $\bigcup_jF_j$ contains every short complete wait,
\[
 \limsup_{h\to0}\sup_g
   \bigl[\Prob_g^h(\mathcal A^c)+\Prob_g^k(\mathcal A^c)\bigr]
 \le4a.
\]
Together with the retained-law bound, this yields
\[
 \limsup_{h\to0}\sup_g
 \|\mathsf L_{h,g}^m-\mathsf L_{k,g}^m\|_\TV\le4a
 \xrightarrow[a\downarrow0]{}0.
\]

Put $b_m=\Prob\{\operatorname{Poisson}(Nk_{\max}T)\ge m\}$.
The flag is exactly $\{K_T\ge m\}$, up to an atomless endpoint, and
soft Poisson domination therefore gives
\[
 \limsup_{h\to0}\sup_g\Prob_g^h(K_T\ge m)\le b_m,\qquad
 \sup_g\Prob_g^k(K_T\ge m)\le b_m.
\]
If $\mathsf L_{h,g}$ is the full compressed-record law, deletion of
the flagged parts costs at most twice their probabilities. Thus
\[
 \limsup_{h\to0}\sup_g
     \|\mathsf L_{h,g}-\mathsf L_{k,g}\|_\TV\le4b_m.
\]
Let $m\to\infty$.
\end{proof}

For the varying root laws of Theorem~\ref{thm:marked-transfer},
\[
 \begin{aligned}
 \|\mathsf L_{h,f_h}-\mathsf L_{k,f}\|_\TV
 &\le\|\mathsf L_{h,f_h}-\mathsf L_{k,f_h}\|_\TV+\|f_h-f\|_1\\
 &\le o_h(1)+|M|^{r/2}\|f_h-f\|_2\longrightarrow0.
 \end{aligned}
\]

\subsection{Convergence of continuous genealogical records}\label{sec:complete-record-bridge}
\begin{lemma}[Continuous record convergence]\label{lem:full-record-weak}
Under the finite-record hypotheses of Lemma~\ref{lem:compressed-record-tv}, the complete continuous records on $[0,T]$ converge weakly, including the victim's pre-death endpoint and the terminal incomplete pieces.
\end{lemma}
\begin{proof}
Fix $m$ and a deterministic grid of mesh $\theta$. Insert every death
time, so adjacent nodes are at most $\theta$ apart. At a death the
survivor endpoints are recorded, and the victim's poststate is the donor's
position. Interpolate known endpoints of each continuous phase piece
by shortest torus geodesics. Extend the victim's last known value
constantly on its final interval, whose duration is at most $\theta$.
On the count flag, return a fixed zero-copy record. This defines the
common measurable reconstruction $\Psi_{m,\theta}$; write $\mathsf C$
for the complete record and $\mathsf R^{m,\theta}$ for its compressed input.

For the concatenated label drivers put
\[
 \omega_T(u)=\max_{i\le N}
       \sup_{\substack{s,t\in[0,T]\\|s-t|\le u}}|B_i(s)-B_i(t)|,
 \qquad \mathcal B_\theta=\{\omega_T(2\theta)<L/4\}.
\]
On $\mathcal B_\theta$ geodesics are unique. Interpolation and constant
extension have error at most $C\omega_T(2\theta)$, also at random death
times. Reparameterizing pieces to $[0,1]$ preserves the bound, so
\[
 d\bigl(\mathsf C,\Psi_{m,\theta}(\mathsf R^{m,\theta})\bigr)
       \le C_m\omega_T(2\theta)
       \quad\text{on }\mathcal B_\theta\cap\{K_T<m\},
\]
where $d$ is the product metric of the complete-record space. In either model,
\[
 \varepsilon_{m,T}(\theta)
   :=\E\bigl[2\wedge C_m\omega_T(2\theta)\bigr]
                 +2\Prob(\mathcal B_\theta^c)
       \longrightarrow0\qquad(\theta\downarrow0).
\]
For probability laws on a metric space $(E,d)$, write
\[
 d_{\rm BL}(\mu,\nu)
   =\sup_{\substack{\|F\|_\infty\le1\\\operatorname{Lip}_d(F)\le1}}
                 |\mu F-\nu F|.
\]
For each such $F$, the reconstruction in either model gives
\[
 \left|\E F(\mathsf C)
        -\E F\bigl(\Psi_{m,\theta}(\mathsf R^{m,\theta})\bigr)\right|
 \le\varepsilon_{m,T}(\theta)+2\Prob(K_T\ge m).
\]
Taking the supremum over $F$ and comparing the compressed records
yields
\[
 \begin{aligned}
 d_{\rm BL}\bigl(\Law_g^h(\mathsf C),\Law_g^k(\mathsf C)\bigr)
 &\le\|\mathsf L_{h,g}^{m}-\mathsf L_{k,g}^{m}\|_\TV
             +2\varepsilon_{m,T}(\theta)\\
 &\quad+2\Prob_g^h(K_T\ge m)+2\Prob_g^k(K_T\ge m).
 \end{aligned}
\]
Lemma~\ref{lem:compressed-record-tv} consequently implies
\[
 \limsup_{h\to0}\sup_g
 d_{\rm BL}\bigl(\Law_g^h(\mathsf C),\Law_g^k(\mathsf C)\bigr)
 \le2\varepsilon_{m,T}(\theta)+4b_m.
\]
Since $\varepsilon_{m,T}(\theta)\to0$ and $b_m\to0$, taking limits in
the order $h\downarrow0$, $\theta\downarrow0$, $m\to\infty$ gives weak
convergence on the Polish complete-record space.
\end{proof}

For the first $m$ complete records, apply the preceding comparison on
a deterministic horizon $T$. Soft inter-death waits have conditional
tail at most $e^{-Nk_{\min}s}$, since the total killing rate is at least
$Nk_{\min}$. If the sum of $m$ waits exceeds $T$, one exceeds $T/m$.
Thus, uniformly in the root law,
\[
 \begin{aligned}
 \Prob_g^k(\tau_m>T)&\le m e^{-Nk_{\min}T/m},\\
 \limsup_{h\to0}\sup_g\Prob_g^h(\tau_m>T)
             &\le m e^{-Nk_{\min}T/m}.
 \end{aligned}
\]
For the second line, the count flag $\mathsf f_m=\one_{\{K_T\ge m\}}$
is part of the truncated law, so
\[
 \Prob_g^h(\tau_m>T)=1-\E_g^h\mathsf f_m
 \le\Prob_g^k(\tau_m>T)
       +\|\mathsf L_{h,g}^m-\mathsf L_{k,g}^m\|_\TV.
\]
Selecting the first $m$ pieces is continuous on
$\{\tau_m<T\}$ outside deaths exactly at the horizon, an event of
limiting probability zero. The window limit therefore gives the
first-$m$ record limit with an error bounded by
$4m e^{-Nk_{\min}T/m}$. Taking $h\to0$ at fixed $T$ and then
$T\to\infty$ removes this error.

Complete-record convergence also yields strong convergence of marked
survivor densities. These densities remain in the $L^2$ input class
for a subsequent copying step.

\begin{lemma}[Strong marked copying operators]\label{lem:strong-marked-operator}
Fix a root partition $\mathcal P$ with $r$ roots, labels $v\ne d$,
$0<\delta<M_0<\infty$, and a bounded continuous function $G$ of the
first complete copying record $\mathsf C_1$. For $\star\in\{h,k\}$,
define the unnormalized output density $Q_\star^Gf$ on $M^{N-1}$ by
\[
 \begin{aligned}
 \int\psi(y)Q_\star^Gf(y)\,dy
 &=\int f(\xi)\E_{\mathcal P,\xi}^\star
 \left[G(\mathsf C_1)\psi((\operatorname{ph}X_j(\tau_1))_{j\ne v});
 \ E_{v,d}^{\delta,M_0}\right]d\xi,\\
 E_{v,d}^{\delta,M_0}
 &=\{\delta\le\tau_1\le M_0,\ (v_1,d_1)=(v,d)\}.
 \end{aligned}
\]
Hard outputs are extended by zero. For signed inputs $f_h$ supported
on $M_h^r$ whose zero extensions converge to $f$ in $L^2(M^r)$,
\begin{equation}\label{eq:strong-marked-copying}
 \|Q_h^Gf_h-Q_k^Gf\|_2\longrightarrow0,\qquad
 \|Q_\star^Gf\|_2\le C_\delta\|G\|_\infty\|f\|_2.
\end{equation}
For $G=1$, also $\|Q_\star^1f\|_{H^1(M^{N-1})}\le C_\delta\|f\|_2$.
\end{lemma}
\begin{proof}
See Appendix~\ref{app:marked-output-proof}.
\end{proof}

\subsection{Stationary records and the winding reconstruction}
The first bound in~\eqref{eq:transfer-inputs} gives
$\sup_h\E_{\pi_h}K_1<\infty$. Stationarity and additivity of expected
counts then give $\E_{\pi_h}K_t=\lambda_ht$, with
$\sup_h\lambda_h\le\Lambda<\infty$.
For $0<\delta<[2(\Lambda+1)]^{-1}$, split the invariant law according
to whether a death occurred in $[-\delta,0]$:
\[
 \begin{gathered}
 \pi_h(dy)=g_{h,\delta}(y)\,dy+\rho_{h,\delta}(dy),\qquad
 g_{h,\delta}(y)=\int\pi_h(dx)\prod_{i=1}^Np_h(\delta,x_i,y_i),\\
 0\le g_{h,\delta}\le G_\delta^N,\qquad
 \rho_{h,\delta}(M^N)\le\Prob_{\pi_h}(K_\delta\ge1)
                                      \le\Lambda\delta.
 \end{gathered}
\]
Here $G_\delta=\sup_{x,y}p_0(\delta,x,y)$ for the free torus kernel
$p_0$. If $a_{h,\delta}=\int g_{h,\delta}$ and
$\bar g_{h,\delta}=g_{h,\delta}/a_{h,\delta}$, then
\[
 a_{h,\delta}\ge\tfrac12,\qquad
 \|\bar g_{h,\delta}\|_2\le
             2G_\delta^N|M|^{N/2}=:B_\delta.
\]
For any fixed-window compressed-record kernels $\mathsf R_h,\mathsf R_k$,
Lemma~\ref{lem:compressed-record-tv}, with $r=N$, consequently gives
\[
 \begin{aligned}
 \|\pi_h(\mathsf R_h-\mathsf R_k)\|_\TV
 &\le a_{h,\delta}
       \|\bar g_{h,\delta}(\mathsf R_h-\mathsf R_k)\|_\TV
                         +2\rho_{h,\delta}(M^N)\\
 &\le o_{h,\delta}(1)+2\Lambda\delta,
 \end{aligned}
\]
where $o_{h,\delta}(1)\to0$ at every fixed $\delta>0$. The uniformity
over $L^2$ root laws is essential here, because
$\bar g_{h,\delta}$ varies with $h$.

Write $Q_h(t),Q_k(t)$ for the hard and soft phase transition kernels.
For a fixed $t_0>0$, absence of all candidate deaths and positivity of
the free heat kernel give
\[
 Q_k(t_0,x,dy)\ge e^{-Nk_{\max}t_0}
                    \prod_i p_0(t_0,x_i,y_i)\,dy
                  \ge\beta_0\,\frac{dy}{|M|^N},
 \qquad \beta_0>0.
\]
Thus $Q_k(t_0)$ contracts differences of probability measures by
$1-\beta_0$. Invariance and the terminal-state projection of the
record comparison imply
\[
 \begin{aligned}
 \pi_h-\pi_k
 &=\pi_h\bigl(Q_h(t_0)-Q_k(t_0)\bigr)
                     +(\pi_h-\pi_k)Q_k(t_0),\\
 \beta_0\|\pi_h-\pi_k\|_\TV
 &\le\|\pi_h(Q_h(t_0)-Q_k(t_0))\|_\TV
                   \le o_{h,\delta}(1)+2\Lambda\delta.
 \end{aligned}
\]
Taking $h\downarrow0$ at fixed $\delta$ and then $\delta\downarrow0$
proves phase-law convergence. For a fixed window the compressed-record laws
then satisfy
\[
 \begin{aligned}
 \|\pi_h\mathsf R_h-\pi_k\mathsf R_k\|_\TV
 &\le\|\pi_h(\mathsf R_h-\mathsf R_k)\|_\TV
                                +\|\pi_h-\pi_k\|_\TV\\
 &\le o_{h,\delta}(1)+2\Lambda\delta+\|\pi_h-\pi_k\|_\TV.
 \end{aligned}
\]
Hence, first $h\downarrow0$ and then $\delta\downarrow0$ give compressed-record
convergence. Lemma~\ref{lem:full-record-weak}, with the stationary count tail, gives
\[
 \limsup_{h\to0}
 d_{\rm BL}\bigl(\Law_{\pi_h}(\mathsf C),
                 \Law_{\pi_k}(\mathsf C)\bigr)
 \le2\varepsilon_{m,T}(\theta)+Ce^{CT-\alpha m}.
\]
The soft Poisson tail is absorbed in $C,\alpha$; let $\theta\downarrow0$
and then $m\to\infty$.

Fix integers $n,m,\ell\ge1$ and a deterministic grid on $[-W,T]$,
where $W=nT_*$, containing $-W,0,T$ and having mesh at most
$(W+T)/\ell$ and at most $\ell+2$ nodes.
Use the ancestral colors $\mathfrak a_{-W,t}(i)$ and define
\[
 \begin{gathered}
 \mathcal A_n=\mathsf A_{-W,0},\qquad
 \mathcal C_m=\{K_{[-W,T]}<m\},\\
 \mathcal B_{W,\ell}=
 \left\{\max_{i\le N,\ a\le3}
 \sup_{\substack{s,t\in[-W,T]\\|s-t|\le (W+T)/\ell}}
            |B_i^{(a)}(s)-B_i^{(a)}(t)|<\pi\right\}.
 \end{gathered}
\]
For the hard and soft stationary systems, with one common $q<1$,
\[
 \Prob(\mathcal A_n^c)\le Cq^n,\qquad
 \Prob(\mathcal C_m^c)\le Ce^{C(W+T)-\alpha m}.
\]
The reflection principle, applied on pairs of adjacent grid intervals,
also gives
\begin{equation}\label{eq:modulus}
 \Prob(\mathcal B_{W,\ell}^c)
       \le CN(\ell+1)e^{-c\ell/(W+T)}.
\end{equation}

Let $\mathsf R_{n,m,\ell}$ be the truncated compressed record on this grid.
Initialize virtual positions $\widehat X_i(-W)$ at canonical
representatives of the initial phases. For consecutive nodes $s,t$,
update each survivor by
\[
 \widehat X_i(t)=\widehat X_i(s)
       +\operatorname{princ}\bigl(\operatorname{ph}X_i(t)
                                  -\operatorname{ph}X_i(s)\bigr),
\]
where $\operatorname{princ}\in[-\pi,\pi)^3$ coordinatewise; at a death
$(i,j)$ then set $\widehat X_i=\widehat X_j$.
On $\mathcal B_{W,\ell}$ principal increments equal Brownian increments,
so induction from $\widehat X_i(-W)-X_i(-W)=Lz_i$, $z_i\in\Z^3$, gives
\[
 \widehat X_i(t)-X_i(t)=Lz_{\mathfrak a_{-W,t}(i)}
 \quad\text{at every record node }t.
\]
On $\mathcal A_n$, one common ancestral color persists through $T$;
the common lattice error cancels, as illustrated in
Figure~\ref{fig:genealogy-reconstruction}:
\[
 \begin{gathered}
 \operatorname{ph}\widehat X_N(0)=\operatorname{ph}X_N(0),\qquad
 \widehat X_i(0)-\widehat X_N(0)=X_i(0)-X_N(0),\\
 \widehat X_N(T)-\widehat X_N(0)=X_N(T)-X_N(0).
 \end{gathered}
\]
\begin{figure}[tbp]
\centering
\includegraphics{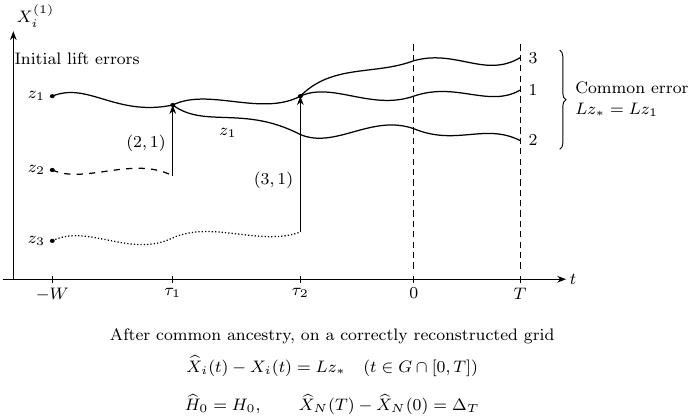}
\caption{Common ancestry removes the initial lift ambiguity.
Line styles encode three initial roots with errors $Lz_i$, $z_i\in\Z^3$.
Copies $(2,1)$ and $(3,1)$ transfer the donor's root and lattice error;
later Brownian pieces remain independent. Write $\widehat X$ for the
grid reconstruction and $\widehat H_0=H(\widehat X(0))$.
On a correctly reconstructed grid $G$ and $\mathsf A_{-W,0}$, the common
error $Lz_1$ cancels from $H_0$ and $\Delta_T$}
\label{fig:genealogy-reconstruction}
\end{figure}

Define $\Psi_{n,m,\ell}$ to return this reconstructed $(H_0,\Delta_T)$
on unflagged records and a fixed value on flagged records. For
$Z=(H_0,\Delta_T)$, the preceding identities prove
\[
 Z=\Psi_{n,m,\ell}(\mathsf R_{n,m,\ell})
        \quad\text{on }\mathcal A_n\cap\mathcal B_{W,\ell}\cap\mathcal C_m.
\]
Pushforward contraction and the three bad-event bounds give
\[
 \begin{aligned}
 &\|\Law_{\nu_h}(Z)-\Law_{\nu_k}(Z)\|_\TV\\
 &\quad\le\|\Law_{\pi_h}(\mathsf R_{n,m,\ell})
                   -\Law_{\pi_k}(\mathsf R_{n,m,\ell})\|_\TV\\
 &\qquad+Cq^n+Ce^{C(W+T)-\alpha m}
                   +CN(\ell+1)e^{-c\ell/(W+T)}.
 \end{aligned}
\]
At fixed $n,m,\ell$ the record-law difference vanishes. Thus
\[
 \lim_{n\to\infty}\lim_{m\to\infty}\lim_{\ell\to\infty}
 \limsup_{h\downarrow0}
 \|\Law_{\nu_h}(Z)-\Law_{\nu_k}(Z)\|_\TV=0,
\]
proving~\eqref{eq:stationary-target-transfer}.

Fix $T>0$ and denote the target-law distance by $d_h^Z(T)$.
The moment estimates give
$\sup_h\E_{\nu_h}|\Delta_T|^2\le C_T$ and
$\E_{\nu_k}|\Delta_T|^2\le C_T$.
For $\chi_R(s)=\max\{-R,\min\{R,s\}\}$, truncation and stationary
increment additivity give, for each coordinate $a\le3$,
\[
 \begin{aligned}
 T|\mathbf v_h(k)^{(a)}-\mathbf v(k)^{(a)}|
 &=|\E_{\nu_h}\Delta_T^{(a)}-\E_{\nu_k}\Delta_T^{(a)}|\\
 &\le|\E_{\nu_h}\chi_R(\Delta_T^{(a)})
                  -\E_{\nu_k}\chi_R(\Delta_T^{(a)})|
       +\frac{2C_T}{R}\\
 &\le R d_h^Z(T)+\frac{2C_T}{R}.
 \end{aligned}
\]
First let $h\downarrow0$ at fixed $R$, then $R\to\infty$.
This proves~\eqref{eq:velocity-transfer}.

For the joint relative-path limit, retain the complete continuous
record on $[-W,T]$. On $\mathcal A_n$, lift its phase pieces
continuously and copy the donor's lift at each death. After coalescence
all possible initial lift choices differ by one common lattice vector;
it cancels from $H_0$ and every relative path on $[0,T]$.
For convergent records on a fixed count component with no death at
$0,T$, let $\lambda_n:[0,T]\to[0,T]$ be piecewise linear, fixing
the endpoints and satisfying $\lambda_n(\tau_r)=\tau_r^{(n)}$ at every
copying time in $(0,T)$. Uniform convergence
of the rescaled pieces and continuity of local lifts give
\[
 \|\lambda_n-\id\|_\infty\to0,\qquad
 \|\mathsf Y_T^{(n)}\circ\lambda_n-\mathsf Y_T\|_\infty\to0,
 \qquad H_0^{(n)}\to H_0.
\]
Thus the output is continuous into the $J_1$ relative-path space.
Local lifts give the same relative output across fundamental-cell boundaries.
Deaths at $0,T$ have probability zero in the limiting soft law.
The continuous-record limit and the continuous mapping theorem thus
give, for the joint output laws $\mathsf J_h,\mathsf J_k$,
\[
 \limsup_{h\to0}d_{\rm BL}(\mathsf J_h,\mathsf J_k)
       \le Cq^n+Ce^{C(W+T)-\alpha m}.
\]
Here the joint output consists of the complete record on $[0,T]$,
the initial hidden state and the real relative paths. First let
$m\to\infty$ at fixed $W=nT_*$, then $n\to\infty$. This proves
the claimed joint weak convergence.

\begin{corollary}[Integer winding transfer]\label{cor:integer-winding-transfer}
Under Theorem~\ref{thm:marked-transfer}, let $c:M\to[0,L)^3$ be the
canonical representative, put $\theta_t=c(\operatorname{ph}X_N(t))$,
and define the componentwise integer winding
\[
 W_T=\left\lfloor\frac{X_N(T)}L\right\rfloor
       -\left\lfloor\frac{X_N(0)}L\right\rfloor
     =\left\lfloor\frac{\theta_0+\Delta_T}L\right\rfloor\in\Z^3.
\]
For each fixed $T>0$, the joint laws of $(H_0,\Delta_T,W_T)$ converge
in total variation. All mixed polynomial moments of $(\Delta_T,W_T)$
and every absolute moment of $W_T$ converge, and for $\star\in\{h,k\}$,
\begin{equation}\label{eq:integer-winding-mean}
 \E_{\nu_\star}W_T=\frac{T\mathbf v_\star(k)}L.
\end{equation}
\end{corollary}
\begin{proof}
The displayed floor map is measurable and common to both models, so
\eqref{eq:stationary-target-transfer} contracts to the asserted joint
law. Componentwise,
\[
 W_T-\frac{\Delta_T}L=\frac{\theta_0-\theta_T}L,
 \qquad |W_T|\le\sqrt3+\frac{|\Delta_T|}L.
\]
The higher-order uniform displacement moments imply uniform
integrability of every fixed-order mixed or absolute moment.
Hidden stationarity gives
$\E\theta_0=\E\theta_T$ and $\E\Delta_T=T\mathbf v_\star(k)$,
proving~\eqref{eq:integer-winding-mean}.
\end{proof}

\begin{corollary}[A periodic domain with positive current]\label{cor:positive-template}
For the rate $k_N$ of~\eqref{eq:epsilon-choice}, one can choose and fix a mesh $h=L/m$ for which
\begin{equation}\label{eq:positive-hard}
 v_h:=\mathbf v_h(k_N)^{(1)}\ge v_{N,\epsilon_N}/2>0.
\end{equation}
For this fixed domain and every $T\ge0$, $\E_{\nu_h}\Delta_T^{(1)}=Tv_h$.
\end{corollary}
\begin{proof}
The first-coordinate projection of the three-dimensional soft model with $k_N(x)=k_N(x^{(1)})$ is exactly the line model of Section~\ref{sec:copying-current}. Its contrast satisfies~\eqref{eq:ratio}, so Corollary~\ref{cor:small-contrast-transfer} preserves its positive signed velocity. Stationary increment additivity gives the identity for every $T$.
\end{proof}
The mesh is fixed from this point onwards. All subsequent constants for this domain and perturbation thresholds may depend on it.

\section{Boundary restarts and ballistic escape}\label{sec:boundary-restarts}
Fix the periodic domain with positive current $D=D_h$ selected in
\eqref{eq:positive-hard}, and use the functions
$\Phi,\mathcal V,\mathcal W$ from Section~\ref{sec:fixed}.
Constants in this section may depend on this fixed mesh.

Foster sublevel sets allow several labels to approach absorbing walls.
The bound on $\Phi$ keeps at least one label in the interior; the others
restart through a rapid cascade of copies. Uniform donor choice determines
the limiting law. Each label that joins an interior ancestor becomes a
possible donor, producing the P\'olya reinforcement below.

\subsection{A short-time boundary estimate}
\begin{lemma}\label{lem:short-wall}
For a continuous label segment starting at wall distance $r$, let
$\tau_{\partial D}$ be its first wall hit. Uniformly over the small
perturbation class \eqref{eq:perturbation},
\begin{equation}\label{eq:wall}
 \Prob(\tau_{\partial D}>t)\le\min(1,Cr/\sqrt t),
 \qquad 0<t\le1.
\end{equation}
The estimate also holds conditionally at an actual stopping time, when
the coefficients depend on the other labels' histories.
\end{lemma}
\begin{proof}
In a fixed tubular neighborhood the distance process has the decomposition
$R_t=r+M_t+\int_0^t\zeta_s\,\dd s$, with $|\zeta_s|\le B_0$ and
$c_-\dd t\le\dd\langle M\rangle_t\le c_+\dd t$. Stop on reaching
distance zero or the outer edge $r_0$ of the tube. For
\[
 g(r)=\frac{1-e^{-kr}}{1-e^{-kr_0}},\qquad k\ge2B_0/c_-,
\]
$g(R_t)$ is a supermartingale along this stopped process. Hence the probability
of reaching $r_0$ first is at most $Cr$.

On the event of remaining in the tube until $t$, the
Dambis--Dubins--Schwarz time change gives a Brownian motion $\beta$ with
$r+\beta_u+(B_0/c_-)u>0$ for every $u\le c_-t$.
The reflection formula for Brownian motion with positive drift bounds
this probability by $C(r/\sqrt t+r)$. Adding the outer-edge probability
and using $t\le1$ gives \eqref{eq:wall}. For starting points outside the
tube, increase $C$ so that its right-hand side is one.
The drift and quadratic variation bounds persist conditionally at
stopping times, giving the conditional estimate.
\end{proof}

\subsection{The P\'olya restart law}
Let $x\in\overline D^N$ have interior-label set $I$ of size $k_0\ge1$ and
wall-label set $B$ of size $b=N-k_0$. For $a:B\to I$, put
$a_j=|a^{-1}(j)|$, and let $R_ax$ replace $x_i$ by $x_{a(i)}$ for
$i\in B$. Write $(k)_b=k(k+1)\cdots(k+b-1)$ and $(k)_0=1$.

\Needspace{13\baselineskip}
\begin{theorem}[P\'olya boundary restart]\label{thm:polya-restart}
Suppose interior configurations $x^{(n)}$ converge to $x$ and admissible
perturbation amplitudes $\eta_n$ tend to zero. For every $T>0$, the joint
law of the hidden endpoint and actual reference displacement converges
weakly to
\begin{equation}\label{eq:restart-endpoint}
 \sum_{a:B\to I}p(a)\,
 \Law_{R_ax}^0\bigl(H_T,\Delta_T+(R_ax)_N-x_N\bigr),
 \qquad p(a)=\frac{\prod_{j\in I}a_j!}{(k_0)_b}.
\end{equation}
Here $\Delta_T=X_N(T)-(R_ax)_N$ in the baseline system started at $R_ax$.
The weights are unchanged by an adaptive order of wall deaths.
\end{theorem}
Figure~\ref{fig:polya-restart} displays the assignment weights for two
interior roots and two wall labels.
\begin{figure}[tbp]
\centering
\includegraphics{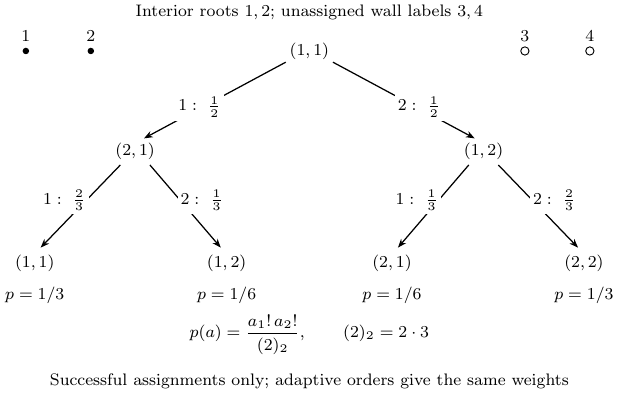}
\caption{A P\'olya restart with $I=\{1,2\}$ and $B=\{3,4\}$.
The tree records successful assignments in the order $3$, then $4$;
copies of unassigned donors are omitted. Internal node pairs give the
populations of the two ancestral roots, including their initial labels.
Leaves give $(a(3),a(4))$ and its probability. The same final weights
hold for adaptive victim orders chosen before the next donor mark is
revealed}
\label{fig:polya-restart}
\end{figure}

\begin{proof}
When $b=0$, this is interior continuity from
Lemma~\ref{lem:interior-prefix-continuity}. Suppose $b>0$.
Give the $k_0$ interior labels distinct colors. In an ideal assignment
process a wall label stays unassigned when it copies an unassigned label,
and acquires a color when it copies an assigned label. After $n$ attempts,
let $\mathcal I_n$ be the assigned set, $k_n=|\mathcal I_n|$, and
$n_j(n)$ its color populations, $j\in I$. Let $\mathcal G_n$ be the
information after attempt $n$. Before completion,
$\mathcal G_{n+1}^-\supset\mathcal G_n$ also reveals the next victim
$i_{n+1}\notin\mathcal I_n$, before its donor mark is revealed.
The victim can depend on all information in this sigma-field. The new
mark is uniform among the other $N-1$ labels, so
\begin{equation}\label{eq:adapted-assignment}
 \begin{split}
 \Prob(\text{new color }j\mid\mathcal G_{n+1}^-)
      &=\frac{n_j(n)}{N-1},\\
 \Prob(\text{no assignment}\mid\mathcal G_{n+1}^-)
      &=\frac{N-1-k_n}{N-1}.
 \end{split}
\end{equation}
Conditional on success, the color law is therefore $n_j(n)/k_n$.

Fix a target assignment $a:B\to I$. Let $r_j(n)$ count its still
unassigned targets of color $j$, and put $w_n=N-k_n$. Define
\[
 Z_n^a=\one_{\{\text{assignments through }n\text{ agree with }a\}}
       \frac{\prod_{j\in I}(n_j(n))_{r_j(n)}}{(k_n)_{w_n}}.
\]
If $a(i_{n+1})=j$ and there has been no mismatch, the three possible
updates are
\[
 Z_{n+1}^a=
 \begin{cases}
  (k_n/n_j(n))Z_n^a,&\text{donor has color }j,\\
  0,&\text{donor has another color},\\
  Z_n^a,&\text{donor is unassigned}.
 \end{cases}
\]
Thus, by \eqref{eq:adapted-assignment},
\[
 \E(Z_{n+1}^a\mid\mathcal G_{n+1}^-)
 =\frac{n_j(n)}{N-1}\frac{k_n}{n_j(n)}Z_n^a
   +\frac{N-1-k_n}{N-1}Z_n^a=Z_n^a.
\]
The Vandermonde identity gives $0\le Z_n^a\le1$, for every adapted
victim order.
Let $J=\inf\{n:k_n=N\}$ be the completion attempt. At each population
size $k$, the conditional success probability is $p_k=k/(N-1)$.
Iteration of this conditional law gives
\begin{equation}\label{eq:restart-attempts}
 \E e^{sJ}
 =\prod_{k=k_0}^{N-1}
 \frac{p_ke^s}{1-(1-p_k)e^s}<\infty
\end{equation}
for sufficiently small $s>0$. Since $Z_J^a$ is the indicator of the
final assignment $a$, bounded optional stopping yields
\[
 \Prob(\text{final assignment}=a)=\E Z_J^a=Z_0^a
       =\frac{\prod_j a_j!}{(k_0)_b}=p(a).
\]
Summing over the $b!/\prod_j a_j!$ assignments of each count vector gives
\[
 \sum_{a:B\to I}p(a)
 =\binom{b+k_0-1}{k_0-1}\frac{b!}{(k_0)_b}=1.
\]

For the actual cascade, let $\mathcal I(t)$ consist of the original
interior labels and their colored descendants. Stop at
$\theta=\inf\{t:|\mathcal I(t)|=N\}$ or the first assigned-label death
$\sigma$. Let $\tau_0=0$ and let $\tau_j$ be the successive unassigned
death times before $\theta\wedge\sigma$, including the completing death;
set subsequent $\tau_j=\infty$. For all large $n$, choose $d>0$ and
$r\downarrow0$ with
\[
 \min_{i\in I}\dist(x_i^{(n)},\partial D)\ge d,
 \qquad \max_{i\in B}\dist(x_i^{(n)},\partial D)\le r.
\]
For an attempt cap $\ell$, put
\[
 t_j=r^{2^{-(j-1)}}\quad(1\le j\le\ell),\qquad t_0=0.
\]
For small $r$, $t_j-t_{j-1}\ge t_j/2$. Using the continuous label
martingales $M_i$ from Section~\ref{sec:fixed}, define
\[
 E_{\ell,r}=\left\{2(\ell+1)\max_i\sup_{s\le t_\ell}|M_i(s)|
                       +Ct_\ell<d/2\right\}.
\]
Up to $t_\ell\wedge\tau_\ell\wedge\theta\wedge\sigma$, each ancestral
displacement is bounded by the expression in braces. Consequently
\begin{equation}\label{eq:cascade-protected}
 \Prob(E_{\ell,r}^c)\le C_{\ell,d}t_\ell,
 \qquad E_{\ell,r}\Longrightarrow
       \sigma>t_\ell\wedge\tau_\ell\wedge\theta.
\end{equation}
Define
\[
 A_{j-1}=\{\tau_{j-1}\le t_{j-1},\
                \tau_{j-1}<\theta\wedge\sigma\}.
\]
On this event select an unassigned label measurably at $\tau_{j-1}$,
and denote its distance by $R_{j-1}$. Its ancestral root has initial
distance at most $r$. The Lipschitz distance function and BDG give
\begin{equation}\label{eq:cascade-wait}
 \begin{aligned}
 \E[\one_{A_{j-1}}R_{j-1}]
     &\le r+C_\ell\sqrt{t_{j-1}}+Ct_{j-1},\\
 \Prob(A_{j-1},\tau_j>t_j,\sigma>t_j)
     &\le \frac{C}{\sqrt{t_j}}
                       \E[\one_{A_{j-1}}R_{j-1}]\\
     &\le C_\ell r^{2^{-j}},\qquad 1\le j\le\ell.
 \end{aligned}
\end{equation}
The first bound uses at most $\ell+1$ ancestral martingale pieces.
For the second, apply \eqref{eq:wall} conditionally at $\tau_{j-1}$,
using $t_j-\tau_{j-1}\ge t_j/2$ on $A_{j-1}$.

Couple the ideal process to the actual victim order and marks until
$\theta\wedge\sigma$, then continue it with an arbitrary adapted order
and fresh marks. Let
\[
 G_{\ell,r}=\{J\le\ell,\ \tau_j\le t_j\ (1\le j\le J),\
                                      \sigma>\tau_J\}.
\]
On $G_{\ell,r}$ the assignments agree and $\theta=\tau_J\le t_\ell$.
By \eqref{eq:cascade-protected}, \eqref{eq:cascade-wait} and
\eqref{eq:restart-attempts},
\begin{equation}\label{eq:restart-time}
 \begin{aligned}
 \Prob(G_{\ell,r}^c)
 &\le\Prob(J>\ell)+\Prob(E_{\ell,r}^c)
       +\sum_{j=1}^\ell
          \Prob(A_{j-1},\tau_j>t_j,\sigma>t_j)\\
 &\le C_Ne^{-c_N\ell}+C_{\ell,d}r^{2^{-\ell}}.
 \end{aligned}
\end{equation}
On $G_{\ell,r}$, $\theta_n\le t_\ell$ and the actual assignment $a_n$
agrees with the ideal assignment. BDG on the ancestral pieces gives
\[
 \E\left[\one_{G_{\ell,r}}
       |X(\theta_n)-R_{a_n}x|^2\right]
 \le C_\ell\bigl(|x^{(n)}-x|^2+t_\ell\bigr).
\]
Together with \eqref{eq:restart-time}, first $n\to\infty$ and then
$\ell\to\infty$ give, for every $\varepsilon>0$,
\[
 \begin{aligned}
 &\Prob(\theta_n\ge\sigma_n)
  +\Prob\left(\theta_n<\sigma_n,\
       \theta_n+|X(\theta_n)-R_{a_n}x|>\varepsilon\right)\to0,\\
 &\Prob(\theta_n<\sigma_n,\ a_n=a)\to p(a).
 \end{aligned}
\]

Put $S_n=(\theta_n\wedge\sigma_n)\wedge T/2$ and
$G_n=\{\theta_n<\sigma_n,\ \theta_n<T/2\}\in\mathcal F_{S_n}$.
The cascade limits give
\[
 \begin{aligned}
 \Prob(G_n^c)&\to0,\\
 (S_n,X(S_n),X_N(S_n)-x_N^{(n)})
 &\Rightarrow\sum_a p(a)
       \delta_{(0,R_ax,(R_ax)_N-x_N)}.
 \end{aligned}
\]
Each $R_ax$ is interior. Lemma~\ref{lem:random-interior-restart}
therefore proves \eqref{eq:restart-endpoint}, retaining the cascade's
reference jump.
\end{proof}

\subsection{Uniform continuity on Foster sublevel sets}
\begin{lemma}\label{lem:sublevel-stability}
If $f$ is continuous on the interior hidden space and
$|f|\le C(1+\sqrt{\mathcal V})$, then for every finite $A,T>0$,
\begin{equation}\label{eq:reward-stability}
 \lim_{\eta\downarrow0}\ \sup_{\mathcal V(H(x))\le A}
 \left|\E_x^\eta[\Delta_T^{(1)}+f(H_T)]
              -\E_x^0[\Delta_T^{(1)}+f(H_T)]\right|=0.
\end{equation}
The supremum includes all allowed smooth coefficient systems and all
absolute lattice translates of the starting configuration.
\end{lemma}
\begin{proof}
Translate the initial reference position to a fundamental cell and
translate the coefficient field with it. This preserves the amplitude
bound even for nonperiodic coefficients. The canonical starting
configurations with $\mathcal V\le A$ have compact closure.
Moreover, $\Phi\le A$ implies
$\max_i\rho_i\ge R_0A^{-1/\gamma}$, so every limiting configuration has
an interior coordinate.

Along convergent subsequences, Lemma~\ref{lem:interior-prefix-continuity}
or Theorem~\ref{thm:polya-restart} identifies the same joint weak limit
for both systems. For each baseline component of this limit, atomless
Brownian exit waits, conditioning at numbered deaths and nonexplosion give
\[
 \Prob\{r_{\min}(H_T)=0\}
 \le\Prob\{\exists k:\tau_k=T\}
 \le\sum_{k\ge1}\Prob(\tau_k=T)=0,
\]
where $r_{\min}(H)=\min_i\rho_i$, extending $\rho$ by zero to $D^c$.

Portmanteau on the ambient hidden space therefore gives
\begin{equation}\label{eq:terminal-interior}
 \lim_{r\downarrow0}\limsup_{\eta\downarrow0}
 \sup_{\mathcal V\le A}\Prob_x^\eta(r_{\min}(H_T)\le r)=0.
\end{equation}
A violating sequence would put positive limiting mass on $r_{\min}=0$,
contradicting the displayed identity. The same proof covers the baseline.

Fix a sufficiently small amplitude bound $\eta_0>0$. The count moments,
ancestral envelopes and \eqref{eq:landing} give
\begin{equation}\label{eq:sublevel-ui}
 \sup_{\eta\le\eta_0,\ \mathcal V\le A}
 \E_x^\eta\bigl(|\Delta_T|^2+|f(H_T)|^2
                         +\mathfrak D(H_T)^2\bigr)\le C_{A,T,f}.
\end{equation}
Let $\chi_{r,R}$ be an ambient continuous cutoff, equal to one when
$r_{\min}\ge2r$ and $\mathfrak D\le R$, and zero when
$r_{\min}\le r$ or $\mathfrak D\ge2R$. Its support is an interior
compact hidden set. With
$[v]_K=\max(-K,\min(v,K))$, the retained reward
\[
 g_{r,R,K}(\xi,d)=[d^{(1)}]_K+\chi_{r,R}(\xi)f(\xi)
\]
is bounded and continuous on the ambient joint space, extending the
second term by zero. Cauchy--Schwarz and \eqref{eq:sublevel-ui} give
\begin{equation}\label{eq:sublevel-cutoff-error}
 \begin{aligned}
 &\left|\E_x^\eta[\Delta_T^{(1)}+f(H_T)]
           -\E_x^\eta g_{r,R,K}(H_T,\Delta_T)\right|\\
 &\quad\le\E_x^\eta[|\Delta_T|;|\Delta_T|>K]
        +\E_x^\eta[|f(H_T)|;r_{\min}(H_T)<2r]\\
 &\qquad{}+\E_x^\eta[|f(H_T)|;\mathfrak D(H_T)>R]\\
 &\quad\le C_{A,T,f}\left(K^{-1}+R^{-1}
           +\Prob_x^\eta\{r_{\min}(H_T)<2r\}^{1/2}\right).
 \end{aligned}
\end{equation}
For fixed $r,R,K$, the compact-subsequence argument already used for
the joint weak limit gives
\[
 \lim_{\eta\downarrow0}\sup_{\mathcal V\le A}
 \left|\E_x^\eta g_{r,R,K}(H_T,\Delta_T)
             -\E_x^0 g_{r,R,K}(H_T,\Delta_T)\right|=0.
\]
Apply \eqref{eq:sublevel-cutoff-error} to both systems, then let
$\eta\downarrow0$, $r\downarrow0$ and $R,K\uparrow\infty$.
Equation~\eqref{eq:terminal-interior} removes the wall term and proves
\eqref{eq:reward-stability}.
\end{proof}

\subsection{Stability of ballistic escape}\label{sec:stable-currents}
A positive stationary velocity yields persistent escape of the entire
population. The estimate below keeps the reference coordinate growing
linearly and the cloud diameter small in proportion, uniformly in time
and under the perturbations required by Kelvin inversion.

The Poisson equation turns the stationary mean into a constant expected
increment. A Foster penalty preserves positive drift at large
$\mathcal V$, and the block envelope carries the estimate between
observations. From a high translated seed, the resulting escape
probability tends to one.

For the translation-equivariant baseline kernel $P=P_T$ with invariant
law $\nu$, put
\[
 \begin{gathered}
 \mathcal C_{\mathcal V}
 =\left\{f\in C(H(D^N)):
       \sup_\xi\frac{|f(\xi)|}{1+\sqrt{\mathcal V(\xi)}}<\infty\right\},\\
 b(\xi)=\E_\xi^0\Delta_T^{(1)},\qquad m=\nu b=v_hT>0.
 \end{gathered}
\]
The block hypotheses are
\[
 \begin{gathered}
 P(\xi,\cdot)\ge\beta\mu(\cdot),\quad \beta>0,\qquad
 \nu P=\nu,\qquad \mathcal V^2\le C\mathcal W,\qquad \nu\mathcal W<\infty,\\
 b\in\mathcal C_{\mathcal V},\qquad
 P\mathcal C_{\mathcal V}\subset\mathcal C_{\mathcal V},
 \end{gathered}
\]
with probability measures $\mu,\nu$ and interior continuous, locally
bounded $\mathcal V,\mathcal W\ge1$. Estimates \eqref{eq:Foster},
\eqref{eq:block-envelope} and \eqref{eq:reward-stability} hold uniformly
over the perturbation class and absolute lattice translates, the last
for every $f\in\mathcal C_{\mathcal V}$. Translated seed neighborhoods
have bounded initial $\mathcal V,\mathcal W$. These conditions follow
from Proposition~\ref{prop:template-foster} and
Lemmas~\ref{lem:sublevel-stability} and
\ref{lem:interior-prefix-continuity} for the fixed domain.

The block envelope gives $|b(\xi)|^2\le C\mathcal V(\xi)$.
Let $(\Xi_n,\widetilde\Xi_n)$ be the minorization coupling of the
baseline block chains with initial laws $\delta_\xi$ and $\nu$, and
let $\tau_c$ be their meeting time. For some $q\in(0,1)$,
\[
 \Prob(\tau_c>n)\le q^n,\qquad
 \E\mathcal V(\Xi_n)\le\rho_V^n\mathcal V(\xi)
                                      +\frac{C_V}{1-\rho_V},\qquad
 \E\mathcal V(\widetilde\Xi_n)=\nu\mathcal V<\infty.
\]
Cauchy--Schwarz gives the summable estimate
\begin{equation}\label{eq:Poisson-tail}
 \begin{aligned}
 |P^nb(\xi)-m|
 &=\left|\E[(b(\Xi_n)-b(\widetilde\Xi_n))
                                   \one_{\{\tau_c>n\}}]\right|\\
 &\le q^{n/2}
         \left(2\E[|b(\Xi_n)|^2+|b(\widetilde\Xi_n)|^2]\right)^{1/2}\\
 &\le Cq^{n/2}
          \bigl(1+\rho_V^{n/2}\sqrt{\mathcal V(\xi)}\bigr).
 \end{aligned}
\end{equation}
Hence
\begin{equation}\label{eq:Poisson}
 u(\xi)=\sum_{n\ge0}(P^nb(\xi)-m),\qquad
 (I-P)u=b-m,\qquad |u(\xi)|\le C(1+\sqrt{\mathcal V(\xi)}).
\end{equation}
Local boundedness of $\mathcal V$ makes the series locally uniformly
summable, so kernel continuity gives interior continuity of $u$.
Moreover,
\[
 \sum_{n\ge0}P|P^nb-m|(\xi)
 \le C\sum_{n\ge0}q^{n/2}
       \bigl(1+\rho_V^{n/2}\sqrt{P\mathcal V(\xi)}\bigr)<\infty.
\]
Thus $Pu=\sum_{n\ge0}(P^{n+1}b-m)$, proving the identity in
\eqref{eq:Poisson}.

With $\Delta G=G(X_T)-G(x)$, the Poisson identity and Foster bounds give
\[
 \begin{aligned}
 G(x)&=x_N^{(1)}+u(H(x)),\\
 \E_x^0\Delta G&=b(H(x))+Pu(H(x))-u(H(x))=m,\\
 \E_x^\eta\Delta G&\ge-C_G(1+\sqrt{\mathcal V(H(x))}).
 \end{aligned}
\]
For a sublevel $A$, define the comparison error
\[
 \omega_A(\eta)=\sup_{\substack{\mathcal V(H(x))\le A\\
                         \text{amplitude}\le\eta}}
       |\E_x^\eta\Delta G-m|.
\]
Lemma~\ref{lem:sublevel-stability}, with $f=u$, gives
$\omega_A(\eta)\to0$ as $\eta\downarrow0$.

Choose the penalty and sublevel in the order
\begin{equation}\label{eq:drift-thresholds}
 \begin{gathered}
 0<a\le\frac{m}{8(C_V+1)},\\
 a(1-\rho_V)s-C_G(1+\sqrt s)-aC_V\ge m/4
                           \quad(s>A),\\
 \omega_A(\eta_*)\le m/4.
 \end{gathered}
\end{equation}
The second condition holds for a sufficiently large $A$, and the third
for a sufficiently small $\eta_*>0$. Set
\[
 F(x)=x_N^{(1)}+u(H(x))-a\mathcal V(H(x)).
\]
Write $\Delta F=F(X_T)-F(x)$ and $s=\mathcal V(H(x))$.
For every perturbation with amplitude at most $\eta_*$, Foster gives
\begin{equation}\label{eq:Poisson-Foster-drift}
 \begin{aligned}
 \E_x^\eta\Delta F
 &\ge\E_x^\eta\Delta G+a(1-\rho_V)s-aC_V\\
 &\ge
 \begin{cases}
  m-\omega_A(\eta_*)-aC_V\ge5m/8,&s\le A,\\
  a(1-\rho_V)s-C_G(1+\sqrt s)-aC_V\ge m/4,&s>A.
 \end{cases}
 \end{aligned}
\end{equation}
The linear Foster term dominates the square-root loss as
$\mathcal V\to\infty$. The block envelope and
$\mathcal V^2\le C\mathcal W$ give
\begin{equation}\label{eq:all-state-drift}
 \begin{aligned}
 \E_x^\eta\Delta F&\ge\kappa:=m/4>0,\\
 \E_x^\eta|\Delta F|^2
 &\le C\E_x^\eta\bigl[|\Delta_T|^2+1
       +\mathcal V(H_T)^2+\mathcal V(H(x))^2\bigr]\\
 &\le C\mathcal W(H(x)).
 \end{aligned}
\end{equation}
The threshold $\eta_*$ is fixed by the baseline block estimates, before
choosing the cloud ratio or seed height.

\begin{theorem}[Stability of a positive current]\label{thm:stable-current}
There are $\eta_*,c_*>0$ such that, for every $\delta>0$, there is
$C_\delta>0$ with the following property. From all particles at
$z+Lj e_1$, where $z\in D$ is fixed and $j$ is a sufficiently large
integer, every allowed perturbation of amplitude at most $\eta_*$
satisfies
\begin{equation}\label{eq:escape}
 \Prob\left\{\forall t\ge0:\ X_N^{(1)}(t)\ge c_*(j+t),\
                  \mathfrak D(X_t)\le\frac\delta4X_N^{(1)}(t)\right\}
       \ge1-C_\delta/j.
\end{equation}
The conclusion is uniform on a fixed sufficiently small product
neighborhood of each translated seed. More generally, it holds under
the block hypotheses above whenever $\nu b>0$ and the translated seeds
have initial reference coordinate $Lj+O(1)$.
\end{theorem}
\begin{proof}
Use the full-state block filtration and put
\[
 \begin{gathered}
 F_n=F(X_{nT}),\qquad
 A_n=\E(F_{n+1}-F_n\mid\mathcal F_{nT})\ge\kappa,\\
 M_n=F_n-F_0-\sum_{r<n}A_r.
 \end{gathered}
\]
Iteration of Foster and \eqref{eq:all-state-drift} yield
\[
 \begin{aligned}
 \sup_n\E\mathcal W(H_{nT})
 &\le\mathcal W(H_0)+\frac{C_W}{1-\rho_W}\le C,\\
 \E M_n^2&=\sum_{r<n}\E|M_{r+1}-M_r|^2\le Cn,
 \end{aligned}
\]
uniformly in the translated seeds. For $c>0$, Doob's inequality on
successively doubled intervals gives
\begin{equation}\label{eq:escape-martingale-error}
 \begin{aligned}
 \Prob\{\exists n:|M_n|>c(j+n)\}
 &\le\Prob\{\max_{n\le2j}|M_n|>cj\}
     +\sum_{r\ge1}\Prob\{\max_{n\le2^{r+1}j}|M_n|>c2^rj\}\\
 &\le\frac{C}{c^2j}
       +\sum_{r\ge1}\frac{C2^{r+1}j}{c^2(2^rj)^2}
 \le\frac{C_c}{j}.
 \end{aligned}
\end{equation}
Choose $c<\min(L,\kappa)/4$. On the complementary event,
\[
 F_n\ge F_0+\kappa n-c(j+n)
      \ge(L-c)j+(\kappa-c)n-C\ge a_0(j+n)
\]
for some $a_0>0$ and all sufficiently large $j$.

The penalty absorbs the corrector pointwise:
\begin{equation}\label{eq:completed-square}
 X_N^{(1)}=F-u+a\mathcal V
 \ge F-C-C\sqrt{\mathcal V}+a\mathcal V
 \ge F-C-\frac{C^2}{4a}.
\end{equation}
Consequently $X_N^{(1)}(nT)\ge a_1(j+n)$ for some $a_1>0$ and all
$n$, after increasing the minimum seed height.

For interpolation and cloud width, evaluate the block excursion on the
shifted path:
\[
 B_n=\mathfrak X_T\bigl((X_{nT+s})_{s\ge0}\bigr).
\]
The block envelope and iterated Foster bound give
\begin{equation}\label{eq:escape-block-error}
 \sup_n\E B_n^2\le C,\qquad
 \Prob\{\exists n:B_n>\varepsilon(j+n)\}
 \le\frac C{\varepsilon^2}\sum_{n\ge0}(j+n)^{-2}
 \le\frac C{\varepsilon^2j}.
\end{equation}
Take $\varepsilon\le\min(a_1/2,\delta a_1/8)$. Outside the two
exceptional events in \eqref{eq:escape-martingale-error} and
\eqref{eq:escape-block-error}, every $t\in[nT,(n+1)T]$ satisfies
\[
 \begin{aligned}
 X_N^{(1)}(t)&\ge a_1(j+n)-B_n
              \ge\tfrac12a_1(j+n)
              \ge\frac{a_1}{2(1+T)}(j+t),\\
 \mathfrak D(X_t)&\le B_n\le\varepsilon(j+n)
              \le\frac{2\varepsilon}{a_1}X_N^{(1)}(t)
              \le\frac\delta4X_N^{(1)}(t).
 \end{aligned}
\]
Their total exceptional probability is at most $C_\delta/j$, proving
\eqref{eq:escape} with $c_*=a_1/[2(1+T)]$, independent of $\delta$.
The same initial moment bounds and $F_0\ge Lj-C$ hold on the stated
product neighborhoods. Every conditional estimate uses the full real
configuration, so it applies to the nonperiodic perturbations as well.
\end{proof}

\begin{corollary}[Asymptotic escape from every state]\label{cor:asymptotic-escape}
Under the block hypotheses and perturbation threshold of
Theorem~\ref{thm:stable-current}, every interior initial configuration
satisfies
\[
 \liminf_{t\to\infty}\frac{X_N^{(1)}(t)}t\ge\frac\kappa T>0,
 \qquad \frac{\mathfrak D(X_t)}t\longrightarrow0
 \quad\text{almost surely}.
\]
In particular, the cloud diameter divided by the reference coordinate
tends to zero almost surely.
\end{corollary}
\begin{proof}
For each fixed interior $x$, Foster gives
$\sup_n\E_x\mathcal W(H_{nT})<\infty$.
Consequently, for every $\varepsilon>0$,
\[
 \sum_{n\ge1}\frac{\E_x|M_n-M_{n-1}|^2}{n^2}<\infty,
 \qquad
 \sum_{n\ge1}\Prob_x(B_n>\varepsilon n)
       \le\frac{C_x}{\varepsilon^2}\sum_{n\ge1}n^{-2}<\infty.
\]
The martingale series converges almost surely, and Kronecker's lemma
gives $M_n/n\to0$. Thus $\liminf_nF_n/n\ge\kappa$;
\eqref{eq:completed-square} transfers this to the reference coordinate.
Borel--Cantelli applied to positive rational $\varepsilon$ gives
$B_n/n\to0$. Interpolation over the blocks proves both assertions.
\end{proof}

\section{Kelvin inversion and explosion in bounded domains}\label{sec:kelvin}
\begin{theorem}[Common-clock conversion]\label{thm:kelvin-explosion}
Let $D$ be a periodic perforated domain of the form \eqref{eq:template} with
positive longitudinal stationary velocity. Suppose its small smooth
perturbations satisfy the block hypotheses of
Theorem~\ref{thm:stable-current}. For every sufficiently high cutting plane
$\{y^{(1)}=R\}$ missing all closed holes, the inverted domain
\[
 \Omega_R=\mathcal I(D\cap\{y^{(1)}>R\}),\qquad
 \mathcal I(y)=y/|y|^2,
\]
admits the Brownian explosion event \eqref{eq:main-event} from every
interior configuration before every positive time.
\end{theorem}

At each operational copy, the victim inherits the donor's position at
that same time. We use one physical clock for all labels and choose
operational coefficients whose inverted covariances coincide.
This preserves the copying genealogy under inversion. Stable escape
keeps the coefficients close to Brownian ones and makes the common
clock integrable.

\subsection{Diffusion coefficients for a common clock}
Let $Y$ be the auxiliary diffusion system, with operational time $u$
and independent label Brownian drivers $W_i$. On high, narrow
configurations define
\begin{equation}\label{eq:Kelvin-coefficients}
 \sigma_i^K(Y)=\frac{|Y_i|^2}{|Y_N|^2}I,\qquad
 b_i^K(Y)=-\frac{|Y_i|^2Y_i}{|Y_N|^4}.
\end{equation}
When $Y_N^{(1)}>R$ and $\mathfrak D(Y)\le\delta|Y_N|$, these satisfy
$\|\sigma_i^K-I\|\le C\delta$ and $|b_i^K|\le C/R$.

Put $Q(Y)=\sum_{i<N}|Y_i-Y_N|^2$. Choose flat smooth cutoffs
$0\le\chi_1,\chi_2\le1$, defining $\chi_2$ on $Y_N^{(1)}>R$, with
\[
 \begin{aligned}
 \chi_1&=0\quad(Y_N^{(1)}\le R),&
 \chi_1&=1\quad(Y_N^{(1)}\ge2R),\\
 \chi_2&=1\quad(Q\le\delta^2|Y_N|^2/16),&
 \chi_2&=0\quad(Q\ge\delta^2|Y_N|^2/4).
 \end{aligned}
\]
With $\chi=\chi_1\chi_2$, define
\begin{equation}\label{eq:Kelvin-extension}
 \sigma_i=I+\chi(\sigma_i^K-I),\qquad b_i=\chi b_i^K,
\end{equation}
using $(I,0)$ directly on $Y_N^{(1)}\le R$.
The cutoffs give a smooth global extension with
\[
 \sup_Y\max_i(\|\sigma_i-I\|+|b_i|)
 \le C(\delta+R^{-1})\le\eta_*,
\]
by choosing $\delta\le1$ small and then $R$ large. Here we used
$|Y_i-Y_N|\le\delta|Y_N|/2$ and $|Y_N|\ge R$ on
$\operatorname{supp}\chi$.

For the main construction take $D=D_h$ from
Corollary~\ref{cor:positive-template} and choose $R\in h\Z$.
The centers have first coordinates $h(k+1/2)$ and radii $a_z<h/4$,
so every such plane misses the closed holes. More generally, any cutting
plane in the theorem separates all holes into those strictly above and
those strictly below it. The cut domain
$U=D\cap\{y^{(1)}>R\}$ is connected: a finite polygonal path in the
half-space meets only finitely many balls, and can be modified around
them within the half-space. In the main construction write
\begin{equation}\label{eq:inverted-domain}
 \Omega_N=\mathcal I(D_h\cap\{y^{(1)}>R\})=\Omega_R.
\end{equation}
The periodic domain and cutting plane determine the domain.

Apply Theorem~\ref{thm:stable-current} with cloud ratio
$\delta/\sqrt N$, and denote the resulting constants by $c,C$.
For a sufficiently large seed $j$ with $cj>3R$, the event
\[
 \mathcal E_j=\left\{\forall u\ge0:\
       Y_N^{(1)}(u)\ge c(j+u),\quad
       \mathfrak D(Y_u)\le\frac{\delta}{4\sqrt N}Y_N^{(1)}(u)\right\}
\]
has probability at least $1-C/j$. On this event, for every operational
time,
\[
 Q\le\frac{N-1}{16N}\delta^2(Y_N^{(1)})^2
       <\delta^2|Y_N|^2/16,\qquad Y_N^{(1)}>3R.
\]
Moreover,
\begin{equation}\label{eq:all-labels-high}
 Y_i^{(1)}
 \ge(1-\delta/(4\sqrt N))Y_N^{(1)}>R.
\end{equation}
Thus all labels stay in $U$, the cutoffs equal one, and
\eqref{eq:Kelvin-coefficients} gives their exact coefficients throughout
$\mathcal E_j$.

\subsection{The stopped Brownian law}
Up to the exit $S$ specified below, define the common clock and inverted
positions by
\begin{equation}\label{eq:common-clock}
 t(u)=\int_0^u|Y_N(s)|^{-4}\,\dd s,\qquad
 Z_i(t(u))=\mathcal I(Y_i(u)).
\end{equation}
The inversion identities in dimension three are
\[
 D\mathcal I(y)=|y|^{-2}\mathcal R(y),\qquad
 \mathcal R(y)=I-2yy^{\mathsf T}/|y|^2,\qquad
 \Delta\mathcal I(y)=-2y/|y|^4.
\]
Since $\mathcal R\mathcal R^{\mathsf T}=I$, It\^o's formula in the
exact-coefficient region gives
\[
 D\mathcal I\,b_i^K=Y_i/|Y_N|^4,\qquad
 \tfrac12(\sigma_i^K(\sigma_i^K)^{\mathsf T}):D^2\mathcal I
       =-Y_i/|Y_N|^4.
\]
The drifts cancel and the label noises are
$\mathcal R(Y_i)|Y_N|^{-2}\,\dd W_i$. Inversion commutes with copying,
preserving each victim and donor under the common clock.

For this law identification, let $\mathcal F_u$ be the natural filtration
enlarged initially by the independent donor sequence
$\mathsf m_k\sim\mathrm{Unif}\{1,\ldots,N-1\}$, independent of $W$.
The original joint driver $W$ remains Brownian in this filtration.

Stop at the first exit $S$ from the high, narrow region
\[
 \mathcal C=\{Y\in U^N:Y_N^{(1)}>2R,\
                    Q<\delta^2|Y_N|^2/16\}.
\]
Normal hole-wall hits are followed by their copies as usual.
The stop includes a copy whose poststate leaves $\mathcal C$, or a hit
of the new half-space boundary. Before it, put
\[
 L_i(u)=\int_0^{u\wedge S}
       \frac{\mathcal R(Y_i(s-))}{|Y_N(s-)|^2}\,\dd W_i(s),
 \qquad A_u=t(u\wedge S).
\]
Predictable left-limit integrands give the joint covariance
\[
 \langle L_i,L_j\rangle_u=\delta_{ij}A_uI_3,\qquad
 \langle L\rangle_u=A_uI_{3N},\quad L=(L_i)_{i\le N}.
\]

Let $a_t=\inf\{u:A_u>t\}$, with empty infimum $\infty$, and use the
time-changed filtration $\mathcal G_t=\mathcal F_{a_t}$ with usual
augmentation and
$\mathcal F_\infty=\sigma(\bigcup_u\mathcal F_u)$.
For $\xi_m=S\wedge m$, continuity and strict increase before $S$ give
\[
 \{a_t<u\}=\{A_u>t\},\qquad
 \{A_{\xi_m}\le t\}=\{\xi_m\le a_t\}\in\mathcal F_{a_t}.
\]
Thus $\theta_m=A_{\xi_m}$ and $\tau^*=\lim_m\theta_m$ are physical
stopping times, with
\[
 \{a_t=\infty\}=\{\tau^*\le t\},\qquad
 A_{a_t\wedge m}\uparrow t\wedge\tau^*.
\]
For $n\ge m$,
\[
 \begin{aligned}
 \E|L(a_t\wedge m)|^2&=3N\,\E A_{a_t\wedge m}\le3Nt,\\
 \E|L(a_t\wedge n)-L(a_t\wedge m)|^2
 &=3N\,\E(A_{a_t\wedge n}-A_{a_t\wedge m})
       \xrightarrow[m,n\to\infty]{}0.
 \end{aligned}
\]
Localizing $L$ at bracket level $n$ gives
\[
 \sup_u\E\left|L\left(u\wedge\inf\{v:A_v\ge n\}\right)\right|^2
 \le3Nn.
\]
Martingale convergence therefore supplies
$L_\infty=\lim_{u\to\infty}L(u)$ on $\{\tau^*<\infty\}$, including
$S=\infty$ with a finite clock. Time change gives a continuous endpoint
extension, constant after $\tau^*$:
\[
 \begin{gathered}
 M_t=\lim_{m\to\infty}^{L^2}L(a_t\wedge m),\qquad
 M_t=L_\infty\quad(\tau^*\le t),\\
 \langle M\rangle_t=(t\wedge\tau^*)I_{3N}.
 \end{gathered}
\]
For $s<t$, $E\in\mathcal G_s$ and $m\ge k$,
$E\cap\{a_s\le k\}\in\mathcal F_{a_s\wedge m}$.
Bounded optional sampling and the $L^2$ limits give
\[
 \begin{aligned}
 \E\bigl[\one_{E\cap\{a_s\le k\}}
       (L(a_t\wedge m)-L(a_s\wedge m))\bigr]&=0,\\
 \E\bigl[\one_{E\cap\{a_s\le k\}}(M_t-M_s)\bigr]&=0
                                  &&(m\to\infty),\\
 \E\bigl[\one_{E\cap\{s<\tau^*\}}(M_t-M_s)\bigr]&=0
                                  &&(k\to\infty).
 \end{aligned}
\]
On $\{\tau^*\le s\}$, $M_t=M_s=L_\infty$. Hence
$\E(M_t\mid\mathcal G_s)=M_s$, including the case $a_s=\infty$.

Let $\widehat B$ be an independent $3N$-dimensional Brownian motion and
augment $\mathcal G_t\vee\sigma(\widehat B_r:r\le t)$. Set
\[
 \widetilde B_t
   =M_t+\int_0^t\one_{\{r>\tau^*\}}\,\dd\widehat B_r,\qquad
 \langle\widetilde B\rangle_t=tI_{3N}.
\]
L\'evy's characterization gives a joint Brownian driver in a filtration
whose initial sigma-field contains every donor mark. Successive
conditioning with the Brownian exponential martingale gives, for
$0=t_0<\cdots<t_q$, vectors
$\vartheta_r\in\R^{3N}$ and bounded $g(\mathsf m)$,
\begin{align}\label{eq:driver-marks}
 &\E\left[g(\mathsf m)\exp\left(i\sum_{r=1}^q
      \vartheta_r\cdot(\widetilde B_{t_r}-\widetilde B_{t_{r-1}})
                           \right)\right]\notag\\
 &\hspace{12mm}=\E g(\mathsf m)
       \exp\left(-\tfrac12\sum_{r=1}^q
                      |\vartheta_r|^2(t_r-t_{r-1})\right).
\end{align}
Thus the physical label drivers are independent of one another and of
the donor sequence. The earlier donor estimates retain their natural
filtration, revealing each mark at its death.

Lemma~\ref{lem:finite-prefix} identifies every finite prefix with
Brownian Fleming--Viot evolution in $\Omega_R$. Operational nonexplosion
gives bounded paths on $[0,S\wedge m]$; if $R_m$ bounds their norms,
\begin{equation}\label{eq:finite-clock}
 \begin{gathered}
 R_m^{-4}(u_2-u_1)\le t(u_2)-t(u_1)
       \le(2R)^{-4}(u_2-u_1),\\
 0\le u_1<u_2\le S\wedge m.
 \end{gathered}
\end{equation}
The clock is a continuous, strictly increasing bijection on this
interval. With operational copy times $u_k$ and physical times
$\tau_k=t(u_k)$ before $S$,
\[
 \#\{k:\tau_k<\theta_m\}
 =\#\{k:u_k<S\wedge m\}<\infty.
\]
Hence the finite-prefix identifications are consistent up to $\tau^*$.

Operational time is recovered from the physical stopped path by
\[
 u(t)=\int_0^t|Z_N(r)|^{-4}\,\dd r.
\]
Let $\mathcal E_{j,m}$ require $S>m$ and the escape inequalities up to
operational time $m$. The recovered clock makes these events measurable
from the stopped physical path. Therefore
\[
 \mathcal E_{j,m}\downarrow\mathcal E_j,\qquad
 \Prob^Z(\mathcal E_j)
 =\lim_{m\to\infty}\Prob^Y(\mathcal E_{j,m})
 =\Prob^Y(\mathcal E_j)\ge1-C/j>0.
\]

\Needspace{18\baselineskip}
\subsection{Explosion at the clock endpoint}
\begin{lemma}[Finite-clock criterion]\label{lem:finite-clock-criterion}
On any event $E\subset\{S=\infty\}$ for which
\[
 t_\infty=\int_0^\infty|Y_N(u)|^{-4}\,\dd u<\infty,
 \qquad |Y_i(u)|\to\infty\quad(i\in[N]),
\]
the Brownian system satisfies, almost surely,
\[
 \tau_\infty=t_\infty,\qquad
 \lim_{t\uparrow\tau_\infty}Z_i(t)=0\quad(i\in[N]).
\]
For $c,j>0$, the reference growth bound
$|Y_N(u)|\ge c(j+u)^\alpha$, $\alpha>1/4$, suffices for the
finite-clock condition, with
\[
 t_\infty\le\frac{c^{-4}}{4\alpha-1}j^{1-4\alpha}.
\]
\end{lemma}
\begin{proof}
Inversion gives $Z_i(t)\to0$ at the clock endpoint.
For three-dimensional Brownian motion the harmonic function $|x|^{-1}$
gives
\[
 \Prob_x(\tau_a<\infty)=a/|x|\longrightarrow0
       \quad(a\downarrow0,\ 0<a<|x|),
\]
where $\tau_a$ is the first hit of the ball of radius $a$. Thus a Brownian
segment started away from zero cannot reach that point.
If only finitely many copies occurred, after some fixed numbered copy
a free Brownian label segment would extend continuously to zero in
finite physical time. Applying point polarity after each numbered
copy, and also at time zero, excludes this by a countable union.

The infinitely many operational copy times tend to infinity by
nonexplosion. Their physical images therefore increase to exactly
$t_\infty$. Together with \eqref{eq:finite-clock}, this proves
$\tau_\infty=t_\infty$ on $E$. The growth bound gives
\[
 \int_0^\infty|Y_N(u)|^{-4}\,\dd u
 \le c^{-4}\int_0^\infty(j+u)^{-4\alpha}\,\dd u
 =\frac{c^{-4}}{4\alpha-1}j^{1-4\alpha}.
\]
\end{proof}

On the transferred escape event, \eqref{eq:all-labels-high} gives
$|Y_i(u)|\to\infty$ for every label, and
\begin{equation}\label{eq:finite-lifetime}
 t_\infty\le C\int_0^\infty(j+u)^{-4}\,\dd u
       =C'j^{-3}<\infty.
\end{equation}
Lemma~\ref{lem:finite-clock-criterion} therefore gives explosion at
zero at time $t_\infty$.

\begin{proposition}\label{prop:all-starts}
For each fixed inverted domain $\Omega_R$ above, and in particular
$\Omega_N$, the event \eqref{eq:main-event} has positive probability
from every interior configuration and before every $t_0>0$.
\end{proposition}
\begin{proof}
Write $\mathcal B_t=\{\tau_\infty<t,\ \lim_{r\uparrow\tau_\infty}Z(r)=0\}$.
Fix $z\in\Omega_R^N$ and put $s=t_0/2$. The inverted seed neighborhoods
of Theorem~\ref{thm:stable-current} and \eqref{eq:finite-lifetime} give
open sets $A_j\subset\Omega_R^N$ with
\[
 \inf_{w\in A_j}\Prob_w(\mathcal B_s)
 \ge p_j:=1-C/j>0,\qquad C'j^{-3}<s,
\]
for sufficiently large $j$.
Choose a product ball $\prod_i B_i\subset A_j$. Connectedness gives
polygonal paths from $z_i$ to $B_i$ with positive clearance from the
complement. Brownian tube probabilities and independence give
\[
 \begin{aligned}
 q_{z,j}&:=\Prob_z\{\tau_1>s,\ Z(s)\in A_j\}\\
 &\ge\prod_i\Prob_{z_i}\{\tau_{\Omega_R}>s,\ W(s)\in B_i\}>0,
 \end{aligned}
\]
where $\tau_{\Omega_R}$ is a single Brownian path's exit time.
The Markov property at $s$ now yields
\[
 \Prob_z(\mathcal B_{t_0})
 \ge\E_z\left[\one_{\{\tau_1>s,\ Z(s)\in A_j\}}
                   \Prob_{Z(s)}(\mathcal B_s)\right]
 \ge q_{z,j}p_j>0.
\]
\end{proof}

\subsection{Geometry of the explosion domain}\label{sec:geometry}
The construction produces a ball with spherical holes accumulating at
the explosion point, as shown in Figure~\ref{fig:inversion-geometry}.
Inversion maps the half-space
$\{y^{(1)}>R\}$ onto
\[
 B_R=B\left(\frac{e_1}{2R},\frac1{2R}\right).
\]
An original ball with center $z$ and radius $a_z$ lying above the cutting
plane maps to a ball with
\begin{equation}\label{eq:inverted-balls}
 c_z=\frac{z}{|z|^2-a_z^2},\qquad
 \rho_z=\frac{a_z}{|z|^2-a_z^2}.
\end{equation}
Since $R\in h\Z$ and $a_z<h/4$, these are exactly the balls with
$z^{(1)}>R$. Consequently
\begin{equation}\label{eq:explicit-domain}
 \Omega_N=B_R\setminus
  \bigcup_{\substack{z\in h(\Z^3+(1/2,1/2,1/2))\\z^{(1)}>R}}
                    \overline B(c_z,\rho_z).
\end{equation}
Inversion preserves disjointness and places each closed inner ball
strictly inside $B_R$. The radial bounds are
\[
 \frac1{|z|+a_z}\le|x|\le\frac1{|z|-a_z}
 \quad\bigl(x\in\partial B(c_z,\rho_z)\bigr).
\]
Thus a compact set away from zero meets only finitely many spheres,
and $|c_z|+\rho_z=(|z|-a_z)^{-1}\to0$. The sole accumulation point
is $0\in\partial B_R$, giving
\[
 \partial\Omega_N=\partial B_R\cup
                  \bigcup_{z^{(1)}>R}\partial B(c_z,\rho_z).
\]
The removed union is relatively closed in $B_R$, so $\Omega_N$ is open;
\eqref{eq:inverted-domain} gives boundedness and connectedness. Its
boundary is locally a separated sphere away from zero, hence smooth.

\begin{proposition}\label{prop:boundary-area}
The domain \eqref{eq:explicit-domain} has finite perimeter and satisfies
\eqref{eq:main-area}. Its total boundary area is
\begin{equation}\label{eq:exact-boundary-area}
 \HH^2(\partial\Omega_N)=\frac\pi{R^2}
       +4\pi\sum_{z^{(1)}>R}\rho_z^2<\infty.
\end{equation}
\end{proposition}
\begin{proof}
For the fixed mesh and potential, $0<a_-\le a_z\le a_+<\infty$.
For large $|z|$,
\[
 4\pi\rho_z^2\asymp |z|^{-4},
\]
with constants depending on that mesh and potential.
A lattice shell $s\le|z|\le2s$ contains at most $Cs^3$ centers, so
\[
 \sum_{\substack{|z|\ge s\\z^{(1)}>R}}4\pi\rho_z^2
 \le C\sum_{n\ge0}(2^ns)^{-1}\le C/s.
\]
The spheres are disjoint; their countable area sum gives
\eqref{eq:exact-boundary-area}.

The radial bounds imply
\[
 \partial B(c_z,\rho_z)\cap B(0,r)\ne\varnothing
 \ \Longrightarrow\ |z|\ge r^{-1}-a_+\ge(2r)^{-1}.
\]
Consequently, for small $r$,
\[
 \HH^2(\partial B_R\cap B(0,r))=\pi r^2,
\]
and hence
\[
 \HH^2(\partial\Omega_N\cap B(0,r))
 \le\pi r^2+\sum_{\substack{|z|\ge(2r)^{-1}\\z^{(1)}>R}}
                     4\pi\rho_z^2\le Cr.
\]

For the lower bound let $\mathcal Z_r$ be the lattice centers satisfying
\[
 2/r\le|z|\le4/r,\qquad z^{(1)}\ge|z|/2.
\]
For small $r$ this solid-angle shell lies above the cutting plane,
and every corresponding inverted sphere lies in $B(0,r)$.
Let $\Sigma=\{y:2\le|y|\le4,\ y^{(1)}\ge|y|/2\}$.
Only $O_h(r^{-2})$ lattice cells meet the sector boundary, so
\[
 \#\mathcal Z_r=h^{-3}|\Sigma|r^{-3}+O_h(r^{-2})\ge cr^{-3}.
\]
Their disjoint areas give
\[
 \HH^2(\partial\Omega_N\cap B(0,r))
 \ge\sum_{z\in\mathcal Z_r}4\pi\rho_z^2
 \ge cr^{-3}r^4=cr.
\]

Enumerate the inner balls and remove the first $m$ to obtain
$\Omega_{N,m}$. Then $\one_{\Omega_{N,m}}\to\one_{\Omega_N}$ in $L^1$,
and lower semicontinuity of perimeter gives
\[
 \operatorname{Per}(\Omega_N)
 \le\liminf_{m\to\infty}\operatorname{Per}(\Omega_{N,m})
 =\HH^2(\partial B_R)+\sum_{z^{(1)}>R}4\pi\rho_z^2<\infty.
\]
The point zero has zero $\HH^2$ measure, so these spheres account for
the whole boundary area.
\end{proof}

\begin{proof}[Proof of Theorem~\ref{thm:bounded-explosion}]
Fix $N\ge3$. The selected rate $k_N$ has positive longitudinal velocity
and a strict small-contrast margin. Corollary~\ref{cor:small-contrast-transfer}
gives a fixed mesh with $v_h>0$ and $a_z<h/4$.
The uniform estimates in the fixed domain and boundary restart comparison verify
the block hypotheses for this mesh. Choose
\[
 0<\delta\le1,\qquad R\in h\Z,\qquad
 C(\delta+R^{-1})\le\eta_*.
\]
Theorem~\ref{thm:kelvin-explosion} applies to the resulting domain
\eqref{eq:explicit-domain}, giving explosion at zero from every interior
configuration before every positive target time in this fixed domain.
Proposition~\ref{prop:boundary-area} supplies the boundary assertions.
\end{proof}

\begin{corollary}[An open family of constructions]\label{cor:open-family}
For every fixed $N\ge3$, there is an open uniform-norm neighborhood of
$k_N$ within the smooth periodic rates such that each rate in that
neighborhood selects a periodic perforated domain with positive
longitudinal current and hence a bounded Brownian Fleming--Viot
explosion domain with all properties in
Theorem~\ref{thm:bounded-explosion}.
\end{corollary}
\begin{proof}
Corollary~\ref{cor:soft-open-family} preserves the positive velocity.
The selected rate has a strict contrast margin, so shrink the
neighborhood to retain $k_{\max}/k_{\min}<\Theta$, where $\Theta$ is
fixed in \eqref{eq:ratio}.
For each rate, Corollary~\ref{cor:small-contrast-transfer} selects a
mesh with positive hard current. Its uniform estimates in the fixed domain and
boundary comparison give stable escape, and
Theorem~\ref{thm:kelvin-explosion} converts this to explosion.
The sphere formulas and area estimates above apply to each resulting
fixed mesh.
\end{proof}

\bmhead*{Supplementary material}
The supplementary material includes Lean code formalizing selected parts of the core theory.

\clearpage

\begingroup
\raggedright
\setlength{\bibsep}{6pt}
\def\encodingdefault{T1}
\def\rmdefault{lmr}
\def\sfdefault{lmss}
\def\ttdefault{lmtt}
\fontencoding{T1}\fontfamily{lmr}\selectfont
\def\PreBibitemsHook{\raggedright\interlinepenalty=10000}
\bibliography{references,preprints}
\endgroup

\clearpage

\appendix
\section*{Appendices}
\section{Exact algebra for the cubic current}\label{app:algebra}
The cubic coefficient in Theorem~\ref{thm:copying-current} follows from
a finite polynomial recurrence for every $N\ge3$. A neutral-resolvent
calculation gives a second derivation at $N=3$ and an explicit bound
for the fourth-order remainder.

\subsection{The general-population recurrence}\label{app:general-recurrence}

Retain first-order $q$-data in the dual polynomial ring $\mathbb Q[y][q]/(q^2)$,
$y=(N-1)^{-1}$. A tagged list
$\kappa=((a_j,m_j))_{j=1}^n$ represents
\[
 E_\kappa(q)=E_n(a_1q+m_1,\ldots,a_nq+m_n),\qquad
 a_j\in\{0,1\},\quad m_j\in\Z,\quad\sum_ja_j=1.
\]
Sort lists, delete $(0,0)$ entries, and collect equal lists. The root is
$\kappa_*=((1,0))$. Mode insertion is the operator $\cA_\ell$ in
\eqref{eq:mode}, with coefficients $y$ and $1-(n-1)y$; tagged exponents
add componentwise.

Let $\Pi$ retain the root coefficient and its root basis element, and
put $\Pi^\perp=\id-\Pi$. On non-root lists define the diagonal and merger operators
\[
 \begin{aligned}
 d_0&=\tfrac12\sum_jm_j^2+\tfrac12 n(n-1)>0,
 &d_1&=\sum_ja_jm_j,\\
 \mathcal D E_\kappa&=(d_0+d_1q)E_\kappa,\\
 \mathcal M E_\kappa&=\tfrac12\sum_{a\ne c}E_{\kappa^{a\to c}},
 \\
 \mathcal D^{-1}E_\kappa
   &=\left(d_0^{-1}-d_1d_0^{-2}q\right)E_\kappa.
 \end{aligned}
\]
Set $\mathcal D^{-1}E_{\kappa_*}=\mathcal M E_{\kappa_*}=0$.
The merger $\kappa^{a\to c}$ adds entry $a$ to entry $c$ and removes
entry $a$. Thus $\Pi^\perp\mathcal M\mathcal D^{-1}$ lowers rank. For a finite
expression $g$ of maximal rank $r$, triangular elimination is
\[
 \begin{aligned}
 \mathcal Rg&=\mathcal D^{-1}
       \sum_{j=0}^{r-1}(\Pi^\perp\mathcal M\mathcal D^{-1})^j\Pi^\perp g,\\
 \mathcal Sg\,E_1(q)&=\Pi(g+\mathcal M\mathcal Rg),\\
 (s_0-\cL_0)\mathcal Rg&=g-\mathcal Sg\,E_1(q).
 \end{aligned}
\]
The sum terminates because $(\Pi^\perp\mathcal M\mathcal D^{-1})^r\Pi^\perp g=0$.
Its coefficient division is exactly~\eqref{eq:dual}; applying insertion
and this elimination along a word gives~\eqref{eq:finite-recursion}.

The first insertion illustrates the elimination. For $\ell\ne0$,
\[
 \cA_\ell E_1(q)=E_2(\ell,q)-E_1(q+\ell).
\]
With $D_\ell=\ell^2/2+1$, elimination gives
\begin{equation}\label{ar:first-elimination}
 \begin{aligned}
 \mathcal R\cA_\ell E_1(q)
 &=\frac{E_2(\ell,q)}{D_\ell}
 -\frac{\ell^2E_1(q+\ell)}{2D_\ell(\ell^2/2+\ell q)},\\
 \mathcal S\cA_\ell E_1(q)&=0.
 \end{aligned}
\end{equation}
Differentiation gives the first corrector in~\eqref{eq:real-h1}.
Up to cubic order the maximal rank is four. Every rank-four coefficient
contains $(1-y)(1-2y)$, so at $N=3$ these zero terms are deleted before
the symmetric functions are interpreted. All divisions have $d_0>0$;
the resulting correctors are finite Fourier sums with linear coordinate
factors, giving~\eqref{eq:corrector-growth}.

\subsection{Ordered resonances and the cubic coefficient}\label{app:general-resonance}
Write $\rho_N(w)$ for the unweighted root derivative of an ordered word,
as in Lemma~\ref{lem:cubic}. For the quadratic pairs, the operations
above give
\[
 \begin{array}{c|c}
 (1,-1)&-(y+2)/3\\
 (-1,1)&(y+2)/3\\
 (2,-2)&-4(y+5)/75\\
 (-2,2)&4(y+5)/75
 \end{array}
\]
and thus exact cancellation of the two quadratic pairs. At cubic order,
the three distinct permutations of $(1,1,-2)$ give
\begin{center}
\begin{tabular}{cc}
\toprule
Ordered word $w$ & $\rho_N(w)$\\
\midrule
$(1,1,-2)$ & $(11y^2+35y-34)/150$\\
$(1,-2,1)$ & $4(17857+2466y-919y^2)/121275$\\
$(-2,1,1)$ & $2(170y^2-23571y-20258)/121275$\\
\bottomrule
\end{tabular}
\end{center}
The simultaneous sign reversal of all frequencies gives the opposite
root derivatives. For $\phi_*$, the common Fourier weight of the three
displayed words is $i/8$, and that of their negatives is $-i/8$.
Their combined contribution to $s_3'(0)$ is consequently $iW_N/4$.
The exact sum is
\begin{multline*}
 \frac14\left[\frac{11y^2+35y-34}{150}
 +\frac{4(17857+2466y-919y^2)}{121275}\right.\\
 \left.{}+\frac{2(170y^2-23571y-20258)}{121275}\right]
 =\frac{(1-y)(2282-3705y)}{323400}.
\end{multline*}
Thus
\begin{equation}\label{ar:general-cubic-certificate}
 \frac14W_N=\mathfrak c_N
 =\frac{(1-y)(2282-3705y)}{323400}.
\end{equation}
Substituting $y=(N-1)^{-1}$ gives the positive population factorization
in~\eqref{eq:population-current}. The polynomial identity holds for
every fixed $N\ge3$.

\subsection{Dependence on the neutral copying rate}
The neutral rate is part of the cubic response. If instead its per-particle value is $N-1$, each ordered merger has rate one. In the finite recurrence this changes the diagonal to $\frac12\sum_jm_j^2+n(n-1)$ and passes the full resolved coefficient to each merger. Repeating the displayed operations for $\cos s+\sin2s$ gives
\[
 \mathfrak c_N^{\rm alt}
 =\frac{(N-2)(2495N-8103)}{1445850(N-1)^2}.
\]
This coefficient is negative at $N=3$ and positive at every integer $N\ge4$. Replacing the second harmonic by $b_N\sin2s$, with $b_3=-1$ and $b_N=1$ for $N\ge4$, selects a positive cubic term for that neutral rate. The normalization~\eqref{eq:intro-rate} makes the common profile $\cos s+\sin2s$ positive for every $N\ge3$ and agrees with the three-particle neutral rate one.

\subsection{A resolvent derivation for three particles}
\label{app:three-particle-verification}
For the remainder of the appendix, $N=3$ and the neutral killing rate is
one per particle. Write $L_0$ for the neutral generator
in~\eqref{eq:normalized-generator}, and put $\mu_0=\nu_{3,0}$,
$Q=(x_1+x_2+x_3)/3$, and $S=\sum_{i<j}(x_i-x_j)^2$.
Here $x_i=X_i^{(1)}$, and $e_r$ denotes the $r$th coordinate vector
in the three-dimensional label space.
This derivation computes the cubic coefficient directly from the
neutral relative law and its resolvent.

\subsubsection{The neutral relative law}
Directly $L_0S=6-S$, and the moment argument of
Lemma~\ref{lem:soft-hidden-moments} justifies its Dynkin equation.
For $n\in\R^3$ with $\sum_a n_a=0$, the characteristic function is
\begin{equation}\label{eq:n3-neutral-characteristic}
 \mathcal H_0(n)=\mu_0e^{in\cdot x}
 =\frac{\sum_{a=1}^3(1+n_a^2)^{-1}}
        {3+\frac12\sum_{a=1}^3n_a^2},\qquad \sum_an_a=0.
\end{equation}
Put $\mathcal H_2(t)=\mu_0e^{it(x_1-x_2)}$ for $t\in\R$.
Exchangeability and the generator applied to this pair test give
$(1+t^2)\mathcal H_2(t)=1$. For the three-frequency test, each directed
copy merges two frequencies, giving two pair tests per remaining
coordinate with combined rate one. Stationarity yields
$(3+|n|^2/2)\mathcal H_0(n)=\sum_a\mathcal H_2(n_a)$ and hence
\eqref{eq:n3-neutral-characteristic}. Finite first relative moments
justify directional derivatives along $e_r-e_3$. Integer-frequency
moments of nonzero total frequency vanish by common translation invariance.

\subsubsection{Finite Fourier spaces and the neutral resolvent}
\label{resp:resolvent}

We use the hidden state~\eqref{eq:hidden-state}, the neutral generator $L_0$,
its invariant law $\mu_0$, and its semigroup $P_t^0$.
All functions below are invariant under
$x\mapsto x+2\pi(1,1,1)$. The relative coordinates retain their real
values.
For $n\in\mathbb Z^3$, set
\[
 E_{n,0}(x)=e^{\mathrm i n\cdot x},\qquad
 E_{n,r}(x)=\mathrm i(x_r-x_3)e^{\mathrm i n\cdot x},
 \quad r=1,2.
\]
For a finite set $\mathcal F\subset\mathbb Z^3$, put
\[
 \mathcal U(\mathcal F)=\operatorname{span}_{\mathbb C}
       \{E_{n,r}:n\in\mathcal F,\ r=0,1,2\}.
\]
The functions at
distinct frequencies, with their scalar and two linear coefficients,
are linearly independent: first separate the integer Fourier modes of
$x_3$, and then use independence of polynomial exponentials in the two
real relative coordinates.

Let $T_{ij}=C_{ij}^{\mathsf T}$, where $C_{ij}$ replaces coordinate $i$
by coordinate $j$. Explicitly,
\[
 (T_{ij}n)_i=0,\qquad (T_{ij}n)_j=n_i+n_j,\qquad
 (T_{ij}n)_\ell=n_\ell\quad(\ell\ne i,j).
\]
Copying reassigns each original frequency component to one label. Thus
the closure of $n$ is contained in
\[
 \left\{n^\alpha:
    n_j^\alpha=\sum_{\ell:\alpha(\ell)=j}n_\ell,
    \quad\alpha:[3]\to[3]\right\},
 \qquad |\{n^\alpha\}|\le3^3.
\]
We include zero in each finite closure, so that constants belong to its space.

\begin{lemma}[Finite neutral resolvent]\label{ar:finite-resolvent}
Suppose that $\mathcal F$ is finite, contains zero, and is closed under
all $T_{ij}$. Then $L_0$ preserves $\mathcal U(\mathcal F)$, its kernel
on this space consists precisely of constants, and
\[
 \operatorname{Ran}(-L_0|_{\mathcal U(\mathcal F)})
       =\{p\in\mathcal U(\mathcal F):\mu_0p=0\}.
\]
Consequently, for every $p\in\mathcal U(\mathcal F)$ there is a unique
$R_0p\in\mathcal U(\mathcal F)$ satisfying
\begin{equation}\label{ar:normalized-resolvent}
 -L_0R_0p=p-\mu_0p,\qquad \mu_0R_0p=0.
\end{equation}
This definition is unchanged upon enlarging the finite closed space.
The operator $R_0$ commutes with common translations, reflection, and
permutations of the labels.
\end{lemma}

\begin{proof}
Put $d_r=e_r-e_3$. Since $T_{ij}d_r$ has coordinate sum zero, define the
integer matrix $M_{ij}$ by
\[
 T_{ij}d_r=\sum_{q=1}^2(M_{ij})_{qr}d_q;
 \qquad (M_{ij})_{qr}=(T_{ij}d_r)_q.
\]
Direct differentiation and substitution give
\begin{align}
 \tfrac12\Delta E_{n,0}&=-\tfrac12|n|^2E_{n,0},
       \label{ar:neutral-basis}\\
 \tfrac12\Delta E_{n,r}
   &=-\tfrac12|n|^2E_{n,r}-(n_r-n_3)E_{n,0},\nonumber\\
 E_{n,0}(C_{ij}x)&=E_{T_{ij}n,0}(x),\nonumber\\
 E_{n,r}(C_{ij}x)
   &=\sum_{q=1}^2(M_{ij})_{qr}E_{T_{ij}n,q}(x).\nonumber
\end{align}
These formulas prove invariance.

For $u\in\mathcal U(\mathcal F)$, differentiation and copying give
\[
 \begin{aligned}
 |u|+|L_0u|&\le C_u(1+\sqrt S),\\
 \sum_a|\partial_a u|^2
       +\sum_{i\ne j}|u(C_{ij}x)-u(x)|^2
       &\le C_u(1+S),\qquad S(C_{ij}x)\le2S(x).
 \end{aligned}
\]
Let $N_T$ count neutral copies and
$U_T=\max_i\sup_{t\le T}|B_i(t)|$ for the longitudinal Brownian drivers.
Copying cannot increase their range, so
\[
 \begin{gathered}
 \sup_{t\le T}S(H_t)
 \le3[\mathfrak D(x)+4(N_T+1)U_T]^2,\\
 N_T\sim\mathrm{Poisson}(3T),\quad N_T\perp U_T,\qquad
 \E_h\sup_{t\le T}S(H_t)\le C_T(1+S(h)).
 \end{gathered}
\]
This bound removes localization from Dynkin's equation for $S$.
Solving $L_0S=6-S$ gives
\begin{equation}\label{ar:neutral-spread}
 \mathbb E_hS(H_t)=6+e^{-t}(S(h)-6),\qquad
 \mu_0S=6.
\end{equation}
The gradient and jump bounds now give integrable quadratic variation
on each finite interval, hence
\begin{equation}\label{ar:neutral-dynkin}
 P_t^0u(h)-u(h)=\int_0^tP_s^0L_0u(h)\,\dd s,
 \qquad \mu_0L_0u=0.
\end{equation}
Lemma~\ref{lem:soft-hidden-moments} gives total variation convergence.
For $R\ge1$, truncation and~\eqref{ar:neutral-spread} give
\[
 \begin{aligned}
 |P_t^0u(h)-\mu_0u|
 &\le C_u(1+R)\|P_t^0(h,\cdot)-\mu_0\|_{\TV}\\
 &\quad+C_u[P_t^0S(h)+\mu_0S](R^{-2}+R^{-1})
 \longrightarrow0
 \end{aligned}
\]
by taking $t\to\infty$ and then $R\to\infty$.
If $L_0u=0$, Dynkin's equation gives $u=P_t^0u\to\mu_0u$.
Thus $\ker L_0=\mathbb C1$. Put
$\mathcal U_0=\{u\in\mathcal U(\mathcal F):\mu_0u=0\}$. Then
\[
 \begin{gathered}
 \operatorname{Ran}(-L_0)\subset\mathcal U_0,\qquad
 \dim\operatorname{Ran}(-L_0)=\dim\mathcal U_0
                  =\dim\mathcal U(\mathcal F)-1,\\
 -L_0|_{\mathcal U_0}:\mathcal U_0\longrightarrow\mathcal U_0
                   \quad\text{is invertible}.
 \end{gathered}
\]
This proves~\eqref{ar:normalized-resolvent}. Solutions in two finite
closed spaces agree because their difference is centered and harmonic
in the finite closed sum. Each stated symmetry preserves $L_0$ and
$\mu_0$; uniqueness in a common closure gives commutation with $R_0$.
\end{proof}

Copying preserves the total frequency $\sum_jn_j$. On a closed space
consisting only of nonzero total frequencies, all neutral means vanish
by common translation invariance, and the preceding kernel argument
shows that $-L_0$ is directly invertible. This is the case for both
resolvents in each ordered resonance calculated below.

\begin{lemma}[Coefficient rules]\label{ar:coefficient-rules}
For a nonzero integer $b$, define
\[
 B_bf=\frac12\sum_{i\ne j}e^{\mathrm i bx_i}
                         \bigl(f(C_{ij}x)-f(x)\bigr),
 \qquad
 A_b=\frac16\sum_{i\ne j}e^{\mathrm i bx_i}(x_j-x_i).
\]
In addition to~\eqref{ar:neutral-basis}, the finite matrix rules are
\begin{align}
 B_bE_{n,0}
  &=\frac12\sum_{i\ne j}
       (E_{T_{ij}n+be_i,0}-E_{n+be_i,0}),\label{ar:mode-rules}\\
 B_bE_{n,r}
  &=\frac12\sum_{i\ne j}
       \left(\sum_{q=1}^2(M_{ij})_{qr}E_{T_{ij}n+be_i,q}
                         -E_{n+be_i,r}\right).\nonumber
\end{align}
Write $[p]_{n,r}$ for the coefficient of $E_{n,r}$ in a basis expansion
of $p$. The only nonzero coefficients of $\mathrm i A_b$ are
\begin{equation}\label{ar:source-coefficients}
 [\mathrm i A_b]_{be_i,r}
       =\frac16(1-3\one_{\{i=r\}}),\qquad r=1,2.
\end{equation}
For integer $n$ of coordinate sum zero, put
$D(n)=3+|n|^2/2$ and $U(n)=\sum_j(1+n_j^2)^{-1}$. Then
\begin{align}
 \mu_0E_{n,0}&=\mathcal H_0(n)=\frac{U(n)}{D(n)},
       \label{ar:neutral-expectations}\\
 \mu_0E_{n,r}
  &=\frac{D(n)\left[-2n_r/(1+n_r^2)^2
                         +2n_3/(1+n_3^2)^2\right]
                  -U(n)(n_r-n_3)}{D(n)^2}.\nonumber
\end{align}
Both means are zero when $\sum_jn_j\ne0$. Thus the generator matrices,
source vectors, neutral expectations, and normalized resolvents have
rational entries in these bases.
\end{lemma}

\begin{proof}
The factor $e^{\mathrm i bx_i}$ is evaluated at the position before
copying, so it shifts the gain frequency by $be_i$ after $T_{ij}$ has
acted. This gives~\eqref{ar:mode-rules}. Summing the two donor
differences for each victim gives~\eqref{ar:source-coefficients}.
The expectation formulas follow from the calculation of~\eqref{eq:n3-neutral-characteristic} by
differentiating $U/D$ along $e_r-e_3$. These directions have coordinate
sum zero, as required by that characteristic function. The vanishing
outside total frequency zero follows from common translation invariance.
Finally, inversion of the centered finite rational matrix gives a
rational resolvent; its mean normalization also uses only the displayed
rational moments.
\end{proof}

For a real mean-zero trigonometric polynomial
$\phi=\sum_{a\in\mathcal F}\widehat\phi_a e^{\mathrm i ax}$,
$\mathcal F\subset\Z\setminus\{0\}$ finite, define
\[
 \cA_\phi f=\frac12\sum_{i\ne j}\phi(x_i)
                       [f(C_{ij}x)-f(x)],\qquad
 \cL_{\epsilon,\phi}=L_0+\epsilon\cA_\phi.
\]
Write $\nu_{\epsilon,\phi}$ for the hidden invariant law of the rate
$1+\epsilon\phi$, and $v_{3,\epsilon,\phi}$ for its longitudinal velocity.

\begin{proposition}[Finite-profile response]\label{ar:finite-profile-response}
For $|\epsilon|\|\phi\|_\infty\le1/2$, the centered neutral correctors
\begin{equation}\label{ar:three-particle-correctors}
 \begin{aligned}
 h_1^\phi&=R_0(\cA_\phi Q),&
 h_2^\phi&=R_0(\cA_\phi h_1^\phi),\\
 \mathcal C_3(\phi)&=\mu_0\cA_\phi h_2^\phi,&
 h_3^\phi&=R_0(\cA_\phi h_2^\phi).
 \end{aligned}
\end{equation}
exist as real finite Fourier-linear functions. The first two response
coefficients vanish, and
\begin{equation}\label{ar:finite-profile-expansion}
 \begin{gathered}
 \mu_0\cA_\phi Q=\mu_0\cA_\phi h_1^\phi=0,\\
 v_{3,\epsilon,\phi}
   =\epsilon^3\mathcal C_3(\phi)
       +\epsilon^4\nu_{\epsilon,\phi}\cA_\phi h_3^\phi,\\
 |\nu_{\epsilon,\phi}\cA_\phi h_3^\phi|\le M_\phi<\infty.
 \end{gathered}
\end{equation}
In terms of the single-mode operators and sources in
Lemma~\ref{ar:coefficient-rules},
\begin{equation}\label{ar:finite-profile-triads}
 \mathcal C_3(\phi)
 =\sum_{\substack{a,b,c\in\mathcal F\\a+b+c=0}}
 \widehat\phi_a\widehat\phi_b\widehat\phi_c\,
       \mu_0B_cR_0B_bR_0A_a.
\end{equation}
\end{proposition}

\begin{proof}
The source $\cA_\phi Q=\sum_a\widehat\phi_aA_a$ is hidden.
Frequency shifts followed by copying closures remain finite, so
Lemma~\ref{ar:finite-resolvent} supplies each corrector.
Real inputs have real centered solutions by uniqueness.
Nonzero total frequencies have zero neutral means, giving
$\mu_0\cA_\phi Q=0$.
Put $\mathcal C_2(\phi)=\mu_0\cA_\phi R_0\cA_\phi Q$.
Commutation of $R_0$ with translations and reflection gives
\[
 \mathcal C_2(\phi(\cdot+\theta))=\mathcal C_2(\phi),\qquad
 \mathcal C_2(\phi(-\cdot))=-\mathcal C_2(\phi).
\]
Translation restricts this quadratic Fourier polynomial to products
$\widehat\phi_a\widehat\phi_{-a}=|\widehat\phi_a|^2$.
Reflection fixes every such product; hence $\mathcal C_2=0$.
The resolvent equations are therefore
\[
 \begin{gathered}
 L_0h_1^\phi=-\cA_\phi Q,\qquad
 L_0h_2^\phi=-\cA_\phi h_1^\phi,\\
 L_0h_3^\phi=\mathcal C_3(\phi)-\cA_\phi h_2^\phi.
 \end{gathered}
\]
Finite Fourier-linear functions, their first two derivatives and copying
increments are bounded by $C_\phi(1+\sqrt S)$.
Lemma~\ref{lem:soft-hidden-moments}, for rates in $[1/2,3/2]$,
thus justifies stationary Dynkin equations and gives
\[
 \begin{aligned}
 \nu_{\epsilon,\phi}\cA_\phi Q
   &=\epsilon\nu_{\epsilon,\phi}\cA_\phi h_1^\phi,\\
 \nu_{\epsilon,\phi}\cA_\phi h_1^\phi
   &=\epsilon\nu_{\epsilon,\phi}\cA_\phi h_2^\phi,\\
 \nu_{\epsilon,\phi}\cA_\phi h_2^\phi
   &=\mathcal C_3(\phi)
       +\epsilon\nu_{\epsilon,\phi}\cA_\phi h_3^\phi.
 \end{aligned}
\]
Since $v_{3,\epsilon,\phi}=\epsilon\nu_{\epsilon,\phi}\cA_\phi Q$,
this proves~\eqref{ar:finite-profile-expansion} and its uniform remainder.
Indeed $\sqrt S\le\sqrt3\,\mathfrak D$ and
\[
 |\nu_{\epsilon,\phi}\cA_\phi h_3^\phi|
 \le C_\phi\left(1+\sqrt3
       \sup_{|\eta|\|\phi\|_\infty\le1/2}\nu_{\eta,\phi}\mathfrak D\right)
 =:M_\phi<\infty.
\]
Expanding the three perturbations gives~\eqref{ar:finite-profile-triads}:
copying preserves total frequency, while each $B_a$ adds $a$;
only $a+b+c=0$ survives neutral averaging.
\end{proof}

\subsubsection{A five-dimensional calculation of the resonant coefficient}
\label{resp:chirality}

The ordered contributions can be calculated in the label-symmetric
subspace. This reduces the first resolvent to dimension two and the
second to dimension at most five. For nonzero integers $a,b$, define
\begin{align*}
 F_a&=\sum_i e^{\mathrm i ax_i},&
 Z_a&=\sum_i\mathrm i(x_i-Q)e^{\mathrm i ax_i},\\
 T_{a,b}&=\sum_{i\ne j}e^{\mathrm i(ax_i+bx_j)},&
 U_{a,b}&=\sum_{i\ne j}\mathrm i(x_i-Q)e^{\mathrm i(ax_i+bx_j)},\\
 V_{a,b}&=\sum_{i\ne j}\mathrm i(x_j-Q)e^{\mathrm i(ax_i+bx_j)}.
\end{align*}
Thus $T_{a,b}=T_{b,a}$ and $V_{a,b}=U_{b,a}$; when $a=b$ the last
two functions coincide. These functions are defined on the hidden space,
since every linear factor is relative to $Q$. The source obeys
$\mathrm i A_a=-Z_a/2$.

Diffusion and copying act on the first two functions as
\begin{equation}\label{ar:first-system}
 L_0F_a=-\frac{a^2}{2}F_a,\qquad
 L_0Z_a=-\frac{a^2+2}{2}Z_a-\frac{2a}{3}F_a.
\end{equation}
The diffusion cross term is $-(2a/3)F_a$ and the net copying term is
$-Z_a$. Since $a\ne0$, both means vanish, and~\eqref{ar:first-system} gives
\begin{equation}\label{ar:first-resolvent}
 R_0(\mathrm i A_a)=p_aF_a+q_aZ_a,\qquad
 p_a=\frac4{3a(a^2+2)},\qquad q_a=-\frac1{a^2+2}.
\end{equation}

For $s=a+b\ne0$, set
\[
 \alpha=\frac{s^2}{2},\qquad \beta=\frac{s^2+2}{2},\qquad
 j=\frac{a^2+b^2+2}{2},\qquad k=j+1.
\]
Put $\boldsymbol\psi_{a,b}=(F_s,Z_s,T_{a,b},U_{a,b},V_{a,b})^{\mathsf T}$.
The basis rules~\eqref{ar:neutral-basis} give
\begin{equation}\label{ar:second-system}
 \begin{gathered}
 L_0\boldsymbol\psi_{a,b}=-G_{a,b}^{\mathsf T}\boldsymbol\psi_{a,b},\\
 G_{a,b}=\begin{pmatrix}
 \alpha&2s/3&-2&0&0\\
 0&\beta&0&-1&-1\\
 0&0&j&(2a-b)/3&(2b-a)/3\\
 0&0&0&k&1\\
 0&0&0&1&k
 \end{pmatrix}.
 \end{gathered}
\end{equation}
Mutual copies merge $a,b$; the remaining copies preserve their multiset.
Diffusion contributes the three off-diagonal frequency terms in $G_{a,b}$.
The next perturbation has source
\begin{equation}\label{ar:second-source}
 B_bF_a=-F_s+\frac12T_{a,b},\qquad
 B_bZ_a=-Z_s+\frac13U_{a,b}+\frac23V_{a,b}.
\end{equation}

Since $s\ne0$ and $k>1$, the resolvent coefficients are
\begin{equation}\label{ar:back-substitution}
 \begin{gathered}
 \mathbf b_{a,b}=(-p_a,-q_a,p_a/2,q_a/3,2q_a/3)^{\mathsf T},\\
 \mathbf w_{a,b}=(f,z,t,u,v)^{\mathsf T}
       =G_{a,b}^{-1}\mathbf b_{a,b},\\
 R_0B_bR_0(\mathrm i A_a)
       =\boldsymbol\psi_{a,b}^{\mathsf T}\mathbf w_{a,b},\qquad
 \det G_{a,b}=\alpha\beta j(k^2-1)>0.
 \end{gathered}
\end{equation}
When $a=b$, $U_{a,a}=V_{a,a}$; the redundant representation reduces by
\begin{equation}\label{ar:equal-frequency}
 P(f,z,t,u,v)^{\mathsf T}=(f,z,t,u+v)^{\mathsf T},\qquad
 u+v=\frac{q_a}{k+1}.
\end{equation}
The matrix is invertible even in this redundant representation; $P\mathbf w_{a,a}$
gives the unique coefficients in the four-function basis.

Let $\delta_r=e_r-(1,1,1)/3$. For $c=-s$, define the terminal row and
its pair and triple moments by
\[
 \begin{gathered}
 \boldsymbol\lambda_{a,b}=(\lambda_F,\lambda_Z,\lambda_T,\lambda_U,\lambda_V)
 =\mu_0B_{-s}\boldsymbol\psi_{a,b}^{\mathsf T},\\
 h_r=\frac1{1+r^2},\quad g_r=\frac{2r}{(1+r^2)^2},\quad
 H=\mathcal H_0(a,b,-s),\quad
 D_r=\partial_{\delta_r}\mathcal H_0(a,b,-s).
 \end{gathered}
\]
The coefficient rules~\eqref{ar:mode-rules} give
\begin{equation}\label{ar:terminal-row}
 \begin{aligned}
 \lambda_F&=-3s^2h_s,&\lambda_Z&=g_s,\\
 \lambda_T&=6(h_s+H-h_a-h_b),\\
 \lambda_U&=-g_s+3(g_a-g_b)-6D_2,\\
 \lambda_V&=-g_s+3(g_b-g_a)-6D_1.
 \end{aligned}
\end{equation}
For $a=b$, $\lambda_U=\lambda_V$, so
$\boldsymbol\lambda_{a,a}\mathbf w_{a,a}
 =\boldsymbol\lambda_{a,a}^{\rm red}P\mathbf w_{a,a}$ with
$\boldsymbol\lambda_{a,a}^{\rm red}=(\lambda_F,\lambda_Z,\lambda_T,\lambda_U)$.
For the first two entries use~\eqref{ar:second-source} at $(s,-s)$ and
\begin{align*}
 F_0&=3,&Z_0&=0,&\mathcal H_0(r,-r,0)&=h_r,\\
 \partial_{\delta_1}\mathcal H_0(r,-r,0)
       &=-g_r/2,&
 \partial_{\delta_2}\mathcal H_0(r,-r,0)
       &=g_r/2.
\end{align*}
For $T,U$, group the copies by the occupied pair and its complement.
The pair gain, pair losses, and triple terms give, per ordered pair,
\[
 \lambda_T=3(2h_s-2h_a-2h_b+2H),\qquad
 \lambda_U=3(-g_s/3+g_a-g_b-2D_2).
\]
The factor $3$ is six ordered pairs times rate $1/2$;
interchanging $a,b$ gives $\lambda_V$. All derivatives are rational by
\[
 \partial_d\mathcal H_0(n)
 =\frac{D(n)\sum_r[-2n_rd_r/(1+n_r^2)^2]
                    -U(n)(n\cdot d)}{D(n)^2},
 \qquad \sum_rd_r=0.
\]

\begin{proposition}[The resonant constant]\label{ar:resonance}
For a real profile with frequencies in $\{-2,-1,1,2\}$,
\begin{equation}\label{ar:cubic-chirality}
 \mathcal C_3(\phi)
   =\kappa\,\operatorname{Im}
              (\widehat\phi_1^2\overline{\widehat\phi_2}),
 \qquad \kappa=\frac{859}{161700}>0.
\end{equation}
\end{proposition}

\begin{proof}
Define the ordered neutral contribution
\[
 \Gamma_{a,b,c}
   =\mu_0B_cR_0B_bR_0(\mathrm i A_a),\qquad
 \Gamma_{a,b,-a-b}=\boldsymbol\lambda_{a,b}\mathbf w_{a,b}.
\]
For the three permutations of $(1,1,-2)$, the partial sums $a$ and
$a+b$ are nonzero, so the preceding inverses require no mean
subtraction. Matrix inversion~\eqref{ar:back-substitution}, followed by
the projection~\eqref{ar:equal-frequency} when $a=b$, gives the vectors below.
For the first vector the coordinates are $(f,z,t,u+v)$; for the others
they are $(f,z,t,u,v)$:
\begin{align*}
 P\mathbf w_{1,1}
   &=\left(-\frac{11}{72},\frac1{12},\frac18,-\frac1{12}\right),\\
 \mathbf w_{1,-2}
   &=\left(-\frac{2224}{4851},\frac2{11},\frac{76}{1617},
                         -\frac{10}{693},-\frac{32}{693}\right),\\
 \mathbf w_{-2,1}
   &=\left(\frac{1462}{4851},\frac1{11},-\frac{17}{1617},
                         -\frac5{693},-\frac{16}{693}\right).
\end{align*}
The expectation rows~\eqref{ar:terminal-row}, in the same coordinates,
are
\begin{align*}
 \boldsymbol\lambda_{1,1}^{\rm red}
   &=\left(-\frac{12}{5},\frac4{25},-\frac{18}{5},\frac{13}{50}\right),\\
 \boldsymbol\lambda_{1,-2}
   &=\left(-\frac32,-\frac12,0,\frac{41}{25},-\frac{53}{50}\right),\\
 \boldsymbol\lambda_{-2,1}
   &=\left(-\frac32,-\frac12,0,-\frac{53}{50},\frac{41}{25}\right).
\end{align*}
Their scalar products yield
\begin{equation}\label{ar:ordered-contributions}
 \Gamma_{1,1,-2}=-\frac{11}{120},\qquad
 \Gamma_{1,-2,1}=\frac{25147}{40425},\qquad
 \Gamma_{-2,1,1}=-\frac{21334}{40425}.
\end{equation}
Thus the sum of these three ordered contributions is
\begin{equation}\label{ar:positive-sum}
 W=-\frac{11}{120}+\frac{25147-21334}{40425}
   =-\frac{11}{120}+\frac{1271}{13475}
   =\frac{859}{323400}>0.
\end{equation}
The integer identity
$120\cdot1271-11\cdot13475=4295$ verifies its positive sign exactly.

The factor $-\mathrm i$ comes from $A_a=-\mathrm i(\mathrm i A_a)$.
Frequency selection leaves
only the three words just evaluated and their negatives. Define the
reflection operator by $\sigma f(x)=f(-x)$. It
gives
$\Gamma_{-a,-b,-c}=-\Gamma_{a,b,c}$: indeed
$\mathrm i A_{-a}=-\sigma(\mathrm i A_a)$,
$B_{-b}\sigma=\sigma B_b$, and $R_0$ commutes with $\sigma$.
Putting $z=\widehat\phi_1^2\widehat\phi_{-2}$, the two sets of words
therefore contribute
\[
 \mathcal C_3(\phi)=-\mathrm iWz+\mathrm iW\overline z
                 =2W\,\operatorname{Im}z.
\]
Since $\widehat\phi_{-2}=\overline{\widehat\phi_2}$, this proves
\eqref{ar:cubic-chirality} with $\kappa=2W$.
\end{proof}

For the chosen profile $\phi_*(s)=\cos s+\sin2s$, the phase parameters
are $(a,b,\alpha,\beta)=(1,1,0,-\pi/2)$. Thus
$\mathcal C_3(\phi_*)=\kappa/8=\mathfrak c_3>0$.
More generally,~\eqref{ar:cubic-chirality} and
\eqref{eq:chiral-invariant} give
$\mathcal C_3(\phi)=\mathfrak c_3\mathcal C(\phi)$, in agreement with
the general-population calculation.

\subsection{The fourth-order remainder for three particles}
\label{app:quantitative-three-certificate}
For $N=3$, retain the profile $\phi_*$ and put $D_4=\cA h_3$.
Using the basis~\eqref{ar:neutral-basis}, write
\[
 D_4=\sum_n\left(s_nE_{n,0}+\sum_{r=1}^2t_{n,r}E_{n,r}\right),\qquad
 \|z\|_1=|\operatorname{Re}z|+|\operatorname{Im}z|,
\]
and define its scalar and relative coefficient norms by
\[
 N_0=\sum_n\|s_n\|_1,\qquad
 N_1=\sum_n\sum_{r=1}^2\|t_{n,r}\|_1.
\]
Successive application of the labelled coefficient rules
\eqref{ar:neutral-basis} and~\eqref{ar:mode-rules} gives the rational bounds
\begin{equation}\label{ar:remainder-norm-certificate}
 N_0<\frac{14}{5},\qquad N_1<\frac{31}{20},\qquad N_0+4N_1<9.
\end{equation}
Constants added to a corrector leave $D_4$ unchanged, since $\cA1=0$.

For $S=\sum_{i<j}(x_i-x_j)^2$, each directed copy obeys
$|S(C_{ij}x)-S(x)|\le S(x)$. Hence
\[
 \begin{gathered}
 \cL_0S=6-S,\qquad |\cA S|\le6S,\\
 \cL_\epsilon S\le6-(1-6|\epsilon|)S,\qquad
 |\epsilon|\le1/12\quad\Longrightarrow\quad\nu_{3,\epsilon}S\le12.
 \end{gathered}
\]
The stationary Dynkin equation is justified by
Lemma~\ref{lem:soft-hidden-moments}. Since $|x_r-x_3|\le\sqrt S$,
\begin{equation}\label{ar:quantitative-three-remainder}
 \begin{aligned}
 |D_4|&\le N_0+N_1\sqrt S,\\
 |\nu_{3,\epsilon}D_4|
 &\le N_0+N_1\sqrt{\nu_{3,\epsilon}S}
 \le N_0+\sqrt{12}\,N_1<N_0+4N_1<9.
 \end{aligned}
\end{equation}
This supplies the remainder bound in
Corollary~\ref{cor:quantitative-three-current}.

\section{Capacity and marked copying operators}\label{app:capacity}
The capacity lower bound must retain two energies: the free gradient
energy carried outside the microscopic shells and the absorption energy
concentrated inside them. We use the shell radius $\beta_h=h^2$ and the
cell comparison radius $\rho_h=h/4$.
\subsection{Convergence of Dirichlet forms and resolvents}

\begin{lemma}[Forms and heat operators]\label{b:capacity}
Let $A_h$ be the Dirichlet operator of $-\frac12\Delta$ on $L^2(M_h)$, and extend its resolvent and killed heat operators by zero to $L^2(M)$.  Let $A_k=-\frac12\Delta+k$ on $M$.  The Dirichlet forms converge in the Mosco sense to~\eqref{eq:effective-forms}.  For every $\lambda,t>0$,
\begin{equation}\label{b:operator-convergence}
 \|(A_h+\lambda)^{-1}_{0}-(A_k+\lambda)^{-1}\|_{2\to2}\to0,
 \qquad \|P_t^h-P_t^k\|_{2\to2}\to0.
\end{equation}
The subscript $0$ means that the Dirichlet resolvent is extended by zero on the holes.
\end{lemma}
\begin{proof}
The common form domain for the $h$th form is the set of $H^1(M)$ functions vanishing almost everywhere on the holes, with value $+\infty$ assigned to other $L^2$ functions.  Its finite energy is $\mathcal E_h(u)=\frac12\int|\nabla u|^2$.  Let $Q_z$ be a mesh cube, $\bar u_z$ its volume mean, and $m_z$ the surface mean on $r=\beta_h$.

The derivative of the radial mean and weighted Cauchy--Schwarz give
\[
 |m_{\beta_h}-m_{h/4}|^2
 \le C\left(\int_{\beta_h}^{h/4}r^{-2}\,dr\right)
        \int_{Q_z}|\nabla u|^2
 \le C\beta_h^{-1}\int_{Q_z}|\nabla u|^2.
\]
Scaling the cube to a unit cube, the Poincar\'e inequality and the trace estimate on its fixed interior sphere give
$|m_{h/4}-\bar u_z|^2\le Ch^{-1}\int_{Q_z}|\nabla u|^2$.  Consequently
\begin{equation}\label{b:trace-means}
 \sum_z h^3|m_z-\bar u_z|^2
 \le Ch\int_{M}|\nabla u|^2.
\end{equation}
The trace theorem makes these means well defined for every zero extended
form function in $H^1$.

The radial minimizer with values zero and one on the inner and outer shell spheres is
\[
 w_z(r)=\frac{1-a_z/r}{1-a_z/\beta_h},\qquad a_z<r<\beta_h.
\]
Its energy is $c_z=2\pi a_z/(1-a_z/\beta_h)$. For a general shell
function, spherical averaging decreases radial energy and removes
nonnegative tangential energy. The averaged function is zero at $a_z$
and has value $m_z$ at $\beta_h$. The radial variational principle gives
\begin{equation}\label{b:shell-energy}
 \tfrac12\int_{a_z<|x-z|<\beta_h}|\nabla u_h|^2\ge c_zm_z^2.
\end{equation}

Suppose $u_h\rightharpoonup u$ in $L^2$ and $\sup_h\mathcal E_h(u_h)<\infty$.
The sequence is bounded in $H^1(M)$, hence every subsequence has a
further subsequence converging strongly in $L^2$ to $u$. Cube Poincar\'e
and~\eqref{b:trace-means} give
\[
 \begin{aligned}
 \left\|\sum_z\bar u_z\one_{Q_z}-u_h\right\|_2
       &\le Ch\|\nabla u_h\|_2\longrightarrow0,\\
 \left\|\sum_z(m_z-\bar u_z)\one_{Q_z}\right\|_2^2
       &\le Ch\|\nabla u_h\|_2^2\longrightarrow0.
 \end{aligned}
\]
Thus both step functions converge strongly to $u$. Moreover,
\[
 \left\|\sum_zk(z)\one_{Q_z}-k\right\|_\infty\longrightarrow0,
 \qquad c_z=h^3k(z)(1+O(h))
\]
uniformly in $z$, so
\[
 \sum_zc_zm_z^2\longrightarrow\int_{M}ku^2.
\]
The union $S_h$ of the capacity shells has volume $O(h^3)$. For every
$q\in L^2(M;\R^3)$,
\[
 \left|\int_{S_h}\nabla u_h\cdot q\right|
 \le\|\nabla u_h\|_2\|q\one_{S_h}\|_2\longrightarrow0.
\]
Hence $\one_{M\setminus S_h}\nabla u_h\rightharpoonup\nabla u$ in $L^2$.
The two energies are retained in the decomposition
\[
 \mathcal E_h(u_h)
 \ge\frac12\int_{M\setminus S_h}|\nabla u_h|^2
           +\sum_zc_z|m_z|^2.
\]
Lower semicontinuity of the first term and convergence of the second give
\[
 \liminf_{h\downarrow0}\mathcal E_h(u_h)
 \ge\frac12\int_M|\nabla u|^2+\int_Mk|u|^2.
\]
This proves the Mosco lower inequality.

For smooth $u$, take $u_h=uw_h$, where $w_h=w_z$ in the shells, zero inside the holes and one outside all shells.  Then $u_h\to u$ strongly in $L^2$.  The total shell volume tends to zero, the capacities have bounded sum, and
\[
 \sum_z\int_{a_z<|x-z|<\beta_h}|\nabla w_z|\,dx
 =\sum_z\frac{4\pi a_z(\beta_h-a_z)}{1-a_z/\beta_h}=O(h^2).
\]
The energy expansion satisfies
\[
 \begin{aligned}
 \left|\int_Muw_h\nabla u\cdot\nabla w_h\right|
       &\le C_u\|\nabla w_h\|_1=O(h^2),\\
 \left|\frac12\int_Mu^2|\nabla w_h|^2
                         -\sum_zu(z)^2c_z\right|
       &\le C_u\beta_h\sum_zc_z=o(1),\\
 \frac12\int_Mw_h^2|\nabla u|^2
       &\longrightarrow\frac12\int_M|\nabla u|^2.
 \end{aligned}
\]
Since $\sum_zu(z)^2c_z\to\int_Mku^2$, we obtain
$\mathcal E_h(uw_h)\to\mathcal E_k(u)$ for smooth $u$.
Smooth approximation in $H^1$, followed by a diagonal choice, gives
the recovery sequence for every $u\in H^1(M)$.

Write $\mathsf R_h=(A_h+\lambda)^{-1}_0$ and
$\mathsf R_k=(A_k+\lambda)^{-1}$. For fixed $f\in L^2$, the zero extended resolvent $\mathsf R_hf$ uniquely minimizes
$\mathcal E_h(u)+\lambda\|u\|_2^2-2\langle f,u\rangle$.  The lower and recovery inequalities identify every weak limit of these minimizers with the minimizer for $\mathcal E_k$.  They also give convergence of the minimum values.  Coercivity and the variational equations then give strong $L^2$ convergence $\mathsf R_hf\to\mathsf R_kf$.

The resolvent equation gives
\[
 \begin{gathered}
 \mathcal E_h(\mathsf R_hf)+\lambda\|\mathsf R_hf\|_2^2
 =\langle f,\mathsf R_hf\rangle,\\
 \|\mathsf R_hf\|_2\le\lambda^{-1}\|f\|_2,
 \quad
 \|\nabla\mathsf R_hf\|_2\le(2/\lambda)^{1/2}\|f\|_2.
 \end{gathered}
\]
Thus the unit-input images are collectively compact in $L^2$. If norm
convergence failed, there would be unit $f_h$ and $\varepsilon>0$ with
$\|(\mathsf R_h-\mathsf R_k)f_h\|_2\ge\varepsilon$. Pass to
$f_h\rightharpoonup f$ and $\mathsf R_hf_h\to g$ strongly.
Self adjointness and fixed-input convergence give, for every $z\in L^2$,
\[
 \langle g,z\rangle=\lim_h\langle f_h,\mathsf R_hz\rangle
 =\langle f,\mathsf R_kz\rangle=\langle\mathsf R_kf,z\rangle.
\]
The operator $\mathsf R_k$ is compact, so $\mathsf R_kf_h\to\mathsf R_kf$ strongly as well.  This contradicts the norm defect and proves resolvent norm convergence.

On $[0,\lambda^{-1}]$, the function
\[
 F_t(s)=\begin{cases}e^{-t(s^{-1}-\lambda)},&s>0,\\0,&s=0\end{cases}
\]
is continuous. Spectral calculus gives $P_t^h=F_t(\mathsf R_h)$ and
$P_t^k=F_t(\mathsf R_k)$, including the zero extension on the holes.
For every polynomial $p$,
\[
 \begin{aligned}
 \|P_t^h-P_t^k\|_{2\to2}
 &\le2\|F_t-p\|_{\infty,[0,\lambda^{-1}]}
       +\|p(\mathsf R_h)-p(\mathsf R_k)\|_{2\to2},\\
 \limsup_{h\downarrow0}\|P_t^h-P_t^k\|_{2\to2}
 &\le2\|F_t-p\|_{\infty,[0,\lambda^{-1}]}.
 \end{aligned}
\]
Uniform polynomial approximation proves the heat operator conclusion
for every $t>0$.
\end{proof}

\subsection{Strong convergence of marked copying operators}\label{app:marked-output-proof}
We prove Lemma~\ref{lem:strong-marked-operator}. Keep its root partition,
victim and donor, duration interval and bounded continuous record mark.
\begin{proof}
Let $\mu_{\star,x}$ be the one-particle lifetime measure. Independence
before the first death gives
\[
 Q_\star^1f(y)=\frac1{N-1}\int f(\xi)
 \int_{[\delta,M_0]}\prod_{j\ne v}
 p_\star(t,\xi_{\mathcal P(j)},y_j)
 \,\mu_{\star,\xi_{\mathcal P(v)}}(dt)\,d\xi.
\]
The lifetime measure has mass at most one. Heat-kernel domination and
the spectral gradient estimate give, uniformly for $t\ge\delta$,
\[
 \sup_x\left(\|p_\star(t,x,\cdot)\|_\infty
             +\|p_\star(t,x,\cdot)\|_{H^1(M)}\right)\le C_\delta.
\]
Indeed, split the heat row at $t/2$ and use
$\|A_\star^{1/2}e^{-tA_\star/2}\|_{2\to2}\le C_\delta$;
the form controls its gradient. Tensor products and Minkowski yield
\begin{equation}\label{eq:marked-output-domination}
 \|Q_\star^1f\|_{H^1}\le C_\delta\|f\|_1,
 \qquad |Q_\star^Gf|\le Q_\star^1|f|\,\|G\|_\infty
 \le C_\delta\|G\|_\infty\|f\|_1.
\end{equation}

Stop the compressed record at its first death on $[0,M_0]$ and
reconstruct the first continuous piece on a grid of mesh $\theta$,
as in Lemma~\ref{lem:full-record-weak}. Denote this reconstruction by
$\Psi_\theta^{(1)}$; it retains the exact survivor endpoints and duration.
On $\{\delta\le\tau_1\le M_0\}$,
\[
 d(\mathsf C_1,\Psi_\theta^{(1)}\mathsf R)
 \le C\omega_{M_0}(2\theta)
 \quad\text{if }\omega_{M_0}(2\theta)<L/4.
\]
Durations in $[\delta,M_0]$ and free Brownian moduli give compact
sets of rescaled paths uniformly in the roots and the model.
Uniform continuity near these compact sets gives the following limit,
with $\xi\in M_h^r$ in the hard model:
\[
 e_G(\theta):=\sup_{\star,\xi}\E_{\mathcal P,\xi}^\star
 \left[|G(\mathsf C_1)-G(\Psi_\theta^{(1)}\mathsf R)|;
                  \delta\le\tau_1\le M_0\right]\longrightarrow0.
\]
Write $Q_{\star,\theta}^G$ for the reconstructed marked operator.
Lemma~\ref{lem:compressed-record-tv}, its truncated-record proof,
and contraction under bounded weighting and endpoint projection imply
\[
 \|Q_{h,\theta}^Gf_h-Q_{k,\theta}^Gf\|_1\to0
 \quad(\theta>0\text{ fixed}),\qquad
 \|Q_\star^Gf-Q_{\star,\theta}^Gf\|_1
 \le e_G(\theta)\|f\|_1.
\]
For signed inputs use their positive and negative parts, which converge
in $L^2$; changes of input cost at most
$\|G\|_\infty\|f_h-f\|_1$ in the soft model.
Taking $h\downarrow0$ and then $\theta\downarrow0$ gives
$\|Q_h^Gf_h-Q_k^Gf\|_1\to0$. Finally,~\eqref{eq:marked-output-domination}
and boundedness of $\|f_h\|_1$ give
\[
 \|Q_h^Gf_h-Q_k^Gf\|_2^2
 \le C_{\delta,G}\left(\sup_h\|f_h\|_1+\|f\|_1\right)
                  \|Q_h^Gf_h-Q_k^Gf\|_1\longrightarrow0.
\]
\end{proof}

\section{Boundary polarity and restart continuity}\label{app:local-boundary}
\subsection{Polarity of simultaneous boundary hits}
\begin{lemma}\label{lem:simultaneous-wall-polarity}
Every finite wall hit of a smooth auxiliary system satisfying~\eqref{eq:perturbation}, for sufficiently small amplitude, has almost surely a unique boundary victim.\end{lemma}
\begin{proof}
Localize one continuous segment in a spatial compact set and choose
two distinct labels in smooth wall-distance coordinates $d_1,d_2$.
For smooth $f,g$, write $\Gamma(f,g)=\sum_i\nabla_i f^{\mathsf T}a_i\nabla_i g$, where $a_i=\sigma_i\sigma_i^{\mathsf T}$. Let $a_j^{\rm nor}=\Gamma(d_j,d_j)$ and $\zeta_j=d_j/\sqrt{a_j^{\rm nor}}$.
Local coefficient smoothness and the block-diagonal noise give
\[
 \Gamma(\zeta_j,\zeta_k)=\delta_{jk}+O(r),\qquad
 |\cL\zeta_j|\le C,\qquad r=(\zeta_1^2+\zeta_2^2)^{1/2}.
\]
At $d_1=d_2=0$, normalization makes the leading covariance the identity;
derivatives of the normalization are multiplied by a vanishing distance.
Constants here may depend on the fixed smooth coefficient system.

For $0<r<e^{-2}$ take $f(r)=\log\log(1/r)$. Its two-dimensional radial
principal term is
\[
 \tfrac12(f''+f'/r)=-\frac1{2r^2\log^2(1/r)}.
\]
Covariance errors and bounded drift are at most
$C/[r\log(1/r)]$. For a sufficiently small chart radius $r_0$,
\[
 \cL f(r)\le-\frac1{2r^2\log^2(1/r)}
               +\frac{C}{r\log(1/r)}
 \le-\frac1{4r^2\log^2(1/r)},\qquad 0<r<r_0.
\]
After a smooth cutoff away from zero and adding a constant, $f\ge0$ and
$\cL f\le C$. Let $\tau_\varepsilon$ be the first hit of
$r=\varepsilon$, $\chi$ the spatial localization exit, and $\tau$
the first wall hit of the segment. For
$\theta=\tau_\varepsilon\wedge\chi\wedge\tau\wedge T$, stopped Itô gives
\[
 \E f(r_\theta)\le f(r_{\rm start})+CT.
\]
Nonnegativity then gives
\[
 \Prob\{\tau_\varepsilon\le\chi\wedge\tau\wedge T\}
 \le\frac{f(r_{\rm start})+CT}{\log\log(1/\varepsilon)}\longrightarrow0.
\]
If necessary begin this estimate on entry into the small distance chart.
Countable local entries, a finite chart cover of each spatial compact set,
increasing compact sets, and the finitely many label pairs exclude a
simultaneous wall hit in the current segment.

At the first hit there is one victim and all survivors are interior.
Copying an interior donor starts the next segment in $D^N$. Conditional
induction and a countable union over numbered copies give unique victims
for every finite prefix.

\end{proof}

\subsection{Interior finite-prefix continuity}
\begin{lemma}\label{lem:interior-prefix-continuity}
Fix a periodic perforated domain $D$, an interior compact set $K\subset D^N$, and $0<T_-<T_+<\infty$. If $x_n\in K$ tends to $x$, $T_n\to T\in[T_-,T_+]$, and the perturbation amplitudes tend to zero, then the joint endpoint and actual reference-displacement laws at $T_n$ converge weakly to the baseline law at $T$. The convergence is uniform over the allowed coefficient systems and absolute lattice translates. The baseline block mean is interior continuous, and its hidden kernel preserves interior continuous functions bounded by $C(1+\sqrt{\mathcal V})$.
\end{lemma}
\begin{proof}
Use common Brownian primitives and donor marks. For free paths $Y_n,Y$
started at $y_n,y$, the difference has drift bounded by $\eta_nt$ and
martingale bracket bounded by $C\eta_n^2t$. Thus BDG gives, whenever
$|y_n-y|+\eta_nT_+\le\varepsilon/2$,
\[
 \sup_{\text{allowed coefficients}}
 \Prob\left\{\sup_{s\le T_+}|Y_n(s)-Y(s)|>\varepsilon\right\}
 \le C\eta_n^2T_+/\varepsilon^2\longrightarrow0.
\]

For a baseline segment with first wall hit $\tau$, for every
$0<\delta<\tau$,
\[
 \inf_{s\le\tau-\delta}\dist(Y(s),\partial D)>0.
\]
At $\tau$ the free Brownian continuation immediately enters the removed
ball. In a signed-distance chart, its normal coordinate has a continuous
martingale part with quadratic variation $t+o(t)$ and bounded drift.
DDS and the Brownian law of the iterated logarithm give negative values
arbitrarily soon after zero. Consequently, for every $\delta>0$ there
is $s\in(\tau,\tau+\delta)$ with $Y(s)$ strictly inside a removed ball.
Uniform approximation on these compact intervals gives the exit-time
and exit-position limits
\[
 \tau_n\xrightarrow{\Prob}\tau,
 \qquad Y_n(\tau_n)\xrightarrow{\Prob}Y(\tau)
 \quad\text{on }\{\tau<\infty\}.
\]
Baseline labels have independent atomless exit waits. Their first finite
minimum therefore has a unique victim, whose label is unchanged with
probability tending to one. Using the common donor mark, the copied
survivor positions converge as well.

Repeat the argument after each fixed numbered copy with fresh primitives.
It gives convergence of every finite prefix, including all survivor
endpoints and reference jumps. On compact sets of initial states,
\[
 \sup_{n,x\in K}\Prob_x(K_{T_+}^{(n)}>m)
 \le e^{CT_+}\sup_{x\in K}\Phi(x)e^{-\alpha m}
 \longrightarrow0.
\]
Here the superscript $(n)$ denotes the $n$th perturbed system. Each
numbered baseline death differs from the deterministic time $T$ almost
surely, by conditioning at the preceding death and using the atomless
exit wait. Thus the last segment converges at durations tending to $T$.
First pass to the limit for each fixed prefix, then let $m\to\infty$
using the displayed bound.

To obtain uniformity, a contrary sequence has a convergent subsequence
of initial points and times, to which the same argument applies. Lattice
translation translates the coefficient field and preserves its amplitude
bound. By~\eqref{eq:landing}, \eqref{eq:Foster}, and the ancestral envelope,
\[
 \sup_{n,x\in K,\ t\in[T_-,T_+]}
 \E_x\bigl[|\Delta_t^{(n)}|^2+1+\mathcal V(H_t^{(n)})\bigr]<\infty.
\]
For a continuous $f$ with $|f|\le C(1+\sqrt{\mathcal V})$, these bounds
give
\[
 \begin{aligned}
 \sup_n\E\bigl[|f(H_{T_n}^{(n)})|;
                       |f(H_{T_n}^{(n)})|>R\bigr]&\le C/R,\\
 \sup_n\E\bigl[|\Delta_{T_n}^{(n)}|;
                       |\Delta_{T_n}^{(n)}|>R\bigr]&\le C/R.
 \end{aligned}
\]
The common coupling gives convergence in probability to the interior
baseline endpoint. Therefore, for $[v]_R=\max(-R,\min(v,R))$,
\[
 [f(H_{T_n}^{(n)})]_R\xrightarrow{\Prob}[f(H_T)]_R,\qquad
 [\Delta_{T_n}^{(n)}]_R\xrightarrow{\Prob}[\Delta_T]_R,
\]
with the vector truncation taken coordinatewise. Bounded convergence
in probability and the uniform tail bounds give
\[
 \E f(H_{T_n}^{(n)})\longrightarrow\E f(H_T),\qquad
 \E\Delta_{T_n}^{(n)}\longrightarrow\E\Delta_T.
\]
This proves continuity of the baseline mean and kernel action.
For the block kernel, Jensen and \eqref{eq:Foster} also give
\[
 |P_Tf(\xi)|\le C\bigl(1+\sqrt{P_T\mathcal V(\xi)}\bigr)
 \le C'\bigl(1+\sqrt{\mathcal V(\xi)}\bigr),
\]
which verifies preservation of the stated growth class.
\end{proof}

\subsection{Random interior restarts}
\begin{lemma}\label{lem:random-interior-restart}
Let $X^{(n)}$ start at $x_n\in D^N$ with perturbation amplitudes
$\eta_n\to0$, and let $\sigma_n$ be finite stopping times. Suppose
\[
 \begin{gathered}
 \sigma_n\xrightarrow{\Prob}0,\qquad \Prob(Y\in D^N)=1,\\
 (Y_n,D_n):=(X^{(n)}(\sigma_n),
                X_N^{(n)}(\sigma_n)-(x_n)_N)
       \Rightarrow(Y,D).
 \end{gathered}
\]
For each $T>0$, the joint endpoint and displacement then converge to
\[
 \begin{aligned}
 &\Law(H_T^{(n)},\Delta_T^{(n)})\\
 &\quad\Rightarrow
 \int\Law_y^0(H_T,\Delta_T+d)\,\Law(Y,D)(\dd y,\dd d).
 \end{aligned}
\]
\end{lemma}
\begin{proof}
For a bounded Lipschitz $F$, put
$Q_s^\eta F(y,d)=\E_y^\eta F(H_s,\Delta_s+d)$.
The full-state strong Markov property, with donor marks revealed at
their deaths, gives
\[
 \left|\E F(H_T^{(n)},\Delta_T^{(n)})
 -\E\left[\one_{\{\sigma_n<T\}}
              Q_{T-\sigma_n}^{\eta_n}F(Y_n,D_n)\right]\right|
 \le\|F\|_\infty\Prob(\sigma_n\ge T).
\]
Lemma~\ref{lem:interior-prefix-continuity} and a finite net in the
displacement variable give, for $K\Subset D^N$ and $R<\infty$,
\[
 \sup_{\substack{y\in K,\ |d|\le R\\ T/2\le s\le T}}
 |Q_s^{\eta_n}F(y,d)-Q_s^0F(y,d)|\to0.
\]
The same lemma makes $Q_s^0F(y,d)$ continuous for $y$ interior and
$s>0$. Tightness and $\Prob(Y\in D^N)=1$ allow restriction to these
compact sets with arbitrarily small lost mass. Since
$(Y_n,D_n,\sigma_n)\Rightarrow(Y,D,0)$,
\[
 \E F(H_T^{(n)},\Delta_T^{(n)})
       \longrightarrow\E Q_T^0F(Y,D).
\]
Bounded Lipschitz tests determine the asserted joint weak limit.
\end{proof}

\section{Homogenization at arbitrary killing rates for three particles}\label{app:three-particle-rates}
For three particles, a lifetime potential supplies one energy estimate
for holes of arbitrary relative radii. Its It\^o identity controls both
the gradient energy between deaths and the donor squares added at copies.
A corner contraction controls small deaths, while a donor with positive
lifetime potential controls large deaths. Combining these estimates
yields uniform exponential counts. Microscopic rescue and finite-record
comparison then give the common-ancestor bound required for
arbitrary-rate transfer.

Throughout this appendix $N=3$ and $0<k_-\le k\le k_+<\infty$ on
$M=(\R/L\Z)^3$. The radii and domains are those of~\eqref{eq:holes};
the outer radius of each local annulus is $\rho_h=h/4$.
All constants may depend on $L,k_-,k_+$ and the fixed rate, and are
uniform over sufficiently small $h$.

\subsection{Lifetime energy and the normal corner}
Fix $c>0$. For the single-particle lifetime $\zeta_h$ define
\begin{equation}\label{d3:lifetime}
 r_h(x)=\E_x(1-e^{-c\zeta_h}),\qquad
 \tfrac12\Delta r_h=c(r_h-1),\qquad r_h|_{\partial M_h}=0.
\end{equation}
Set $a_- =\min_z a_z$, $a_+=\max_z a_z$ and $b_h=\vartheta a_-$ for a fixed $0<\vartheta<1/100$. Let $d_h$ be the spherical boundary distance up to $2b_h$, smoothly flattened to $5b_h/2$ beyond $3b_h$, as specified below. Put
\[
 \begin{gathered}
 R_h=(\sum_{i=1}^3d_h(X_i)^2)^{1/2},\quad
 \mathcal G_h=\sum_i|\nabla r_h(X_i)|^2,\\
 J_t=\sum_{s\le t}^{\rm deaths}r_h(X_{d,s})^2,\quad A_t=\int_0^t\mathcal G_h(X_s)\dd s+J_t.
 \end{gathered}
\]
Let $\ell(1)=0$, $\ell'(s)=\chi_0(s)/s$, where $\chi_0$ is smooth,
equals one on $[0,1]$, vanishes on $[2,\infty)$ and takes values in $[0,1]$.
For $0<\gamma\le1/2$ define $\Phi_{h,\gamma}=e^{-\gamma\ell(R_h/b_h)}$.
In the core it is $(b_h/R_h)^\gamma$. The normal-coordinate generator gives
\begin{equation}\label{d3:corner-generator}
 \frac{\cL\Phi_{h,\gamma}}{\Phi_{h,\gamma}}
 =\frac{\gamma(\gamma-1)}{2R_h^2}
 -\frac{\gamma}{R_h^2}\sum_i\frac{d_h(X_i)}{a_{z_i}+d_h(X_i)}\le0,
 \qquad R_h\le b_h.
\end{equation}
For $F=\sum_i r_h(X_i)^2$, It\^o with the positive donor jumps gives
\begin{equation}\label{d3:energy-identity}
 F_t=F_0+M_t+A_t-2c\int_0^t\sum_i r_h(X_i)(1-r_h(X_i))\dd s,
 \qquad \langle M\rangle_t\le4A_t.
\end{equation}
The two estimates used below are
\begin{equation}\label{d3:corner-energy}
 \cL\Phi_{h,\gamma}\le C_0\gamma\mathcal G_h\Phi_{h,\gamma},
\end{equation}
and, for a sufficiently small fixed $\gamma$,
\begin{equation}\label{d3:interior-count}
 \E_xe^{\alpha K_{T\wedge\tau_m}}
 \le C_1\Phi_{h,\gamma}(x)^{1/4}e^{C_2T}.
\end{equation}
Here $\tau_m$ is the $m$th death time, and $X_{d,s}$ in $J_t$ denotes
the donor position just before the copy at time $s$. The corner exponent
$\gamma$ is chosen below for this appendix.

\subsection{Single particle lifetime estimates}

\begin{lemma}[Capture and lifetime]\label{ah:lifetime}
Let $A=k_+/(2\pi)$. If $R\ge2$ and $x$ lies outside all balls
$B(z+L\ell,Ra_z)$, $\ell\in\Z^3$, then
\begin{equation}\label{ah:capture}
 \mathbb P_x(\zeta_h\le t)\le R^{-1}+CA(t+h^2),\qquad 0<t\le1.
\end{equation}
If $x$ lies on a sphere $|x-z|=\rho_h$, the $R^{-1}$ term can be omitted.  There are $t_0,m_0>0$, independent of small $h$ and of $z$, such that
\begin{equation}\label{ah:outer-lifetime}
 r_h(x)\ge m_0\qquad (|x-z|=\rho_h).
\end{equation}
The same lower bound holds at points whose distance from every centre is at least $\rho_h$.
\end{lemma}
\begin{proof}
For one ball of radius $a$ and $r=|x-z|>a$, the radial heat equation gives
\begin{equation}\label{ah:single-ball}
 \mathbb P_x(\tau_{B(z,a)}\le t)
 =\frac ar\operatorname{erfc}\left(\frac{r-a}{\sqrt{2t}}\right).
\end{equation}
For $u(t,r)=\mathbb P_x(\tau_{B(z,a)}\le t)$, the function $v(t,r)=ru(t,r)$
satisfies
\[
 \partial_tv=\tfrac12\partial_{rr}v\quad(r>a),\qquad
 v(t,a)=a,\qquad v(0,r)=0.
\]
The reflected one-dimensional heat kernel gives~\eqref{ah:single-ball}.

Use the union bound over the entire periodic lattice of centres.  There is at most one centre at distance less than $h/4$, and its contribution is at most $R^{-1}$.  For all other centres, $a_z\le Ah^3$ and $r-a_z\ge r/2$ for small $h$.  Group centres in lattice shells of width $h$.  There are at most $C(j+1)^2$ centres in the $j$th shell; apart from a fixed number of initial shells their distances are at least $c_1jh$.  The inequality $\operatorname{erfc}(s)\le e^{-s^2}$ therefore bounds their total contribution by
\[
 CAh^2\left(1+\sum_{j\ge1}j e^{-c_2j^2h^2/t}\right)
 \le CA(t+h^2).
\]
On $|x-z|=h/4$ the nearest contribution is also at most $CAh^2$.
The same argument applies if every centre is at least $h/4$ away.

Choose $t_0\in(0,1)$ so that $CAt_0\le1/4$, then take $h$ small enough that $CAh^2\le1/4$.  At the stated points $\mathbb P_x(\zeta_h>t_0)\ge1/2$.  Since
\[
 r_h(x)=c\int_0^\infty e^{-ct}\mathbb P_x(\zeta_h>t)\,dt,
\]
we may take $m_0=ct_0e^{-ct_0}/2$.
\end{proof}

\begin{lemma}[Normal gradients]\label{ah:gradient}
For $a_z<r=|x-z|\le\rho_h/2$,
\begin{equation}\label{ah:radial-gradient}
 \partial_r r_h(x)\ge\frac{m_0a_z}{r^2}
                         -C(\rho_h^{-1}+c\rho_h).
\end{equation}
Consequently, if $Ra_z\le\rho_h/2$ and
\begin{equation}\label{ah:growing-condition}
 CR^2a_z(\rho_h^{-1}+c\rho_h)\le m_0/2,
\end{equation}
then $|\nabla r_h|\ge m_0/(2R^2a_z)$ on $a_z<r\le Ra_z$.  In particular,
\begin{equation}\label{ah:fixed-collar-gradient}
 \begin{gathered}
 |\nabla r_h(x)|\ge c_0/a_z
 \quad (a_z<r\le a_z+2b_h),\\
 \mathcal G_h(X)\ge c_1/a_+^2\quad(R_h(X)\le2b_h).
 \end{gathered}
\end{equation}
\end{lemma}
\begin{proof}
The following estimates are local to one sphere, with constants fixed
by the scaled ball and annulus estimates.
Extend $c(1-r_h)$ from the annulus to the inner ball, with values in
$[0,c]$, and solve $-\frac12\Delta u=f$ on $B(z,\rho_h)$ with the outer
boundary values of $r_h$. The maximum principle and
\eqref{ah:outer-lifetime} give
$m_0\leq u\leq1+c\rho_h^2/3$. Rescale this ball to the unit ball.
The Poisson representation of its bounded boundary data and interior
$W^{2,p}$ estimates for its bounded forcing, with a fixed $p>3$, give
\[
 \sup_{B(z,\rho_h/2)}|\nabla u|
 \le M=C(\rho_h^{-1}+c\rho_h).
\]
They also give
$[\nabla u]_{C^\alpha(B(z,\rho_h/2))}\leq C\rho_h^{-\alpha}M$
for a fixed $0<\alpha<1-3/p$. Since $a_z/\rho_h=O(h^2)$, the
$C^{1,\alpha}$ norm on the unit sphere of
$u(z+a_z\omega)-u(z)$ is at most $Ca_zM$. All constants here are
independent of $h,z$ and of the extension of the forcing.

The difference $v=u-r_h$ is harmonic on the annulus, vanishes on the outer sphere, and equals $u(a_z\omega)$ on the inner sphere.  Decompose it as
\[
 v=u(z)\frac{a_z/r-a_z/\rho_h}{1-a_z/\rho_h}+v_e.
\]
The maximum principle gives $\|v_e\|_\infty\leq Ca_zM$ on the whole annulus.
After scaling by $a_z$, its inner boundary is the unit sphere and its
outer radius is at least four. The local harmonic boundary estimate on
$1\leq r/a_z\leq2$ therefore has a constant independent of that outer
radius and gives $|\nabla v_e|\leq CM$. For
$2a_z\leq r\leq\rho_h/2$, an interior ball of radius comparable to $r$
lies in the annulus, so
$|\nabla v_e|\leq C a_zM/r$. Combining the two regions gives
\[
 |\nabla v_e(x)|\le
 \begin{cases}
 CM,&a_z<r\le2a_z,\\
 Ca_zM/r,&2a_z\le r\le\rho_h/2,
 \end{cases}
 \qquad \sup_{a_z<r\le\rho_h/2}|\nabla v_e|\le CM.
\]
Consequently
\[
 \partial_r r_h=\partial_r u-\partial_r v
 \ge\frac{u(z)a_z}{(1-a_z/\rho_h)r^2}-CM,
\]
which proves~\eqref{ah:radial-gradient} after enlarging $C$.

Condition~\eqref{ah:growing-condition} makes the error at most half the radial main term.  In a fixed ratio collar $r\le a_z+2b_h$, the ratios $r/a_z$ are uniformly bounded and $a_z/\rho_h=O(h^2)$.  This proves the first assertion in~\eqref{ah:fixed-collar-gradient}.  When $R_h\le2b_h$, each clipped distance is an actual spherical distance at most $2b_h$, which proves its second assertion.
\end{proof}

Choose a smooth nonincreasing function $\chi$ on $[2,3]$, equal to one near $2$ and to zero near $3$, with integral $1/2$.  The clipped distance used throughout is the actual spherical distance $s$ for $s\le2b_h$, is
$b_h(2+\int_2^{s/b_h}\chi(u)\,du)$ for $2b_h<s<3b_h$, and is $5b_h/2$ thereafter.  The collars are disjoint.  Hence $d_h$ is smooth on $M_h$, nondecreasing with the actual distance within each collar, and never exceeds that distance.

\begin{lemma}[A donor with clearance]\label{ah:donor}
There is a uniform $\kappa>0$ such that
\begin{equation}\label{ah:donor-bound}
 d_h(x)\ge b_h/(2\sqrt3)\quad\Longrightarrow\quad r_h(x)\ge\kappa.
\end{equation}
\end{lemma}
\begin{proof}
If $x$ is outside all the $\rho_h$ spheres, Lemma~\ref{ah:lifetime} applies.  Otherwise it is in one annulus $a_z<|x-z|<\rho_h$, which contains no other hole.  Let $\eta$ be its exit time.  The harmonic radial formula gives
\[
 \mathbb P_x(\text{outer exit})
 =\frac{1-a_z/|x-z|}{1-a_z/\rho_h}
 \ge s_0:=\frac{b_h/(2\sqrt3)}{a_++b_h/(2\sqrt3)}>0.
\]
The last expression is at least the fixed number
$s_*=\vartheta k_-/(2\sqrt3 k_++\vartheta k_-)>0$.
Also $\mathbb E_x\eta\le\rho_h^2/3$, by comparison with exit from the outer ball.  The strong Markov decomposition of $r_h$ gives
\[
 r_h(x)\ge m_0\mathbb E_x[e^{-c\eta};\text{outer exit}]
 \ge m_0(s_*-c\rho_h^2/3)\ge m_0s_*/2
\]
for small $h$ such that $c\rho_h^2/3\le s_*/2$.  Take the fixed constant $\kappa=m_0s_*/2$.
\end{proof}

\subsection{Boundary energy and the number of deaths}

The cutoff definition gives $\Phi_{h,\gamma}\ge2^{-\gamma}$ everywhere.
Between deaths, in $R_h\le2b_h$, let $z_i$ be the centre of the sphere
whose collar contains $X_i$. The normal coordinates satisfy
\begin{equation}\label{ah:normal-generator}
 |\nabla d_h(X_i)|=1,\qquad
 \tfrac12\Delta d_h(X_i)=\frac1{a_{z_i}+d_h(X_i)}.
\end{equation}
The noises in these three normals have covariance the identity, as the Brownian drivers of different labels are independent.  At a death the victim has normal coordinate zero, so copying gives
\begin{equation}\label{ah:normal-copy}
 R_h(X(s))^2=R_h(X(s-))^2+d_h(X_{d,s})^2.
\end{equation}
Thus every copy decreases $\Phi_{h,\gamma}$. Radial differentiation
gives~\eqref{d3:corner-generator}. In the transition band the curvature
term is nonpositive, and Lemma~\ref{ah:gradient} gives
\[
 \begin{aligned}
 (\cL\Phi_{h,\gamma})^+
 &\le C\gamma b_h^{-2}\Phi_{h,\gamma}\\
 &\le C\gamma(a_+/b_h)^2\mathcal G_h\Phi_{h,\gamma}
 \le C_0\gamma\mathcal G_h\Phi_{h,\gamma},
 \quad b_h<R_h<2b_h.
 \end{aligned}
\]
The ratio $a_+/b_h\le k_+/(\vartheta k_-)$ is fixed. In the core the
drift is nonpositive, and for $R_h\ge2b_h$ it is zero. This proves
\eqref{d3:corner-energy}, uniformly in $h$ and $\gamma\in(0,1/2]$.

Use the natural filtration, revealing each donor uniform at its death.
Lemma~\ref{lem:finite-prefix} gives a unique victim at every finite copy
and $R_h>0$ throughout each finite prefix. For fixed $h$, elliptic
regularity gives $r_h\in C^1(\overline M_h)$, so the local It\^o formulas
extend to each wall hit and include the donor jump.

For a fixed prefix $\tau_m$, put
\[
 \sigma_{n,\varepsilon}=\inf\{t:R_h(X_t)\le\varepsilon b_h
       \text{ or }A_t\ge n\text{ or }\max_i|X_i(t)|\ge n\}.
\]
Localize exponential processes additionally by their quadratic variations.
Bounded positions and a positive minimum of $R_h$ on each finite prefix
give, for fixed $h,m,T$,
\[
 \lim_{\varepsilon\downarrow0}\lim_{n\to\infty}
       (T\wedge\tau_m\wedge\sigma_{n,\varepsilon})
       =T\wedge\tau_m\quad\text{almost surely}.
\]
The bounds below are uniform in $m,n,\varepsilon$ and small $h$.
Fatou removes the auxiliary stops at fixed $h,m$; the count bound then
removes $m$.

\begin{lemma}[Stopped lifetime-energy estimate]\label{ah:energy}
For each interior initial configuration and each finite death prefix,
\begin{equation}\label{ah:energy-exp}
 \mathbb E_x\exp(A_{T\wedge\tau_m}/8)
 \le\exp(3/4+3cT/8).
\end{equation}
The same bound holds at every stopping time $S\le T\wedge\tau_m$,
with $A_S$ including the donor square when $S$ is a death.
\end{lemma}
\begin{proof}
Apply It\^o's formula on each Brownian piece to $F=\sum_i r_h(X_i)^2$.  Since
$\frac12\Delta(r_h^2)=|\nabla r_h|^2-2cr_h(1-r_h)$ and the copy adds the donor square, one obtains~\eqref{d3:energy-identity}.  The continuous local martingale is
\[
 \begin{gathered}
 M_t=2\sum_i\int_0^t r_h(X_i)\nabla r_h(X_i)\cdot dB_i,\\
 \langle M\rangle_t=4\int_0^t\sum_i r_h(X_i)^2|\nabla r_h(X_i)|^2\,ds\le4A_t.
 \end{gathered}
\]
Use $0\le F\le3$ and $\sum_i r_i(1-r_i)\le3/4$ to obtain
$A_t\le3+3ct/2-M_t$.  For $0<\lambda<1/2$, the nonnegative exponential local martingale
\[
 Z_t=\exp(-\lambda M_t-\lambda^2\langle M\rangle_t/2)
\]
has expectation at most one.  The preceding inequalities imply
\[
 Z_t\ge\exp\{(\lambda-2\lambda^2)A_t
                         -\lambda(3+3ct/2)\}.
\]
Take $\lambda=1/4$ and $t=S\le T\wedge\tau_m$ at the auxiliary stops.
Then
\[
 \mathbb E_xe^{A_S/8}
 \le e^{3/4+3cT/8}\mathbb E_xZ_S
 \le e^{3/4+3cT/8}.
\]
The localization convention above and Fatou give~\eqref{ah:energy-exp}.
\end{proof}

\begin{lemma}[Death counts]\label{ah:count}
The estimate~\eqref{d3:interior-count} holds from every interior state, with uniform constants $\alpha,C_1,C_2>0$.  The processes in $D_h$ and in their phase quotients are nonexplosive.  The estimate holds after removing the count stopping.
\end{lemma}
\begin{proof}
Call a death small if its prestate has $R_h\le b_h/2$, and call it large otherwise; let $K^s,K^\ell$ be their counts.  At a small death, put $u=d_h(X_{d,s})^2/R_h^2$.  The two survivor squares sum to $R_h^2$, and~\eqref{ah:normal-copy} keeps the poststate inside $R_h<b_h$.  The conditional ratio is
\[
 \tfrac12\{(1+u)^{-\gamma/2}+(2-u)^{-\gamma/2}\}
 \le q_\gamma=\tfrac12(1+2^{-\gamma/2}).
\]
The function in braces is convex on $[0,1]$ and has equal endpoint values, proving the inequality.  Set $\theta_s=-\log q_\gamma$.  The continuous drift inequality and the conditional copy inequality show that
\[
 e^{\theta_sK_t^s-C_0\gamma A_t}\Phi_{h,\gamma}(X(t))
\]
is a nonnegative local supermartingale at each finite prefix and the
auxiliary stops above. The conditional copy estimate uses the next donor
uniform, which has not yet been revealed. At a large death the factor
$e^{-C_0\gamma\Delta A}$ and the corner function both decrease.
Since $\Phi_{h,\gamma}\ge2^{-\gamma}$, Fatou's lemma at fixed $m$ gives
\begin{equation}\label{ah:small-count}
 \mathbb E_x e^{\theta_sK_t^s-C_0\gamma A_t}
 \le2^\gamma\Phi_{h,\gamma}(x).
\end{equation}

At a large death at least one survivor has distance greater than $b_h/(2\sqrt3)$.  Its squared lifetime potential is at least $\kappa^2$ by Lemma~\ref{ah:donor}, and it is selected with probability at least $1/2$.  Choose $0<\lambda_\ell\le1/8$ and set
$\theta_\ell=-\log\{(1+e^{-\lambda_\ell\kappa^2})/2\}>0$.  Conditional on the prestate,
\[
 \mathbb E[e^{\theta_\ell-\lambda_\ell r_h(X_{d,s})^2}]
 \le1
\]
at a large death.  Small deaths only decrease the exponential involving $J$.  Thus
\begin{equation}\label{ah:large-count}
 \mathbb E_xe^{\theta_\ell K_t^\ell-\lambda_\ell J_t}\le1.
\end{equation}

Fix $\gamma$ so that $C_0\gamma\le1/8$.  By Cauchy--Schwarz,~\eqref{ah:small-count} and Lemma~\ref{ah:energy},
\[
 \mathbb E_xe^{\theta_sK_t^s/2}
 \le(2^\gamma\Phi_{h,\gamma}(x))^{1/2}
       (\mathbb E_xe^{C_0\gamma A_t})^{1/2}.
\]
The same inequality applied to~\eqref{ah:large-count} gives
\[
 \mathbb E_xe^{\theta_\ell K_t^\ell/2}
 \le(\mathbb E_xe^{\lambda_\ell J_t})^{1/2}
 \le(\mathbb E_xe^{A_t/8})^{1/2}.
\]
For $\alpha=\min(\theta_s,\theta_\ell)/4$, H\"older gives
\[
 \begin{aligned}
 \mathbb E_xe^{\alpha(K_t^s+K_t^\ell)}
 &\le\left(\mathbb E_xe^{\theta_sK_t^s/2}
                \mathbb E_xe^{\theta_\ell K_t^\ell/2}\right)^{1/2}\\
 &\le(2^\gamma\Phi_{h,\gamma}(x))^{1/4}
                (\mathbb E_xe^{A_t/8})^{1/2}\\
 &\le C_1\Phi_{h,\gamma}(x)^{1/4}e^{C_2T},
 \qquad t=T\wedge\tau_m.
 \end{aligned}
\]
This proves~\eqref{d3:interior-count} for stopped prefixes.

The choices of $\kappa,C_0,\gamma,\lambda_\ell,\theta_s,\theta_\ell$
are made once for the fixed potential and are independent of all
localization levels and of small $h$. Thus, in particular,
\[
 \mathbb P_x(\tau_m\le T)
 \le C_1\Phi_{h,\gamma}(x)^{1/4}e^{C_2T-\alpha m}\longrightarrow0.
\]
With finitely many copies on a finite interval, every actual position lies on one of finitely many Brownian pieces and remains finite.  Thus spatial escape cannot end the recursive construction.  The process is nonexplosive.  Monotone convergence removes the count stopping in the estimate.
\end{proof}

\begin{lemma}[Microscopic rescue]\label{ah:rescue}
Set $t_h=8b_h^2$ and define the observed event
\[
 \mathsf E_h^{\rm res}
 =\{\min_i\dist(X_i(t_h),\partial D_h)\ge b_h/20\}
 \in\mathcal F_{t_h}.
\]
There are $p_0,\delta,\beta_0>0$, uniform in small $h$ and in the
initial configuration, such that the sublaw on this event followed by
a quiet interval of length $\delta$ has a density $f_{h,x}$ with
\begin{equation}\label{ah:rescue-density}
 \begin{gathered}
 \inf_x\mathbb P_x^h(\mathsf E_h^{\rm res})\ge p_0,\qquad
 \inf_x\int_{M^3}f_{h,x}\ge\beta_0,\\
 \sup_x\|f_{h,x}\|_{H^1(M^3)}\le C_\delta.
 \end{gathered}
\end{equation}
\end{lemma}
\begin{proof}
Put $Q_h(t)=R_h(X_t)^2$ and
\[
 \sigma=\inf\{t\ge0:Q_h(t)\ge4b_h^2\}.
\]
Thus $\sigma=0$ if the initial state is already outside the thin corner.
Before $\sigma$ all three distances are actual spherical distances,
and~\eqref{ah:normal-generator} gives
\[
 \mathcal LQ_h=3+2\sum_i\frac{d_h(X_i)}{a_{z_i}+d_h(X_i)}\ge3.
\]
Every copy adds a nonnegative donor square. From $Q_h(0)<4b_h^2$,
a continuous exit has $Q_h(\sigma)=4b_h^2$, and a copying exit has
$Q_h(\sigma)\le2Q_h(\sigma-)\le8b_h^2$. Stopped Dynkin gives
\[
 \begin{aligned}
 3\mathbb E(t\wedge\sigma)
 &\le\mathbb E[Q_h(t\wedge\sigma)-Q_h(0)]\le8b_h^2,\\
 \mathbb E\sigma&\le8b_h^2/3,
 \qquad \mathbb P(\sigma\le16b_h^2/3)\ge1/2.
 \end{aligned}
\]
First remove the finite-prefix stop using Lemma~\ref{ah:count}, then
let $t\to\infty$ by monotone convergence. The same bounds hold when
$\sigma=0$.

At $\sigma$, choose a protected label $P$ with clipped clearance at
least $2b_h/\sqrt3$. Put $b=b_h$ and fix $0<\varepsilon<10^{-3}$.
Let $\mathcal I=\{i:d_h(X_i(\sigma))<b/10\}$; it contains at most two
labels. For $i\in\mathcal I$, write
$X_i(\sigma)=z_i+(a_{z_i}+d_i)n_i$, where $0<d_i<b/10$ and $|n_i|=1$.
Define the driver controls
\[
 \ell_1(s)=\frac b4\min(2s/b^2,1),\qquad
 g_i(s)=
 \begin{cases}
 -\ell_1(s)n_i,&i\in\mathcal I,\\
 0,&i\notin\mathcal I,
 \end{cases}
\]
and the whole-driver tube event
\[
 \mathsf T=
 \bigcap_{i=1}^3\left\{\sup_{0\le s\le8b^2}
   |B_i(\sigma+s)-B_i(\sigma)-g_i(s)|\le\varepsilon b\right\}.
\]
Let $\mathsf D$ require that the next donor of each label in
$\mathcal I$ is $P$.

On $\mathsf T$, the protected and initially thick labels move by at most
$\varepsilon b$ from their positions at $\sigma$ until any wall hit.
For a thin label with no wall hit by $\sigma+b^2/2$, its free endpoint
would satisfy
\[
 |X_i(\sigma+b^2/2)-z_i|
 \le a_{z_i}-(1/4-1/10-\varepsilon)b<a_{z_i}.
\]
Thus every thin label hits before that time.
The control meets no other sphere, since $b<a_-/100$ and the collars
are disjoint. On $\mathsf D$, every thin label then copies $P$, in
whatever order the hits occur.
After a thin label's first copy at $\tau$, its remaining control
variation is at most $b/4$, and its driver increment error between
any two window times is at most $2\varepsilon b$. Thus, until any
putative further death,
\[
 |X_i(v)-X_P(v)|\le b/4+4\varepsilon b .
\]
For a copied thin label and an initially thick label, respectively,
the clearance bounds are
\[
 \begin{aligned}
 \dist(X_i(v),\partial D_h)
 &\ge(2/\sqrt3-1/4-5\varepsilon)b>b/20,\\
 \dist(X_i(v),\partial D_h)
 &\ge(1/10-\varepsilon)b>b/20.
 \end{aligned}
\]
They exclude any further death and hold through $t_h$ on
$\{\sigma\le16b^2/3\}$.

Conditioned on $\mathcal F_\sigma$, space-time scaling by $(b,b^2)$
turns the controls into either zero or
$-\frac14\min(2s,1)n_i$, on the fixed interval $[0,8]$ with tube width
$\varepsilon$. Brownian scaling and rotation invariance give
\[
 \mathbb P(\mathsf T\mid\mathcal F_\sigma)\ge p_{\rm tube}>0,
 \qquad
 \mathbb P(\mathsf T\cap\mathsf D\mid\mathcal F_\sigma)
       \ge p_{\rm tube}/4.
\]
On the tube event, prescribe donor $P$ at each successive first death
of a thin label. Each fresh donor uniform chooses $P$ with probability
$1/2$. A successful copy keeps that label safe, so at most two
prescriptions are needed, giving the factor $1/4$.
On $\{\sigma\le16b^2/3\}$,
$\sigma+b^2/2<t_h=8b^2\le\sigma+8b^2$, so both forced hits and the
terminal time lie in the controlled window. Hence
\[
 \mathbb P_x\left\{\min_i
            \dist(X_i(t_h),\partial D_h)\ge b_h/20\right\}
 \ge\frac12\frac{p_{\rm tube}}4=:p_0.
\]

For a rescued position, the nearest-sphere contribution and the lattice
bound in Lemma~\ref{ah:lifetime} give
\[
 \begin{aligned}
 P_\delta^h\one(y)
 &\ge1-\frac{a_+}{a_++b_h/20}-CA(\delta+h^2)\\
 &\ge s_1-CA(\delta+h^2),
 \qquad s_1:=\frac{\vartheta k_-}{20k_++\vartheta k_-}>0.
 \end{aligned}
\]
Choose $\delta>0$ and then $h$ so that $CA(\delta+h^2)\le s_1/2$.
Since $\mathsf E_h^{\rm res}\in\mathcal F_{t_h}$, conditioning at $t_h$
gives the quiet-segment density
\[
 f_{h,x}(y_1,y_2,y_3)
 =\E_x\left[\one_{\mathsf E_h^{\rm res}}
              \prod_{i=1}^3p_\delta^h(X_i(t_h),y_i)\right].
\]
Consequently
\[
 \int_{M^3}f_{h,x}
 =\E_x\left[\one_{\mathsf E_h^{\rm res}}
              \prod_{i=1}^3P_\delta^h\one(X_i(t_h))\right]
 \ge p_0(s_1/2)^3=:\beta_0.
\]

The killed heat row satisfies, uniformly in $h,x$,
\begin{equation}\label{ah:heat-row}
 \|p_\delta^h(x,\cdot)\|_2+
 \|\nabla p_\delta^h(x,\cdot)\|_2\le C_\delta.
\end{equation}
The free torus heat kernel dominates the killed row, giving its $L^2$
bound. For the nonnegative Dirichlet operator $A_h$,
\[
 \begin{aligned}
 \|\nabla p_\delta^h(x,\cdot)\|_2
 &=\sqrt2\|A_h^{1/2}e^{-\delta A_h/2}
                         p_{\delta/2}^h(x,\cdot)\|_2\\
 &\le C\delta^{-1/2}\|p_{\delta/2}^h(x,\cdot)\|_2
 \le C_\delta.
 \end{aligned}
\]
All rows are extended by zero across the holes. For
$p_i=p_\delta^h(y_i,\cdot)$, the tensor product satisfies
\[
 \left\|\prod_{i=1}^3p_i\right\|_{H^1(M^3)}^2
 =\prod_i\|p_i\|_2^2
   +\sum_i\|\nabla p_i\|_2^2\prod_{j\ne i}\|p_j\|_2^2
 \le C_\delta.
\]
Minkowski's inequality in the density formula gives
$\|f_{h,x}\|_{H^1(M^3)}\le C_\delta$, proving~\eqref{ah:rescue-density}.
\end{proof}

\begin{lemma}[Phase invariant law and stationary estimates]\label{ah:corner}
For fixed sufficiently small $h$, the phase process has a unique invariant probability $\pi_h$.  It satisfies
\begin{align}
 \mathbb E_{\pi_h}\sum_{s\le T}r_h(X_{d,s})
 &=cT\,\pi_h\sum_i(1-r_h(X_i))\le3cT,
 \label{ah:weighted-deaths}\\
 T\pi_h\mathcal G_h+\mathbb E_{\pi_h}J_T
 &=2cT\,\pi_h\sum_i r_h(X_i)(1-r_h(X_i))\le3cT/2,
 \label{ah:stationary-identity}\\
 \sup_{h<h_0}\pi_h\Phi_{h,\gamma}&<\infty,
 \qquad
 \mathbb E_{\pi_h}e^{\alpha K_T}\le C_1e^{C_2T}.
 \label{ah:stationary-count}
\end{align}
In particular, $\sup_{h<h_0}\mathbb E_{\pi_h}(K_T+1)^p\le C_p(1+T)^p$ for all $p>0$.
\end{lemma}
\begin{proof}
For fixed $h$, Proposition~\ref{prop:template-foster} supplies $\pi_h$.

Apply stopped It\^o to $\sum_i r_h(X_i)$ at fixed $h$. The donor jump
is nonnegative, and the fixed-$h$ gradient bound gives a square
integrable continuous martingale. Thus
\[
 \begin{aligned}
 &\E_{\pi_h}\left[\sum_i r_h(X_i(T\wedge\tau_m))
                          -\sum_i r_h(X_i(0))\right]\\
 &\quad=c\E_{\pi_h}\int_0^{T\wedge\tau_m}
                          \sum_i(r_h(X_i(s))-1)\,ds
        +\E_{\pi_h}\sum_{s\le T\wedge\tau_m}r_h(X_{d,s}),\\
 &0\le\E_{\pi_h}\sum_{s\le T\wedge\tau_m}r_h(X_{d,s})\le3+3cT.
 \end{aligned}
\]
As $m\to\infty$, nonexplosion and monotone convergence pass the jump
sum to the limit; endpoints and drift are bounded. Stationarity cancels
the endpoints and gives~\eqref{ah:weighted-deaths}.

For $F=\sum_i r_h(X_i)^2$, integrate Lemma~\ref{ah:energy} against
$\pi_h$ and let $m\to\infty$. Monotone convergence gives
\[
 \begin{aligned}
 A_{T\wedge\tau_m}&\uparrow A_T,\\
 \E_{\pi_h}e^{A_T/8}&\le e^{3/4+3cT/8},\\
 \E_{\pi_h}\langle M\rangle_T
 &\le4\E_{\pi_h}A_T\le C(1+T).
 \end{aligned}
\]
The martingale now has mean zero. Apply stationarity to
\eqref{d3:energy-identity} to obtain~\eqref{ah:stationary-identity},
with constants independent of $h$.

For the corner moment, use the nonincreasing regularization, with
$0<\eta\leq b_h/4$,
\[
 \Phi_\eta(R)=b_h^\gamma\eta^{-\gamma}
   \exp\left\{-\gamma\int_0^R
       \frac{u\,\chi_0(u/b_h)}{u^2+\eta^2}\,\dd u\right\}.
\]
It equals $b_h^\gamma(R^2+\eta^2)^{-\gamma/2}$ in the core and is
constant for $R\geq2b_h$. It is smooth as a function of the normal
coordinates even at $R=0$. In the transition band its value is at most
one, and its first two radial derivatives are bounded by
$C\gamma/b_h$ and $C\gamma/b_h^2$, uniformly in $\eta/b_h\leq1/4$.
There the normal curvature is at most $C/a_-\leq C'/b_h$, so the
negative part of its continuous drift is at most
$C\gamma b_h^{-2}$. The function is bounded for fixed $h,\eta$, and
$-\Phi_\eta$ has nonnegative copy jumps by
\eqref{ah:normal-copy}. In the core, differentiation and positive
normal curvature give
\begin{equation}\label{ah:regularized-corner}
 \mathcal L(-\Phi_\eta)
 \ge\frac{\gamma(1-\gamma)}2
       \frac{b_h^\gamma}{(R_h^2+\eta^2)^{\gamma/2+1}}.
\end{equation}
Indeed the radial Laplacian contribution is
\[
 \frac\gamma2b_h^\gamma
 \frac{(1-\gamma)R_h^2+3\eta^2}
      {(R_h^2+\eta^2)^{\gamma/2+2}},
\]
and the curvature contribution is nonnegative.

At fixed $h,\eta$, bounded derivatives justify Dynkin for
$-\Phi_\eta$ at $T\wedge\tau_m$. Its nonnegative jump sum satisfies
\[
 \E_{\pi_h}\sum_{s\le T\wedge\tau_m}
       [\Phi_\eta(R_h(X_{s-}))-\Phi_\eta(R_h(X_s))]
 \le 2\|\Phi_\eta\|_\infty+T\|\cL\Phi_\eta\|_\infty.
\]
This bound is finite at fixed $h,\eta$ and independent of $m$.
Monotone convergence of the jumps and bounded convergence of the
endpoints and drift give, after $m\to\infty$,
\[
 0=T\pi_h\cL(-\Phi_\eta)
   +\E_{\pi_h}\sum_{s\le T}
        [\Phi_\eta(R_h(X_{s-}))-\Phi_\eta(R_h(X_s))],
 \qquad \pi_h\cL(-\Phi_\eta)\le0.
\]
Combine the core lower bound~\eqref{ah:regularized-corner} with the
transition bound $\cL(-\Phi_\eta)\ge-C\gamma b_h^{-2}$ to obtain
\[
 \pi_h\left[
 \frac{b_h^\gamma}{(R_h^2+\eta^2)^{\gamma/2+1}};
 R_h\le b_h\right]\le Cb_h^{-2}.
\]
The constant is independent of $h,\eta$. Now let $\eta\downarrow0$
at fixed $h$ in the nonnegative core integral:
\[
 \begin{aligned}
 \pi_h[b_h^\gamma R_h^{-\gamma-2};R_h\le b_h]
       &\le Cb_h^{-2},\\
 \pi_h[\Phi_{h,\gamma};R_h\le b_h]
       &\le b_h^2\pi_h[b_h^\gamma R_h^{-\gamma-2};R_h\le b_h]
        \le C,\\
 \pi_h\Phi_{h,\gamma}&\le C+1.
 \end{aligned}
\]
The last line uses $\Phi_{h,\gamma}\le1$ outside the core.

Integrating Lemma~\ref{ah:count} now gives
\[
 \E_{\pi_h}e^{\alpha K_T}
 \le C_1e^{C_2T}\pi_h\Phi_{h,\gamma}^{1/4}
 \le C_1e^{C_2T}(\pi_h\Phi_{h,\gamma})^{1/4}
 \le C_1'e^{C_2T}.
\]
For $u_T=\alpha^{-1}[\log(C_1'\vee1)+C_2T]$, the Chernoff bound and
tail integration yield
\[
 \begin{gathered}
 \mathbb P_{\pi_h}(K_T>u)\le\min(1,e^{-\alpha(u-u_T)}),\\
 \E_{\pi_h}(K_T+1)^p\le C_p(1+u_T)^p\le C_p'(1+T)^p.
 \end{gathered}
\]
\end{proof}

\begin{lemma}[Growing collars and short waits]\label{ah:short-wait}
Let $\Lambda_h=h^{-1/2}$ and let $E_h$ be the event that at least one particle lies in $a_z<|x-z|\le\Lambda_ha_z$ for some sphere.  Then
\begin{equation}\label{ah:collar-mass}
 \pi_h(E_h)\le Ch^4,
 \qquad
 \mathbb P_{\pi_h}(\text{next death}\le t)\le Ct+C\sqrt h,
 \quad0<t\le1.
\end{equation}
\end{lemma}
\begin{proof}
We have $\Lambda_ha_z=O(h^{5/2})<\rho_h/2$ and
$\Lambda_h^2a_z(\rho_h^{-1}+c\rho_h)=O(h)$.  Lemma~\ref{ah:gradient} gives $\mathcal G_h\ge m_0^2/(4\Lambda_h^4a_+^2)$ on $E_h$.  By~\eqref{ah:stationary-identity},
\[
 \pi_h(E_h)\le6c\Lambda_h^4a_+^2/m_0^2=O(h^4).
\]
Outside the collars, apply Lemma~\ref{ah:lifetime} to the three future
motions and use the union bound. Thus
\[
 \begin{aligned}
 \mathbb P_{\pi_h}(\text{next death}\le t)
 &\le\pi_h(E_h)+3/\Lambda_h+CA(t+h^2)\\
 &\le Ct+C\sqrt h,\qquad0<t\le1.
 \end{aligned}
\]
This proves~\eqref{ah:collar-mass}.
\end{proof}

\subsection{Uniform death counts and interior rescue}

\begin{proposition}[Uniform death counts and interior rescue]\label{ah:uniform-interface}
Fix $L,k_-,k_+,c,\vartheta$ and the smooth positive rate $k$.
There are $h_0>0$ and $\gamma\in(0,1/2]$ such that the stopped and
unstopped count estimate~\eqref{d3:interior-count} and the stationary
estimates~\eqref{ah:stationary-identity}, \eqref{ah:stationary-count}
and~\eqref{ah:collar-mass} hold uniformly for every admissible
$h<h_0$. In the same mesh range, the observed rescue event and its
quiet-segment density satisfy
\begin{equation}\label{ah:uniform-rescue-interface}
 \begin{gathered}
 \inf_x\mathbb P_x^h(\mathsf E_h^{\rm res})\ge p_0,\qquad
 \inf_x\int_{M^3}f_{h,x}\ge\beta_0,\\
 \sup_x\|f_{h,x}\|_{H^1(M^3)}\le C_\delta,
 \qquad p_0,\beta_0,\delta>0.
 \end{gathered}
\end{equation}
All constants are uniform over initial configurations, absolute
translations and initial windings.
\end{proposition}

\begin{proof}
The lifetime and gradient estimates first fix $m_0,\kappa,C_0$.
The count parameters and rescue delay are then chosen with
\[
 C_0\gamma\le\tfrac18,\qquad \lambda_\ell\le\tfrac18,
 \qquad CA(\delta+h_0^2)\le\tfrac12s_1,
 \qquad \beta_0=p_0(s_1/2)^3,
\]
where $s_1>0$ is the survival margin in the rescue proof.
The scaled tube gives $p_0$ before any further reduction of $h_0$;
the fixed positive delay gives $C_\delta$ by~\eqref{ah:heat-row}.
All choices are independent of the localization levels.

The stationary argument removes the prefix cutoff $m$ and then
the regularization $\eta$ at each fixed mesh, yielding constants
uniform in $h$. Take $h_0$ to be the minimum of the finitely many
thresholds in the lifetime, gradient, count, rescue and growing-collar
estimates. Periodicity and Brownian increments preserve the bounds
under changes of initial lift.
\end{proof}

\subsection{Common ancestry at arbitrary killing rates}
\begin{lemma}[Uniform coalescence]\label{ah:coalescence}
For the fixed positive potential $k$ there are $h_0,T_*,\beta>0$ such that for every $h<h_0$ and every initial configuration in $D_h^3$, the particles at time $T_*$ share one ancestor from time zero with probability at least $\beta$.  Consequently the failure probability by time $nT_*$ is at most $(1-\beta)^n$.  The same conclusions hold for the soft model.
\end{lemma}
\begin{proof}
Restart at the end of the observed rescue and quiet interval in
Lemma~\ref{ah:rescue}. Its actual root subdensities satisfy
$\int f_{h,x}\ge\beta_0$ and $\|f_{h,x}\|_2\le C_\delta$.
Write
$(\tau_j,i_j,d_j)$ for the subsequent death times, victims and donors,
and prescribe the event
\[
 \mathsf W=\left\{
 \begin{gathered}
 K_2=2,\quad\tau_1\in[1/4,1/2],\quad(i_1,d_1)=(2,1),\\
 \tau_2\in[3/4,1],\quad(i_2,d_2)=(3,1)
 \end{gathered}\right\}.
\]
In the soft process, the total death rate is at most
$3k_+$. Each selected victim has rate at least $k_-$ and selects
donor $1$ with probability $1/2$. Integration over the two
intervals gives
\[
 \inf_y\mathbb P_y^k(\mathsf W)
 \ge\int_{1/4}^{1/2}\!\int_{3/4}^1
                    (k_-/2)^2e^{-6k_+}\,dv\,du
 =:\beta_1>0.
\]
On this word the three terminal particles descend from label $1$ at
the start of the word.

The zero-extended normalized root densities satisfy
\[
 g_{h,x}=\frac{f_{h,x}}{\int f_{h,x}},\qquad
 \|g_{h,x}\|_2\le C_\delta/\beta_0,
 \qquad g_{h,x}=0\text{ on }M^3\setminus M_h^3.
\]
Lemma~\ref{lem:compressed-record-tv}, including the terminal requirement
of no further death before time two, gives
\[
 \sup_{\substack{g\ge0,\ \int g=1,\ g=0\text{ on }M^3\setminus M_h^3\\
                         \|g\|_2\le C_\delta/\beta_0}}
       |\mathbb P_g^h(\mathsf W)-\mathbb P_g^k(\mathsf W)|
       \longrightarrow0.
\]
Reduce $h_0$ so that this error is at most $\beta_1/2$. The actual
rescued density then gives
\[
 \inf_x\int_{M^3}f_{h,x}(y)\mathbb P_y^h(\mathsf W)\,dy
 \ge\frac{\beta_1}{2}\inf_x\int_{M^3}f_{h,x}
 \ge\beta_0\beta_1/2.
\]
Take a common
\[
 T_*>\sup_{h<h_0}(8b_h^2+\delta+2),\qquad
 \beta=\beta_0\beta_1/2>0.
\]
Common ancestry persists under every later copy, so the success bound
holds at $T_*$. A winding shift of any initial particle preserves its
phase motion and every event probability used above. Hence the bound
holds for every Euclidean initial configuration.

Let $E_n$ be the event that the particles at $nT_*$ have a common
ancestor from time zero, with $E_0=\varnothing$. Success in a block
implies $E_{n+1}$, and $E_n$ also persists, so
\[
 \mathbb P_x(E_{n+1}^c\mid\mathcal F_{nT_*})
 \le(1-\beta)\one_{E_n^c},\qquad
 \mathbb P_x(E_n^c)\le(1-\beta)^n.
\]
For the soft system the prescribed word gives probability $\beta_1$
from every state, which is at least $\beta$.
\end{proof}

Lemmas~\ref{ah:corner} and~\ref{ah:coalescence} verify both estimates
in~\eqref{eq:transfer-inputs}. The marked transfer principle therefore
gives Corollary~\ref{cor:three-particle-transfer}, including convergence
of the full vector current and uniform moments of every fixed order.

\end{document}